\documentclass[final,times,1p,leqno,11pt]{elsarticle}

\usepackage[latin1]{inputenc}
\usepackage{amscd}
\usepackage{amsfonts}
\usepackage{amsmath}
\usepackage{amssymb}
\usepackage{amsthm}
\usepackage{enumerate}
\usepackage{latexsym}
\usepackage{mathrsfs}
\usepackage{multicol}
\usepackage{tikz} \usetikzlibrary{positioning}
\usepackage[all]{xy}

\makeatletter
\def\ps@pprintTitle{%
     \let\@oddhead\@empty
     \let\@evenhead\@empty
     \def\@oddfoot{\footnotesize\itshape%
       \hfill}
     \let\@evenfoot\@oddfoot}
\makeatother
\journal{\relax}

\theoremstyle{definition}
\newtheorem{thm}{Theorem}[subsection]

\newtheorem{cor}[thm]{Corollary}

\newtheorem{lem}[thm]{Lemma}
\newtheorem{prop}[thm]{Proposition}
\newtheorem{rem}[thm]{Remark}

\numberwithin{equation}{subsection}

\newcounter{cnt}

\newcommand{\wt}{\mathrm{wt}}

\newcommand{\ev}{\mathrm{ev}}

\newcommand{\sh}{\mathrm{sh}}

\def\cal#1{\text{$\mathcal{#1}$}}
\def\lie#1{\text{$\mathfrak{#1}$}}
\def\tlie#1{\tilde{\mathfrak{#1}}}
\def\hlie#1{\hat{\mathfrak{#1}}}
\def\tprod_#1^#2{\text{\scriptsize$\prod\limits_{\text{\footnotesize$#1$}}^{\text{\footnotesize$#2$}}$}}
\def\opl_#1^#2{\text{\scriptsize$\bigoplus\limits_{\text{\footnotesize$#1$}}^{\text{\footnotesize$#2$}}$}}
\def\otm_#1^#2{\text{\scriptsize$\bigotimes\limits_{\text{\footnotesize$#1$}}^{\text{\footnotesize$#2$}}$}}
\def\tsum_#1^#2{\text{\scriptsize$\sum\limits_{\text{\footnotesize$#1$}}^{\text{\footnotesize$#2$}}$}}
\def\tbinom#1#2{\text{$\left(\begin{smallmatrix} #1\\#2\end{smallmatrix}\right)$}}
\def\ch{{\rm ch}}

\newcommand{\bb}[1]{\text{$\mathbb{#1}$}}

\renewcommand{\epsilon}{\varepsilon}

\newcommand{\gb}[1]{{\pmb{#1}}}
\DeclareMathOperator{\gr}{gr}

\def\hp{\hat{P}}

\DeclareMathOperator{\sobre}{\twoheadrightarrow}

\def\v{{\scalebox{.75}{$\vartheta$}}}
\def\wa{\widehat{\cal W}}

\begin{document}

\begin{frontmatter}
\title{On Demazure and local Weyl modules for affine hyperalgebras\tnoteref{fn1}}

\author[ict]{Angelo Bianchi}
\emailauthor{acbianchi@ime.unicamp.br}{A.~B.}

\author[ict]{Tiago Macedo}
\emailauthor{tmacedo@unifesp.br}{T.~M.}

\author[ime]{Adriano Moura}
\emailauthor{aamoura@ime.unicamp.br}{A.~M.}

\address[ime]{UNICAMP - IMECC, Campinas - SP - Brazil, 13083-859}
\address[ict]{UNIFESP - ICT, São José dos Campos - SP - Brazil, 12247-014}
\tnotetext[fn1]{Partially supported by FAPESP grants 2011/22322-4 (A.~B.) and 2009/05887-8 (T. M.), and CNPq grant 303667/2011-7 (A.~M.)}

\begin{abstract}
We establish the existence of Demazure flags for graded local Weyl modules for hyper current algebras in positive characteristic. If the underlying simple Lie algebra is simply laced, the flag has length one, i.e., the graded local Weyl modules are isomorphic to Demazure modules. This extends to the positive characteristic setting results of Chari-Loktev, Fourier-Littelmann, and Naoi for current algebras in characteristic zero. Using this result, we prove that the character of local Weyl modules for hyper loop algebras depend only on the highest weight, but not on the (algebraically closed) ground field, and deduce a tensor product factorization for them.
\end{abstract}
\end{frontmatter}


\section*{Introduction}

Let $\lie g$ be a semisimple finite-dimensional Lie algebra over the complex numbers and, given an algebraically closed field $\bb F$, let $G_\bb F$ be a connected, simply connected, semisimple algebraic group over $\bb F $of the same Lie type as $\lie g$. The category of finite-dimensional
$G_\bb F$-modules is equivalent to that of the hyperalgebra $U_\bb F(\lie g)$. The hyperalgebra is a Hopf
algebra obtained from the universal enveloping algebra of $\lie g$ by first choosing a certain integral form and then changing scalars to $\bb F$ (this process is often referred to as a reduction modulo $p$). If the characteristic of $\bb F$ is positive, the category of finite-dimensional $G_\bb F$-modules is not semisimple, and the modules obtained by reduction modulo $p$ of simple $\lie g$-modules - called Weyl modules - provide examples of indecomposable, reducible modules. The Weyl modules have several interesting properties which are independent of $\bb F$ such as: a description in terms of generators and relations, being the universal highest-weight modules of the category of finite-dimensional $G_\bb F$-modules, their characters are given by the Weyl character formula.

Consider now the loop algebra $\tlie g=\lie g\otimes\bb C[t,t^{-1}]$. The finite-dimensional representation theory of $\tlie g$ was initiated by Chari and Presley in \cite{cpnew}, where the simple modules were classified in terms of tensor products of evaluation modules. Differently from the category of finite-dimensional $\lie g$-modules, the category of finite-dimensional $\tlie g$-modules is not semisimple. Therefore, it is natural to ask if there is a notion analogue to that of Weyl modules for $\tlie g$. In \cite{CPweyl}, Chari and Presley proved that the simple finite-dimensional  $\tlie g$-modules are highest-weight in an appropriate sense and introduced the Weyl modules for $\tlie g$ in terms of generator and relations which are the natural analogues of the relations for the original Weyl modules. The highest-weight vector is now an eigenvector for the action of the loop algebra $\tlie h$ over the Cartan subalgebra $\lie h$ of $\lie g$. Because of this, it eventually became common practice to use the terms $\ell$-weight and highest-$\ell$-weight. In particular, it was shown in \cite{CPweyl} that the just introduced Weyl modules share a second property with their older relatives: they are the universal finite-dimensional highest-$\ell$-weight modules. These results were immediately quantized and, still in \cite{CPweyl}, the notion of Weyl modules for the quantum loop algebra $U_q(\tlie g)$ was introduced. Chari and Presley conjectured (and proved for $\lie g=\lie{sl}_2$) that the Weyl modules for $\tlie g$ were classical limits of quantum Weyl modules. Moreover, all Weyl modules for $\tlie g$ could be obtained as classical limits of quantum Weyl modules which are actually irreducible. This can be viewed as the analogue of the property that the original Weyl modules are obtained by reduction modulo $p$ from simple $\lie g$-modules.

Motivated by bringing the discussion of the last paragraph to the positive characteristic setting, it was initiated in \cite{JMhyper} the study of the finite-dimensional representation theory of the hyperalgebras associated to $\tlie g$, which we refer to as hyper loop algebras. Several basic properties of the underlying abelian category were established and, in particular, the notion of Weyl modules was introduced. Moreover, it was shown that certain Weyl modules for $\tlie g$ can be reduced modulo $p$. In analogy with the previous paragraphs, it is natural to conjecture that the reduction modulo $p$ of a Weyl module is again a Weyl module (the difference is that now we cannot restrict attention to Weyl modules which are irreducible since there are too few of these).

In the mean time, two partial proofs of Chari and Presley's conjecture appeared \cite{chlo:wfd,foli:weyldem}. Namely, it follows from a tensor product factorization of the Weyl modules for $\tlie g$ proved in \cite{CPweyl} together with the fact that the irreducible quantum Weyl modules are tensor products of fundamental modules, that it suffices to compute the dimension of graded analogues of Weyl modules for the current algebra $\lie g[t]=\lie g\otimes\bb C[t]$.
These graded analogues of Weyl modules were introduced in \cite{FL04} as a particular case of a class of modules (named local Weyl modules) for algebras of the form $\lie g\otimes A$, where $A$ is a commutative associative algebra (see also \cite{cfk,fkks} and references therein for more on the recent development of the representation theory of such algebras). For $\lie g$ of type $A$, the dimensions of the graded Weyl modules were computed in \cite{chlo:wfd} by explicitly exhibiting a vector space basis. As a consequence, it was observed that they are isomorphic to certain Demazure modules.
For a general simply laced Lie algebra, this isomorphism was proved in \cite{foli:weyldem} by using a certain presentation of Demazure modules by generators and relations as well as by studying fusion products. In particular, the dimension of the graded Weyl modules could be computed resulting in a proof of the conjecture. It was also shown in \cite{foli:weyldem} that such isomorphisms do not exist in general in the non simply laced case.
It was pointed out by Nakajima that the general case could be deduced by using global bases theory (this proof remains unpublished, but a brief sketch is given in the introduction of \cite{foli:weyldem}). The relation with Demazure modules in the non simply laced case was finally established in completely generality in \cite{naoi:weyldem} where it was shown that the graded Weyl modules for $\lie g[t]$ admit Demazure flags, i.e., filtrations whose quotients are Demazure modules. Such flags are actually obtained from results of Joseph \cite{joseph03,joseph06} (see also \cite{lit:contract}) on global bases for tensor products of Demazure modules. Therefore, in the non simply laced case, the relation between Weyl and Demazure modules is, so far, dependent on the theory of global bases, although in a different manner than Nakajima's proposed proof.

The goal of the present paper is to extend to the positive characteristic context the results of \cite{foli:weyldem,naoi:weyldem} and prove the conjecture of \cite{JMhyper} on the reduction modulo $p$ of Weyl modules for hyper loop algebras. Moreover, we also prove a tensor product factorization of Weyl modules  which is the hyperalgebraic analogue of that proved in \cite{CPweyl}. However, due to the extra technical difficulties which arise when dealing with hyperalgebras in positive characteristic, there are several differences in our proofs from those used in the characteristic zero setting.  For instance, the tensor product factorization was originally used to restrict the study to computing the dimension of the graded Weyl modules for current algebras. In the positive characteristic setting, we actually deduce the tensor product factorization from the computation of the dimension. Also, for proving the existence of the Demazure flags, some arguments  used in \cite{naoi:weyldem} do not admit a hyperalgebraic analogue. Our approach to overcome these issues actually makes use of the characteristic zero version of the same statements. We also use the fact proved in \cite{mathieu88,mathieu89} that the characters of Demazure modules do not depend on the ground field. Different presentations of Demazure modules in terms of generator and relations are needed for different parts of the argument. For $\lie g $ of type $G_2$, technical issues for proving one of these presentations require that we restrict ourselves to characteristic different than 2 and 3. Outside type $G_2$, there is no restriction in the characteristic of the ground field.

While this paper was being finished, Chari and Venkatesh released the preprint \cite{chaven} where several new ideas for studying Demazure, local Weyl modules, and Kirillov-Reshetikhin modules are introduced. In particular, several results of \cite{CPweyl,foli:weyldem,naoi:weyldem} are recovered and generalized. Moreover, new (and simpler) presentations in terms of generators and relations for Demazure modules are obtained. It will be interesting to study if the ideas and results of \cite{chaven} can be brought to the positive characteristic setting as well.

The paper is organized as follows. We start Section \ref{s:main} fixing the notation regarding finite and affine types Kac-Moody algebras and reviewing the construction of the hyperalgebras. Next, using generators and relations, we define the Weyl modules for hyper loop algebras, their graded analogues for hyper current algebras, and the subclass of the class of Demazure modules which is relevant for us. We then state our main result (Theorem \ref{t:isos}) and recall the precise statement \eqref{e:conj} of the conjecture in \cite{JMhyper}.
Theorem \ref{t:isos} is stated in 4 parts. Part \eqref{t:isosdg} states the isomorphism between graded Weyl modules and Demazure modules for simply laced $\lie g$. Part \eqref{t:isofil} states the existence of Demazure flags for graded Weyl modules. Part \eqref{t:isoslg} establishes an isomorphism between a given graded Weyl module and a twist of certain Weyl module for the hyper loop algebra. Finally, part \eqref{t:isostp}  is the aforementioned tensor product factorization.
In Section \ref{ss:morenot}, we fix some further notation and establish a few technical results needed in the proofs.

Section \ref{sec:rev.g} brings a review of the finite-dimensional representation theory of the finite type hyperalgebras while Section \ref{sec:fdweyl} gives a very brief account of the relevant results from \cite{JMhyper}. Section \ref{ss:grmhca} is concerned with the category of finite-dimensional graded modules for the hyper current algebras. The main results of this subsection are Theorem \ref{t:fdrhca}, where the basic properties of the category are established, and Corollary \ref{c:gweylfd} which states that the graded Weyl modules for $\lie g[t]$ admit integral forms. Assuming Theorem \ref{t:isos}\eqref{t:isofil}, we prove \eqref{e:conj} in Section \ref{ss:JM.conjecture}. The proof actually makes use of the characteristic zero version of all parts of Theorem \ref{t:isos} as well as \cite[Corollary A]{naoi:weyldem} (stated here as Proposition \ref{p:dimwbyfund}). In Section \ref{ss:Mrel}, we prove a second presentation of Demazure modules in terms of generator and relations. It basically replaces a highest-weight generator by a lowest-weight one. This is the presentation which allows us to use the results of \cite{mathieu88,mathieu89} on the independence of  the characters of Demazure modules on the ground field.

In the first 3 subsections of Section \ref{s:joseph} we collect the results of \cite{joseph03,joseph06} on crystal and global bases which we need to prove Theorem \ref{t:filtration} which is an integral analogue of \cite[Corollary 4.16]{naoi:weyldem} on the existence of higher level Demazure flags for Demazure modules when the underlying simple Lie algebra $\lie g$ is simply laced. We remark that the proof of Theorem \ref{t:filtration} is the only one where the theory of global bases is used. We further remark that, in order to prove Theorem \ref{t:isos}\eqref{t:isofil}, we only need the statement of  Theorem \ref{t:filtration} for $\lie g$ of type $A$. It may also be interesting to observe that the only other place quantum groups are being used here is in the proof of the characteristic zero version of Theorem \ref{t:isos}\eqref{t:isoslg} (see Lemmas 1 and 3 and equation (15) in \cite{foli:weyldem}).

Theorem \ref{t:isos} is the proved in Section \ref{s:proof}. In particular, in Section \ref{ss:relnsl}, we prove a positive characteristic analogue of \cite[Proposition 4.1]{naoi:weyldem} which is a third presentation of Demazure modules in terms of generator and relations in the case that $\lie g$ is not simply laced.  This is where the restriction on the characteristic of the ground field for type $G_2$ appears.  Parts \eqref{t:isofil} and \eqref{t:isoslg} of Theorem \ref{t:isos} are proved in Sections \ref{ss:isofil} and \ref{pf13b}, respectively. Finally, in Sections \ref{ss:tpprime} and \ref{ss:ptp}, we prove that the tensor product of finite-dimensional highest-$\ell$-weight modules for hyper loop algebras with relatively prime highest $\ell$-weights is itself a highest-$\ell$-weight module and deduce Theorem \ref{t:isos}\eqref{t:isostp}. As an application of Theorem \ref{t:isos}, we end the paper proving that the graded Weyl modules are fusion products of Weyl modules with ``smaller'' highest weights (Proposition \ref{p:fusion}).

\section{The main results}\label{s:main}

\subsection{Finite type data}\label{ss:ftd}

Let $\lie g$ be a finite-dimensional simple Lie algebra over \bb C with a fixed Cartan subalgebra $\lie h \subset \lie g$. The associated root system will be denoted by $R \subset \lie h^\ast$. We fix a simple system $\Delta=\{\alpha_i : i\in I\} \subset R$ and denote the corresponding set of positive roots by $R^+$. The Borel subalgebra associated to $R^+$ will be denoted by $\lie b^{+} \subset \lie g$ and the opposite Borel subalgebra will be denoted by $\lie b^{-} \subset \lie g$. We fix a Chevalley basis of the Lie algebra $\lie g$ consisting of $x_{\alpha}^\pm \in \lie g_{\pm \alpha}$, for each $\alpha \in R^+$, and $h_i \in \lie h$, for each $i \in I$. We also define $h_\alpha\in\lie h, \alpha\in R^+$, by $h_\alpha = [x_\alpha^+,x_\alpha^-]$ (in particular, $h_i=h_{\alpha_i}, i\in I$) and set $R^\vee = \left\{ h_\alpha \in \lie h : \alpha \in R \right\}$. We often simplify notation and write $x_i^\pm$ in place of $x_{\alpha_i}^\pm, i\in I$. Let $(\ ,\ )$ denote the invariant symmetric bilinear form on $\lie g$ such that $(h_\theta,h_\theta)=2$, where $\theta$ is the highest root of $\lie g$.
Let $\nu:\lie h\to\lie h^*$ be the linear isomorphism induced by $(\ ,\ )$ and keep denoting by $(\ ,\ )$ the nondegenerate bilinear form induced by $\nu$ on $\lie h^*$. Notice that
\begin{equation}\label{e:norm(,)}
(x_\alpha^+,x_\alpha^-)=\frac{2}{(\alpha,\alpha)}  \qquad\text{for all}\qquad \alpha\in R^+
\end{equation}
and
\begin{equation}\label{e:norm(,)'}
(\alpha,\alpha)= \begin{cases}
2, &\text{if $\alpha$ is long},\\
2/r^\vee, &\text{if $\alpha$ is short,}
\end{cases}
\end{equation}
where $r^\vee\in\{1,2,3\}$ is the lacing number of $\lie g$. For notational convenience, set
\begin{equation}
r^\vee_\alpha = \frac{2}{(\alpha,\alpha)} = \begin{cases}
1, &\text{if $\alpha$ is long},\\
r^\vee, &\text{if $\alpha$ is short.}
\end{cases}
\end{equation}

We shall need the following fact \cite[Section 4.2]{car:sLieG}. Given $\alpha\in R$, let $x_\alpha=x^\pm_{\pm \alpha}$ according to whether $\alpha\in\pm R^+$. For $\alpha,\beta\in R$ let $p=\max\{n:\beta-n\alpha\in R\}$. Then, there exists $\epsilon\in\{-1,1\}$ such that
\begin{equation}\label{e:cbstructure}
[x_\alpha,x_\beta] = \epsilon (p+1) x_{\alpha+\beta}.
\end{equation}

The weight lattice is defined as $P = \left\{ \lambda \in \lie h^\ast : \lambda (h_\alpha) \in \bb Z, \forall \ \alpha \in R \right\}$, the subset of dominant weights is $P^+ = \left\{ \lambda \in P : \lambda (h_\alpha) \in \bb N, \forall \ \alpha \in R^+ \right\}$, the coweight lattice is defined as $P^\vee = \left\{ h \in \lie h : \alpha(h) \in \bb Z, \forall \ \alpha \in R \right\}$, and the subset of dominant coweights is $P^{\vee +} = \left\{ h \in P^\vee : \alpha(h) \in \bb N, \forall \ \alpha \in R^+ \right\}$. The fundamental weights will be denoted by $\omega_i, i\in I$.
The root lattice of $\lie g$ will be denoted by $Q$ and we let $Q^+ = \mathbb Z_{\ge 0}R^+$. We consider the usual partial order on $\lie h^*$: $\mu\le\lambda$ if and only if $\lambda-\mu\in Q^+$. The Weyl group of $\lie g$, denoted $\mathcal W$, is the subgroup of ${\rm Aut}_\mathbb C(\lie h^*)$ generated by the simple reflections $s_i, i\in I$, defined by $s_i(\mu) = \mu - \mu(h_i)\alpha_i$ for all $\mu\in\lie h^*$. As usual, $w_0$ will denote the longest element in \cal W.

\subsection{Affine type data}\label{ss:atd}

Consider the loop algebra $\tlie g = \lie g \otimes \bb C [t, t^{-1}]$, with Lie bracket given by $[x \otimes t^r , y \otimes t^s] = [x, y] \otimes t^{r+s}$, for any $x, y \in \lie g$, $r,s \in \bb Z$. We identify $\lie g$ with the subalgebra $\lie g\otimes 1$ of $\tlie g$. The subalgebra $\lie g[t]=  \lie g \otimes \bb C [t]$ is the current algebra of $\lie g$. If $\lie a$ is a subalgebra of $\lie g$, let $\tlie a=\lie a\otimes \mathbb C[t,t^{-1}]$ and $\lie a[t] = \lie a \otimes \bb C [t]$. Let also $\lie a [t]_{\pm}:= \lie a \otimes (t^{\pm 1} \bb C [t^{\pm 1}])$. In particular, as vector spaces,
\begin{equation*}
\tlie g = \tlie n^-\oplus\tlie h\oplus\tlie n^+ \qquad\text{and}\qquad \lie g[t]=\lie n^-[t]\oplus\lie h[t]\oplus\lie n^+[t].
\end{equation*}

The affine Kac-Moody algebra $\hlie g$ is the 2-dimensional extension $\hlie g:=\tlie g \oplus  \bb C c \oplus \bb C d$  of $\tlie g$ with Lie bracket given by
\[[x \otimes t^r, y \otimes t^s] = [x, y] \otimes t^{r+s} + r\ \delta_{r,-s}\ (x, y)\ c, \qquad [c,\hlie g]=\{0\}, \qquad\text{and}\qquad [d,x\otimes t^r]=r\ x\otimes t^r \]
for any $x,y \in \lie g, r,s \in \bb Z$. Observe that, if $\hlie g'=[\hlie g,\hlie g]$ is the derived subalgebra of $\hlie g$, then $\hlie g' = \tlie g\oplus\mathbb Cc$, and we have a nonsplit short exact sequence of Lie algebras $0\to \mathbb Cc\to\hlie g'\to \tlie g\to0$.

Set $\hlie h' = \lie h\oplus\mathbb Cc.$
Notice that $\lie g, \lie g[t]$, and $\lie g[t]_\pm$ remain subalgebras of $\hlie g$. Set
\[\hlie h = \lie h \oplus \bb C c \oplus \bb C d, \qquad \hlie n^{\pm} = \lie n^{\pm} \oplus \lie g [t]_{\pm}, \qquad\text{and}\qquad \hlie b^{\pm} = \hlie n^{\pm} \oplus \hlie h.\]
The root system, positive root system, and set of simple roots associated to the triangular decomposition $\hlie g=\hlie n^-\oplus\hlie h\oplus\hlie n^+$ will be denoted by $\hat{R}, \hat R^+$ and $\hat{\Delta}$ respectively. Let $\hat I=I\sqcup\{0\}$ and $h_0=c-h_\theta$, so that $\{h_i:i\in\hat I\}\cup\{d\}$ is a basis of $\hlie h$. Identify $\lie h^*$ with the subspace $\{\lambda\in\hlie h^*:\lambda(c)=\lambda(d)=0\}$. Let also $\delta\in\hlie h^*$ be such that $\delta(d)=1$ and $\delta(h_i)=0$ for all $i\in\hat I$ and define $\alpha_0=\delta-\theta$. Then,  $\hat\Delta=\{\alpha_i:i\in\hat I\}$, $\hat R^+ = R^+\cup\{\alpha+r\delta: \alpha\in R\cup\{0\}, r\in\mathbb Z_{>0}\}$,  $\hlie g_{\alpha + r \delta} = \lie g_\alpha \otimes t^r$, if $\alpha \in R, r \in \bb Z$, and $\hlie g_{r \delta} = \lie h \otimes t^r$, if $r \in \bb Z \setminus \{0\}$. Observe that
\begin{equation}\label{e:rootlevel}
\alpha(c) = 0 \qquad\text{for all}\qquad \alpha\in\hat R.
\end{equation}
A root $\gamma \in \hat R$ is called real if $\gamma = (\alpha + r \delta)$ with $\alpha \in R, r \in \bb Z$, and imaginary if $\gamma = r \delta$ with $r \in \bb Z \setminus \{0\}$. Set $x_{\alpha,r}^\pm=x_\alpha^\pm\otimes t^r, h_{\alpha,r}=h_\alpha\otimes t^r, \alpha\in R^+,r\in\mathbb Z$. We often simplify notation and write $x_{i,r}^\pm$ and $h_{i,r}$ in place of $x_{\alpha_i,r}^\pm$ and $h_{\alpha_i,r},  i\in I, r\in\bb Z$. Observe that $\{x_{\alpha,r}^\pm, h_{i,r}:\alpha\in R^+,i\in I, r\in\mathbb Z\}$ is a basis of $\tlie g$. Given $\alpha\in R^+$ and $r\in\bb Z_{>0}$, set
$$x_{\pm\alpha+r\delta}^+ = x_{\alpha,r}^\pm, \qquad x_{\pm\alpha+r\delta}^- = x_{\alpha,-r}^\mp, \qquad\text{and}\qquad h_{\pm\alpha+r\delta} =  [x_{\pm\alpha+r\delta}^+,x_{\pm\alpha+r\delta}^-] = \pm h_\alpha + rr^\vee_\alpha c.$$

Define also $\Lambda_i\in\hlie h^*, i\in\hat I$, by the requirement $\Lambda_i(d)=0, \Lambda_i(h_j)=\delta_{ij}$ for all $i,j\in\hat I$.  Set $\hat P = \mathbb Z\delta \oplus ( \oplus_{i\in\hat I}^{}\ \mathbb Z\Lambda_i ), \hat P^+ = \mathbb Z\delta \oplus ( \oplus_{i\in\hat I}^{}\ \mathbb N\Lambda_i), \hat P'=\oplus_{i\in\hat I}^{}\ \mathbb Z\Lambda_i$, and $\hat P'^+=\hat P'\cap\hat P^+$. Notice that
\[ \Lambda_0 (h) = 0 \quad \text{iff} \quad h \in \lie h \oplus \bb C d
\qquad \text{and} \qquad
\Lambda_i - \omega_i = \omega_i(h_\theta)\Lambda_0 \quad\text{for all}\ i \in I.\]
Hence, $\hat{P} = \bb Z \Lambda_0 \oplus P \oplus \bb Z \delta$. Given $\Lambda\in\hat P$, the number $\Lambda(c)$ is called the level of $\Lambda$. By \eqref{e:rootlevel}, the level of $\Lambda$ depends only on its class modulo the root lattice $\hat Q$. Set also $\hat Q^+=\mathbb Z_{\ge 0}\hat R^+$ and let
$\widehat{\mathcal W}$ denote the  affine Weyl group, which is generated by the simple reflections $s_i,i\in\hat I$.
Finally, observe that $\{\Lambda_0, \delta\}\cup\Delta$ is a basis of $\hlie h^\ast$.

\subsection{Integral forms and hyperalgebras}\label{ss:Zforms}

We use the following notation. Given a $\mathbb Q$-algebra  $U$ with unity, an element $x\in U$, and $k\in\mathbb N$, set
\begin{equation*}
x^{(k)} = \frac{1}{k!} x^k \qquad\text{and}\qquad \binom{x}{k} = \frac{1}{k!} x(x-1)\cdots(x-k+1).
\end{equation*}
In the case $U=U(\tlie g)$, we also introduce elements $\Lambda_{x,\pm r}\in U(\tlie g), x\in\lie g,r\in\mathbb N$, by the following identity of power series in the variable $u$:
\begin{equation*}
\Lambda_x^\pm(u):= \sum_{r\ge 0} \Lambda_{x,\pm r} u^r = \exp\left(- \sum_{s>0} \frac{x\otimes t^{\pm s}}{s} u^s\right).
\end{equation*}
Most of the time we will work with such elements with $x=h_\alpha$ for some $\alpha\in R^+$. We then simplify notation and write $\Lambda^\pm_\alpha(u) = \Lambda^\pm_{h_\alpha}(u)$ and, if $\alpha=\alpha_i$ for some $i\in I$, we simply write $\Lambda_i^\pm(u)=\Lambda_{h_i}^\pm(u)$. To shorten notation, we also set $\Lambda_x(u)=\Lambda_x^+(u)$.

Consider the \bb Z-subalgebra $U_\bb Z(\hlie g')$ of $U(\hlie g')$ generated by the set $\{( x_{\alpha, r}^{\pm})^{(k)}:\alpha \in R^+, r \in \bb Z, k \in \bb N \}$. By \cite[Theorem 5.8]{garala}, it is a free $\mathbb Z$-submodule of $U(\hlie g')$ and satisfies $\bb C \otimes_\bb Z U_\bb Z (\hlie g') = U(\hlie g')$. In other words, $U_\bb Z(\hlie g')$ is an integral form of $U(\hlie g')$. Moreover, the image of $U_\mathbb Z(\hlie g')$ in $U(\tlie g)$ is an integral form of $U(\tlie g)$ denoted by $U_\mathbb Z(\tlie g)$.
For a Lie subalgebra $\lie a$ of $\hlie g'$ set
$$U_\mathbb Z(\lie a) = U(\lie a)\cap U_\mathbb Z(\hlie g')$$
and similarly for subalgebras of $\tlie g$. The subalgebra $U_\mathbb Z(\lie g)$ coincides with the $\mathbb Z$-subalgebra of $U(\hlie g)$ generated by $\{( x_{\alpha}^{\pm})^{(k)}:\alpha \in R^+, k \in \bb N \}$. The subalgebra $U_\bb Z (\lie n^\pm)$ of $U_\bb Z (\lie g)$ is generated, as \bb Z-subalgebra, by the set $\{ ( x_{\alpha}^{\pm} )^{(k)}:\alpha \in R^+, k \in \bb N \}$ while $U_\mathbb Z(\lie h)$ is generated, as a $\mathbb Z$-subalgebra, by $\left\{ \binom{h_{i}}{k}:i \in I, k \in \bb N\right\}$. Similarly, the subalgebra $U_\bb Z (\lie n^\pm[t])$ of $U_\bb Z (\lie g[t])$ is generated, as \bb Z-subalgebra, by the set $\{ ( x_{\alpha,r}^{\pm } )^{(k)}:\alpha \in R^+, k \in \bb N, r \in \bb Z_{\ge 0} \}$ while $U_\bb Z (\lie h[t]_+)$ is generated by $\left\{\Lambda_{i, r}:i \in I, r \in \bb Z_{>0} \right\}$. In fact, the latter is free commutative over the given set.
The PBW Theorem implies that multiplication establishes isomorphisms of $\mathbb Z$-modules
\begin{gather*}
U_\mathbb Z(\hlie g')\cong U_\mathbb Z(\hlie n^-)\otimes U_\mathbb Z(\hlie h')\otimes U_\mathbb Z(\hlie n^+), \\
U_\mathbb Z(\tlie g)\cong U_\mathbb Z(\tlie n^-)\otimes U_\mathbb Z(\tlie h)\otimes U_\mathbb Z(\tlie n^+) \text{ and } \\
U_\mathbb Z(\lie g[t])\cong U_\mathbb Z(\lie n^-[t])\otimes U_\mathbb Z(\lie h[t])\otimes U_\mathbb Z(\lie n^+[t]) .
\end{gather*}
Moreover, restricted to $U_\mathbb Z(\tlie h)$ this gives rise to an isomorphism of $\mathbb Z$-algebras
\begin{equation*}
U_\mathbb Z(\tlie h)\cong U_\mathbb Z(\lie h[t]_-)\otimes U_\mathbb Z(\lie h)\otimes U_\mathbb Z(\lie h[t]_+).
\end{equation*}
In general, it may not be true that $U_\bb Z(\lie a)$ is an integral form of $U(\lie a)$. However, if $\lie a$ has a basis consisting of real root vectors, an elementary use of the PBW Theorem implies that this is true.
We shall make use of algebras of this form later on.

Given a field  $\mathbb F$, define the $\mathbb F$-hyperalgebra of $\lie a$ by $U_\mathbb F(\lie a) =  \mathbb F\otimes_{\mathbb Z}U_\mathbb Z(\lie a)$, where $\lie a$ is any of the Lie algebras with \bb Z-forms defined above. Clearly, if the characteristic of $\mathbb F$ is zero, the algebra $U_\mathbb F(\tlie g)$ is naturally isomorphic to $U(\tlie g_\mathbb F)$ where $\tlie g_\mathbb F=  \mathbb F\otimes_\mathbb Z\tlie g_\mathbb Z$ and $\tlie g_\mathbb Z$ is the $\mathbb Z$-span of the Chevalley basis of $\tlie g$, and similarly for all algebras $\lie a$ we have considered. For fields of positive characteristic we just have an algebra homomorphism $U(\lie a_\mathbb F)\to U_\mathbb F(\lie a)$ which is neither injective nor surjective. We will keep denoting by $x$ the image of an element $x\in U_\mathbb Z(\lie a)$ in $U_\mathbb F(\lie a)$. Notice that we have $U_\mathbb F(\tlie g)=U_\mathbb F(\tlie n^-)U_\mathbb F(\tlie h)U_\mathbb F(\tlie n^+)$.

Given an algebraically closed field $\mathbb F$, let $\mathbb A$ be a Henselian discrete valuation ring of characteristic zero having $\mathbb F$ as its residue field. Set $U_\mathbb A(\lie a)=\mathbb A\otimes_\mathbb Z U_\mathbb Z(\lie a)$. Clearly $U_\mathbb F(\lie a)\cong \mathbb F\otimes_\mathbb A U_\mathbb A(\lie a)$. We shall also fix an algebraic closure $\bb K$ of the field of fractions of \bb A. For an explanation why we shall need to move from integral forms to $\bb A$-forms, see Remark \ref{r:Aforms} (and \cite[Section 4C]{JMhyper}). As mentioned in the introduction, we assume the characteristic of $\mathbb F$ is either zero or at least 5 if $\lie g$ is of type $G_2$.

Notice that the Hopf algebra structure of the universal enveloping algebras induce such structure on the hyperalgebras.
For any Hopf algebra $H$, denote by $H^0$ its augmentation ideal.

\subsection{$\ell$-weight lattice} \label{l-weight.lattice}

For a ring $A$, we shall denote by $A^\times$ its set of unities. Consider the set $\cal P_\mathbb F^+$ consisting of $|I|$-tuples $\gb\omega = (\gb\omega_i)_{i\in I}$, where $\gb\omega_i \in \mathbb F[u]$ and $\gb \omega_i (0) = 1$ for all $i \in I$. Endowed with coordinatewise polynomial multiplication, $\cal P_\bb F^+$ is a monoid. We denote by $\cal P_\mathbb F$ the multiplicative abelian group associated to $\cal P_\mathbb F^+$ which will be referred to as the $\ell$-weight lattice associated to $\lie g$. One can describe $\cal P_\bb F$ in another way. Given $\mu\in P$ and $a\in\mathbb F^\times$, let $\gb\omega_{\mu,a}$ be the element of $\cal P_\mathbb F$ defined as
\begin{equation*}
(\gb\omega_{\mu,a})_i(u) = (1-au)^{\mu(h_i)} \quad\text{for all}\quad i\in I.
\end{equation*}
If $\mu=\omega_i$ is a fundamental weight, we simplify notation and write $\gb\omega_{\omega_i,a} = \gb\omega_{i,a}$. We refer to $\gb\omega_{i,a}$ as a fundamental $\ell$-weight, for all $i \in I$ and $a \in \bb F^\times$. Notice that $\cal P_\mathbb F$ is the free abelian group on the set of fundamental $\ell$-weights. One defines $\cal P_\bb K$ in the obvious way. Let also $\cal P_\bb A^\times$ be the submonoid of $\cal P_\bb K^+$ generated by $\gb\omega_{i,a}, i\in I, a\in\bb A^\times$.

Let $\wt:\cal P_\mathbb F \to P$ be the unique group homomorphism such that $\wt(\gb\omega_{i,a}) = \omega_i$ for all $i\in I,a\in\mathbb F^\times$. Let also $\gb\omega\mapsto \gb\omega^-$ be the unique group automorphism of $\cal P_\mathbb F$ mapping $\gb\omega_{i,a}$ to $\gb\omega_{i,a^{-1}}$ for all $i\in I,a\in\mathbb F^\times$. For notational convenience we set $\gb\omega^+=\gb\omega$.

The abelian group $\cal P_\mathbb F$ can be identified with a subgroup of the monoid of $|I|$-tuples of formal power series with coefficients in $\mathbb F$ by identifying the rational function $(1-au)^{-1}$ with the corresponding geometric formal power series $\sum_{n\geq0} (au)^n$. This allows us to define an inclusion $\cal P_\mathbb F \hookrightarrow U_\mathbb F(\tlie h)^*$. Indeed, if $\gb\omega\in\cal P_\bb F$ is such that $\gb \omega^\pm_i(u) = \sum_{r\geq 0} \omega_{i,\pm r} u^r \in \cal P_\bb F$, set
\begin{gather*}
\gb\omega\left(\tbinom{h_i}{k}\right) = \tbinom{\wt(\gb\omega)(h_i)}{k}, \quad \gb\omega(\Lambda_{i,r})=\omega_{i,r}, \quad\text{for all}\quad i\in I,r,k\in\mathbb Z,k\ge 0,\\ \text{and}\quad
\gb\omega(xy)=\gb\omega(x)\gb\omega(y), \quad\text{for all}\quad x,y\in U_\mathbb F(\tlie h).
\end{gather*}

\subsection{Demazure and local Weyl modules} \label{sec:demweyl}

Given $\gb\omega\in\cal P_\mathbb F^+$, the local Weyl module $W_\mathbb F(\gb\omega)$ is the quotient of $U_\mathbb F(\tlie g)$ by the left ideal generated by
\begin{gather}
U_\mathbb F(\tlie n^+)^0, \quad h-\gb\omega(h), \quad (x_\alpha^-)^{(k)} \quad\text{for all}\quad h\in U_\mathbb F(\tlie h),\ \alpha\in R^+,\ k>\wt(\gb\omega)(h_\alpha).
\end{gather}
It is known that the local Weyl modules are finite-dimensional (cf. Theorem \ref{t:fdrhla} \eqref{t:wmfdim}).

For $\lambda\in P^+$, the graded local Weyl module $W_\mathbb F^c(\lambda)$ is the quotient of $U_\mathbb F(\lie g[t])$ by the left ideal $I_\bb F^c(\lambda)$ generated by
\begin{gather}\label{e:glWrel}
U_\mathbb F(\lie n^+[t])^0, \quad U_\mathbb F(\lie h [t]_+)^0, \quad h- \lambda(h), \quad (x_\alpha^-)^{(k)}, \quad \text{for all}\quad h\in U_\mathbb F(\lie h),\ \alpha\in R^+,\ k>\lambda(h_\alpha).
\end{gather}
Also, given $\ell \ge 0$, let $D_\mathbb F(\ell,\lambda)$ denote the quotient of $U_\mathbb F(\lie g[t])$ by the left ideal $I_\bb F(\ell, \lambda)$ generated by $I_\bb F^c(\lambda)$ together with
\begin{gather} \label{e:FLrel}
(x_{\alpha,s}^-)^{(k)} \quad \text{for all}\quad  \alpha\in R^+, s,k\in\mathbb Z_{\ge 0}, \ k>\ \max \{0, \lambda(h_\alpha) - s\ell r^\vee_\alpha \}.
\end{gather}
In particular, $D_\bb F (\ell, \lambda)$ is a quotient of $W_{\bb F}^{c} (\lambda)$.

The algebra $U_\bb F (\lie g [t])$ inherits a $\mathbb Z$-grading from the grading on the polynomial algebra $\bb C [t]$. The ideals $I_{\bb F}^{c} (\lambda)$ and $I_\bb F (\ell, \lambda)$ are clearly graded and, hence, the modules $W_{\bb F}^{c} (\lambda)$ and $D_\bb F (\ell, \lambda)$ are graded. If $V$ is a graded module, let $V[r]$ be its $r$-th graded piece. Given $m\in\mathbb Z$, let $\tau_m (V)$ be the $U_\mathbb F(\lie g[t])$-module such that $\tau_s(V)[r] = V[r-m]$ for all $r\in\mathbb Z$.  Set
$$D_\mathbb F(\ell,\lambda,m) = \tau_m(D_\mathbb F(\ell,\lambda)).$$

\begin{rem}{\label{ref:UF}}
Local Weyl modules were simply called Weyl modules in \cite{CPweyl}. Certain infinite-dimension\-al modules, which were called maximal integrable modules in \cite{CPweyl}, are now called
global Weyl modules. The modern names, local and global Weyl modules were
coined by Feigin and Loktev in \cite{FL04}, where they introduced these modules in the context of
generalized current algebras. We will not consider the global Weyl modules in this paper. We refer the reader to \cite{cfk,fkks,fms} and references therein for recent developments in the theory of global and local Weyl modules for (equivariant) map algebras. See also \cite{cha:ibma} for the initial steps in the study of the hyperalgebras of (equivariant) map algebras.
\end{rem}

We are ready to state the main theorem of the paper.

\begin{thm} \label{t:isos}
Let $\lambda \in P^+$.
\begin{enumerate}[(a)]

\item\label{t:isosdg} If $\lie g$ is simply laced, then $D_\mathbb F(1,\lambda)$ and $W_\mathbb F^c(\lambda)$ are isomorphic $U_\bb F (\lie g [t])$-modules.

\item\label{t:isofil} There exist $k\ge 1, m_j\in\bb Z_{\ge 0}$, and $\lambda_j\in P^+, j=1,\dots,k$, (independent of $\mathbb F$) such that the $U_\bb F (\lie g[t])$-module $W_\bb F^c (\lambda)$ admits a filtration $(0) = W_0 \subset W_1 \subset \cdots \subset W_{k-1} \subset W_k = W_\bb F^c (\lambda)$, with $$W_j/W_{j-1}\cong D_\bb F (1, \lambda_j,m_j).$$ 

\item\label{t:isoslg} For any $a\in\mathbb F^\times$, there exists an automorphism $\varphi_a$ of $U_\mathbb F(\lie g[t])$ such that the pull-back of $W_\mathbb F(\gb\omega_{\lambda, a})$ by $\varphi_a$ is isomorphic to $W_\mathbb F^c(\lambda)$.

\item\label{t:isostp} If $\gb\omega= \prod_{j=1}^m \gb\omega_{\lambda_j,a_j}$ for some $m\ge 0, \lambda_j\in P^+,a_j\in\mathbb F^\times, j=1,\dots,m$, with $a_i\ne a_j$ for $i\ne j$, then
\begin{equation*}
W_\mathbb F(\gb\omega) \cong
\otm_{j=1}^m W_\mathbb F(\gb\omega_{\lambda_j,a_j}).
\end{equation*}
\end{enumerate}
\end{thm}

Assume the characteristic of $\mathbb F$ is zero. Then, part \eqref{t:isosdg} of this theorem was proved in \cite{CPweyl} for $\lie g=\lie{sl}_2$, in \cite{chlo:wfd} for type $A$, and in \cite{foli:weyldem} for types ADE.
Part  \eqref{t:isofil} was proved in \cite{naoi:weyldem}. Part \eqref{t:isoslg} for simply laced $\lie g$ was proved in \cite{foli:weyldem} using part \eqref{t:isosdg} (see Lemmas 1 and 3 and equation (15) of \cite{foli:weyldem}). The same proof works in the non simply laced case once part  \eqref{t:isofil} is established. The last part was proved in \cite{CPweyl}. We will make use of Theorem \ref{t:isos} in the characteristic zero setting for extending it to the positive characteristic context.
Both \cite{chlo:wfd} and \cite{foli:weyldem} use the $\lie{sl}_2$-case of part \eqref{t:isosdg} in the proofs. A characteristic-free proof of Theorem \ref{t:isos}\eqref{t:isosdg} for $\lie{sl}_2$ was given in \cite{jm:weyl}.

We will see in Section \ref{ss:Mrel} that the class of modules $D_\mathbb F(\ell,\lambda)$ form a subclass of the class of Demazure modules. In particular, it follows from \cite[Lemme 8]{mathieu88} that $\dim(D_\mathbb F(\ell,\lambda))$ depends only on $\ell$ and $\lambda$, but not on $\mathbb F$ (see also the Remark on page 56 of \cite{mathieu89} and references therein). Together with Theorem \ref{t:isos}\eqref{t:isofil}, this implies the following corollary.

\begin{cor}\label{c:indf}
For all $\lambda\in P^+$, we have $\dim W^c_\bb F(\lambda) = \dim W^c_\bb C(\lambda)$.\hfill\qedsymbol
\end{cor}

As an application of this corollary, we will prove a conjecture of \cite{JMhyper} which we now recall.
The following theorem was proved in \cite{JMhyper}.

\begin{thm}\label{t:weylform}
Suppose $\gb\omega\in\mathcal P_\mathbb A^\times$ and let $\lambda=\wt(\gb\omega)$, $v$ the image of $1$ in $W_\mathbb K(\gb\omega)$, and $L_\mathbb A(\gb\omega) = U_\mathbb A(\tlie g)v$. Then, $L_\mathbb A(\gb\omega)$ is a free $\mathbb A$-module such that $\mathbb K\otimes_\mathbb A L_\mathbb A(\gb\omega)\cong W_\mathbb K(\gb\omega)$.\hfill\qedsymbol
\end{thm}

Let $\gb\varpi$ be the image of $\gb\omega$ in $\cal P_\mathbb F$. It easily follows that $\mathbb F\otimes_\mathbb A L_\mathbb A(\gb\omega)$  is a quotient of $W_\mathbb F(\gb\varpi)$ and, hence,
\begin{equation}\label{e:conjknown}
\dim W_\mathbb K(\gb\omega) \le \dim W_\mathbb F(\gb\varpi).
\end{equation}
It was conjectured in \cite{JMhyper} that
\begin{equation}\label{e:conj}
\mathbb F\otimes_\mathbb A L_\mathbb A(\gb\omega) \cong W_\mathbb F(\gb\varpi).
\end{equation}
We will prove \eqref{e:conj} in Section \ref{ss:JM.conjecture}. In particular, it follows that
\begin{equation}\label{e:conjdim}
\dim W_\mathbb F(\gb\varpi) = \dim W_\mathbb C^c(\lambda).
\end{equation}

\begin{rem}\label{r:Aforms}
Theorem \ref{t:isos}\eqref{t:isostp} was also conjectured in \cite{JMhyper} and it is false if $\bb F$ were not algebraically closed (see \cite{jm:hlanac} in that case). Observe that, for all $\gb\varpi\in\mathcal P_\mathbb F^+$, there exists $\gb\omega\in\mathcal P_\mathbb A^\times$ such that $\gb\varpi$ is the image of $\gb\omega$ in $\cal P_\mathbb F$. This is the main reason for considering $\mathbb A$-forms instead of $\mathbb Z$-forms. The block decomposition of the categories of finite-dimensional representations of hyper loop algebras was established in \cite{JMhyper,jm:hlanac} assuming \eqref{e:conj} and Theorem \ref{t:isos}\eqref{t:isostp}. The proof of one part of \cite[Theorem 4.1]{bimo:htwisted} also relies on these two results. Therefore, by proving \eqref{e:conj} and Theorem \ref{t:isos}\eqref{t:isostp}, we confirm these results of \cite{bimo:htwisted,JMhyper,jm:hlanac}. A version of Theorem \ref{t:isos} for twisted affine Kac-Moody algebras was obtained in \cite{foku} in the characteristic zero setting. We will consider the characteristic-free twisted version of Theorem \ref{t:isos} elsewhere.
\end{rem}

\section{Further notation and technical lemmas}\label{ss:morenot}

\subsection{Some commutation relations}\label{ss:comut}

We begin recalling the following well-known relation in $U_\bb Z(\lie g)$
\begin{gather}\label{e:koslem}
(x_{\alpha}^+)^{(l)}(x_{\alpha}^-)^{(k)} = \sum_{m=0}^{\min\{k,l\}} (x_{\alpha}^-)^{(k-m)}\binom{h_\alpha-k-l+2m}{m}(x_{\alpha}^+)^{(l-m)} \quad\text{for all}\quad \alpha\in R^+, l,k\in\bb Z_{\ge 0}.
\end{gather}
Since for all $\alpha\in R^+, s\in\bb Z$, the span of $x_{\alpha,\pm s}^\pm, h_\alpha$ is a subalgebra isomorphic to $\lie{sl}_2$, we get the following relation in $U_\bb Z(\tlie g)$
\begin{gather}
(x_{\alpha,s}^+)^{(l)}(x_{\alpha,-s}^-)^{(k)} = \sum_{m=0}^{\min\{k,l\}} (x_{\alpha,-s}^-)^{(k-m)}\binom{h_\alpha-k-l+2m}{m}(x_{\alpha,s}^+)^{(l-m)}.
\end{gather}
Next, we consider the case when the grades of the elements in the left-hand-side is not symmetric.

Given $m>0$, consider the Lie algebra endomorphism $\tau_m$ of $\tlie g$ induced by the ring endomorphism of $\mathbb C[t,t^{-1}], t\mapsto t^m$. Notice that the restriction of $\tau_m$ to $\lie g[t]$ gives rise to an endomorphism of $\lie g[t]$. Moreover, denoting by $\tau_m$ its extension to an algebra endomorphism of $U(\tlie g)$, notice that $U_\mathbb Z(\lie a)$ is invariant under $\tau_m$ for $\lie a=\lie g, \lie n^\pm,\lie h, \tlie n^\pm, \tlie h, \lie n^\pm[t],\lie h[t],\lie h[t]_+$. In fact $\tau_m ((x_{\alpha,r}^\pm)^{(k)}) = (x_{\alpha,mr}^\pm)^{(k)}$ and $\tau_m (\Lambda_{\alpha, r})$ satisfies $\sum_{i\geq 0} \tau_m (\Lambda_{\alpha, r}) u^r = \exp \left( -\sum_{s \geq 1} \frac{h_{\alpha, ms}}{s} u^s \right)$ for all $r, m \in \bb Z$ and $\alpha \in R^+$. Consider the following power series:
\[
X^-_{\alpha,m,s}(u) = \tsum_{r=1}^\infty  x^-_{\alpha,m(r-1)+s}\ u^{r}
\qquad \text{and} \qquad
\Lambda_{\alpha,m}^\pm (u) = \tau_m(\Lambda_\alpha^\pm(u)).
\]

\begin{lem}\label{l:garland}
Let $\alpha \in R^+$, $k,l \ge 0, m>0, s\in\mathbb Z$. Then
$$\left(x^+_{\alpha,m-s}\right)^{(l)}\left(x^-_{\alpha,s}\right)^{(k)} = (-1)^l \left((X_{\alpha,m,s}^-(u))^{(k-l)}\Lambda_{\alpha,m}^+(u)\right)_k \mod U_\mathbb Z(\tlie g) U_\mathbb Z(\tlie n^+)^0,$$
where the subindex $k$ denotes the coefficient of $u^k$ of the above power series. Moreover, if $0\le s\le m$, the same holds modulo $U_\mathbb Z(\lie g[t]) U_\mathbb Z(\lie n^+[t])_\mathbb Z^0$.
\end{lem}

\begin{proof}
The case $m=1, s=0$ was proved in \cite[Lemma 7.5]{garala} (cf. \cite[Equation (1-11)]{JMhyper}). Consider the Lie algebra endomorphism $\sigma_{s}: \tlie{sl}_\alpha \to \tlie{sl}_\alpha$ given by $x^\pm_{\alpha,r}\mapsto x^\pm_{\alpha,r\mp s}$. The first statement of the lemma is obtained from the case $m=1,s=0$ by applying $(\sigma_{s}\circ\tau_m)$. The second statement is then clear.
\end{proof}

Sometimes it will be convenient to work with a smaller set of generators for the hyperalgebras.

\begin{prop}[{\cite[Corollary 4.4.12]{mitz}}]\label{p:srv}
The ring $U_\mathbb Z(\hlie g')$ is generated by $(x_i^\pm)^{(k)}, i\in\hat I,k\ge 0$ and $U_\mathbb Z(\lie g)$ is generated by $(x_i^\pm)^{(k)}, i\in I,k\ge 0$.\hfill\qedsymbol
\end{prop}

\subsection{On certain automorphisms of hyper current algebras}\label{ss:autom}

Let $\lie a, \lie b$ be such that $U_\bb Z(\lie a)$ have been defined. Then, given a homomorphism of $\bb A$-algebras $f:U_\bb A(\lie a)\to U_\bb A(\lie b)$, we have an induced homomorphism $U_\bb F(\lie a)\to U_\bb F(\lie b)$. We will now use this procedure to define certain homomorphism between hyperalgebras.
As a rule, we shall use the same symbol to denote the induced homomorphism in the hyperalgebra level.

Recall that there exists a unique involutive Lie algebra automorphism $\psi$ of $\lie g$ such that $x_i^\pm\mapsto x_i^\mp$ and $h_i\mapsto -h_i$ for all $i\in I$. It admits a unique extension to an automorphism of $\lie g[t]$ such that $\psi(x\otimes f(t)) = \psi(x)\otimes f(t)$ for all $x\in\lie g, f\in\bb C[t]$.
Keep denoting by $\psi$ its extension to an automorphism of $U(\lie g[t])$.
In particular, it easily follows that
\begin{equation}\label{e:inv+-def}
\psi\left((x_{\alpha,r}^\pm)^{(k)}\right) = (x_{\alpha,r}^\mp)^{(k)} \qquad\text{for all}\quad \alpha\in R^+, r,k\ge 0.
\end{equation}
Since $U_\bb Z(\lie g[t])$ is generated by the elements $(x_{\alpha,r}^\pm)^{(k)}$, it follows that the restriction of $\psi$ to $U_\bb Z(\lie a)$ induces an automorphism of $U_\bb Z(\lie a)$, for $\lie a=\lie g,\lie h,\lie g[t], \lie h[t],\lie h[t]_+$.
Notice that we have an inclusion $P\hookrightarrow {\rm Hom}_\bb Z(U_\mathbb Z(\lie h),\bb Z)$ determined by
\begin{equation} \label{eq:P.in.U_Fh}
\mu\left(\tbinom{h_i}{k}\right) = \tbinom{\mu(h_i)}{k} \quad\text{and}\quad \mu(xy)=\mu(x)\mu(y) \quad\text{for all}\quad  i\in I,k\ge 0, x,y\in U_\mathbb Z(\lie h).
\end{equation}
Therefore,
\begin{equation}\label{e:inv+-w}
\mu\left(\psi\left(\tbinom{h_i}{k}\right) \right) = \tbinom{-\mu(h_i)}{k} \quad\text{for all}\quad  i\in I, k>0, \mu\in P.
\end{equation}

Suppose now that $\gamma$ is a Dynkin diagram automorphism of $\lie g$ and keep denoting by $\gamma$ the $\lie g$-automor\-phism determined by
$x_i^\pm \mapsto x^\pm_{\gamma(i)}, h_i\mapsto h_{\gamma(i)}, i\in I$. It admits a unique extension to an automorphism of $\lie g[t]$ such that $\gamma(x\otimes f(t)) = \gamma(x)\otimes f(t)$ for all $x\in\lie g, f\in\bb C[t]$. Keep denoting by $\gamma$ its extension to an automorphism of $U(\lie g[t])$.
Let $\gamma$ also denote the associated automorphism of $P$ determined by $\gamma(\omega_i)=\omega_{\gamma(i)},i\in I$. In particular, $\gamma (\alpha_i)=\alpha_{\gamma(i)}, i\in I$.  It then follows that, for each $\alpha\in R^+, k>0$, there exist $\epsilon_{\alpha,k}^\pm\in\{-1,1\}$ (depending on how the Chevalley basis was chosen) such that
\begin{equation}\label{e:ddan}
\gamma\left((x_{\alpha,r}^\pm)^{(k)}\right) = \epsilon_{\alpha,k}^\pm(x_{\gamma(\alpha),r}^\pm)^{(k)} \qquad\text{for all}\quad r\ge 0.
\end{equation}
This implies that the restriction of $\gamma$ to $U_\bb Z(\lie a)$ induces an automorphism of $U_\bb Z(\lie a)$, for any $\lie a$ in the set $\{\lie g,\lie n^\pm,\lie h,\lie g[t],\lie n^\pm[t]$, $\lie h[t],\lie h[t]_+\}$. It is also easy to see that
\begin{equation}\label{e:ddah}
\mu\left(\gamma\left(\tbinom{h_i}{k}\right) \right) = \tbinom{(\gamma^{-1}(\mu))(h_i)}{k} \quad\text{for all}\quad  i\in I, k>0, \mu\in P.
\end{equation}

We end this subsection constructing the automorphism mentioned in Theorem \ref{t:isos}\eqref{t:isoslg}. Thus, let $a\in\mathbb F, \tilde a\in\mathbb A$ such that the image of $\tilde a$ in $\mathbb F$ is $a$, and $\varphi_{\tilde a}$ the Lie algebra automorphism of  $\lie g[t]_\mathbb K$ given by $x\otimes t\mapsto x\otimes(t-\tilde a)$. Keep denoting by $\varphi_{\tilde a}$ the induced automorphism of $U_\bb K(\lie g[t])$ and observe that $\varphi_{\tilde a}$ is the identity on $U_\bb K(\lie g)$. One easily checks that
\begin{equation*}
\varphi_{\tilde a}\left((x_{\alpha,r}^\pm)^{(k)}\right) = \sum_{k_0 + \dots + k_r = k} \prod_{s=0}^r \tbinom{r}{s}^{k_s} (-\tilde a)^{k_s(r-s)} (x_{\alpha,s}^\pm )^{(k_s)}\in U_\bb A(\lie g[t]).
\end{equation*}
Hence, $\varphi_{\tilde a}$ induces an automorphism of $U_\bb A(\lie g[t])$. Notice that, in the hyperalgebra level, we have
\begin{equation}\label{e:t->t-a}
(x_{\alpha,r}^\pm)^{(k)} \mapsto  \sum_{k_0 + \dots + k_r = k} \prod_{s=0}^r \tbinom{r}{s}^{k_s} (-a)^{k_s(r-s)} (x_{\alpha,s}^\pm )^{(k_s)}.
\end{equation}
This justifies a change of notation from $\varphi_{\tilde a}$ to $\varphi_{a}$.

\subsection{Subalgebras of rank 1 and 2} \label{ss:saofr12}

For any $\alpha \in R^+$, consider the Lie subalgebra of \lie g generated by $x_{\alpha}^{\pm}$ which is isomorphic to $\lie{sl}_2$. Denote this subalgebra by $\lie{sl}_\alpha$. Consider also $\lie n_{\alpha}^{\pm} = \bb C x_{\alpha}^{\pm}$, $\lie h_{\alpha} = \bb C h_{\alpha}$ and $\lie b_{\alpha}^{\pm} = \bb C h_\alpha \oplus \bb C x_{\alpha}^{\pm}$. Notice that $U_\mathbb Z(\lie g)\cap U(\lie{sl}_\alpha)$ coincides with the $\mathbb Z$-subalgebra $U_\mathbb Z(\lie{sl}_\alpha)$ of $U(\lie g)$ generated by $(x_\alpha^\pm)^{(k)},k\ge 0$ (see details in \cite{macedo:PhD}).
This implies that $U_\mathbb Z(\lie g)\cap U(\lie{sl}_\alpha)$ is naturally isomorphic to $U_\mathbb Z(\lie{sl}_2)$ and, hence, the corresponding subalgebra $U_\mathbb F(\lie{sl}_\alpha)$ of $U_\mathbb F(\tlie g)$ is naturally isomorphic to $U_\mathbb F(\lie{sl}_2)$.
Similarly, for any $\alpha \in R^+, r\in\mathbb Z$, the Lie subalgebra $\lie{sl}_{\alpha,r}$ of $\tlie g$ generated by $x_{\alpha,\pm r}^{\pm}$ is isomorphic to $\lie{sl}_2$ and $U_\mathbb Z(\tlie g)\cap U(\lie{sl}_{\alpha,r})$ coincides with the $\mathbb Z$-subalgebra of $U(\tlie g)$ generated by $(x_{\alpha,\pm r}^\pm)^{(k)}, k\ge 0$. We shall denote the corresponding subalgebra of $U_\mathbb F(\tlie g)$ by $U_\mathbb F(\lie{sl}_{\alpha,r})$. We also consider the subalgebra $\tlie{sl}_\alpha$ of $\tlie g$ generated by $x_{\alpha,r}^\pm, r\in\mathbb Z$ and the subalgebra $\lie{sl}_\alpha[t]$ of $\lie g[t]$ generated by $x_{\alpha,r}^\pm, r\geq0$. The corresponding subalgebras $U_\mathbb F(\tlie{sl}_\alpha)$ and $U_\mathbb F(\lie{sl}_\alpha[t])$ of $U_\mathbb F(\tlie g)$ are naturally isomorphic to $U_\mathbb F(\tlie{sl}_2)$ and $U_\mathbb F(\lie{sl}_2[t])$.

We will also need to work with root subsystems of rank $2$.
Suppose $\alpha,\beta\in R^+$ form a simple system of a root subsystem $R'$ of rank $2$ and let $\lie t$ denote a simple Lie algebra of type $R'$. Denote by $\lie g_{\alpha,\beta}$ the subalgebra of $\lie g$ generated by $x_{\alpha}^\pm$ and $x_\beta^\pm$, which is isomorphic to $\lie t$. Notice that, for $r,s\in\mathbb Z$, the subalgebra $\lie g_{\alpha,\beta}^{r,s}$ of $\tlie g$ generated by $x_{\alpha,\pm r}^\pm$ and $x_{\beta,\pm s}^\pm$ is also isomorphic to $\lie t$.
Let $U'_\mathbb Z(\lie g_{\alpha,\beta})$ be the subalgebra of $U_\mathbb Z(\lie g)$ generated by $(x_\alpha^\pm)^{(k)}$, $(x_\beta^\pm)^{(k)}$, $k\ge 0$, and $U'_\mathbb Z(\lie g_{\alpha,\beta}^{r,s})$ the subalgebra of $U_\mathbb Z(\tlie g)$ generated by $(x_{\alpha,\pm r}^\pm)^{(k)}, (x_{\beta,\pm s}^\pm)^{(k)}$, $k\ge 0$.
Proposition \ref{p:srv} implies that $U'_\mathbb Z(\lie g_{\alpha,\beta})$ and $U'_\mathbb Z(\lie g_{\alpha,\beta}^{r,s})$ are naturally isomorphic to $U_\mathbb Z(\lie t)$. Recall that if $\lie a$ is a subalgebra of $U(\tlie g)$, then $U_\mathbb Z(\lie a)=U(\lie a)\cap U_\mathbb Z(\tlie g)$. As in the rank 1 case, we have:
\begin{equation}\label{e:r2subalg}
U'_\mathbb Z(\lie g_{\alpha,\beta})=U_\mathbb Z(\lie g_{\alpha,\beta})\qquad\text{and}\qquad U'_\mathbb Z(\lie g_{\alpha,\beta}^{r,s})=U_\mathbb Z(\lie g_{\alpha,\beta}^{r,s}).
\end{equation}
The details can be found in \cite{macedo:PhD}.
It follows from \eqref{e:r2subalg} that $U_\mathbb F(\lie g_{\alpha,\beta})=\bb F\otimes_\bb ZU_\mathbb Z(\lie g_{\alpha,\beta})\subseteq U_\mathbb F(\lie g)$ and $U_\mathbb F(\lie g_{\alpha,\beta}^{r,s})=\bb F\otimes_\bb ZU_\mathbb Z(\lie g_{\alpha,\beta}^{r,s})\subseteq U_\mathbb F(\tlie g)$ are isomorphic to $U_\mathbb F(\lie t)$.

\subsection{The algebra $\lie g_\sh$}\label{ss:gsh}

Another important subalgebra used in the proof of Theorem \ref{t:isos} is the subalgebra $\lie g_\sh$ generated by the root vectors associated to short simple roots.

Let $\Delta_\sh = \{\alpha\in\Delta: (\alpha,\alpha)<2\}$ denote the set of simple short roots. In particular, if $\lie g$ is simply laced, $\Delta_\sh=\emptyset$. Let $R^{+}_{\sh} = \bb Z \Delta_{\sh} \cap R^{+} $ and  $R_{\sh} = \bb Z \Delta_{\sh} \cap R$ (and notice that, if $\lie g$ is not simply laced, $R_\sh\ne\{\alpha\in R: (\alpha,\alpha)<2\}$). Set $I_\sh = \{i\in I:\alpha_i\in\Delta_\sh\}$ and define $P_{\sh} = \oplus_{i \in I_{\sh}} \bb Z \omega_i$ and $P^{+}_{\sh} = P_{\sh} \cap P^+$. Consider also the subalgebras $\lie h_{\sh} = \oplus_{i \in I_{\sh}} \bb C h_{i}, \lie b_{\sh}^{\pm} = \lie h_{\sh} \oplus \lie n^{\pm}_{\sh}$, where $\lie n^{\pm}_{\sh} = \oplus_{\pm \alpha \in R_{\sh}^{+}} \lie g_{\alpha}$, and $\lie g_{\sh} = \lie n_{\sh}^{-} \oplus \lie h_{\sh} \oplus \lie n_{\sh}^{+}$. Then, if $\Delta_\sh\ne\emptyset$, $\lie g_\sh$ is a simply laced Lie subalgebra of $\lie g$ with Cartan subalgebra $\lie h_\sh$ and $\Delta_\sh$ can be identified with the choice of simple roots associated to the given triangular decomposition. The subsets $Q_\sh, Q_\sh^+$, and the Weyl group $\mathcal W_\sh$ are defined in the obvious way.
The restriction of $(\ ,\ )$ to $\lie g_\sh$ is an invariant symmetric and nondegenerate bilinear form on $\lie g_\sh$, but the normalization is not the same as the one we fixed for $\lie g$. Indeed, $(\alpha,\alpha)=2/r^\vee$ for all $\alpha\in R_\sh$. The set $\{x_\alpha^\pm, h_i: \alpha\in R_\sh^+, i\in I_\sh\}$ is a Chevalley basis for $\lie g_\sh$.

Observe that $U_\mathbb Z(\lie g)\cap U(\lie g_\sh)$ coincides with the $\mathbb Z$-subalgebra of $U(\lie g)$ generated by $(x_\alpha^\pm)^{(k)}, \alpha\in\Delta_\sh$, and, hence, Proposition \ref{p:srv} implies that $U_\mathbb F(\lie g_\sh)$ can be naturally identified with a subalgebra of $U_\mathbb F(\lie g)$. Similar observation apply to $U_\mathbb Z(\lie a)$ for $\lie a=\lie n_\sh^\pm,\lie h_\sh$.

Consider the linear map $\lie h^*\to \lie h_\sh^*, \lambda\mapsto \bar\lambda$, given by restriction and let $i_\sh:\lie h^*_\sh\to \lie h^*$ be the linear map such that $i_{\sh} \left(\overline{\alpha}\right) = \alpha$ for all $\alpha \in \Delta_{\sh}$.
In particular, $\overline{i_\sh(\mu)} = \mu$ for all $\mu\in\lie h_\sh^*$. Given $\lambda\in P$, consider the function $\eta_\lambda: P_\sh\to P$ given by
\begin{equation} \label{defn:eta.lambda}
\eta_\lambda(\mu) = i_\sh(\mu) + \lambda - i_\sh\left(\overline\lambda\right).
\end{equation}

\begin{lem}\label{l:naoiweight}
If $\lambda\in P^+, \mu\in P_\sh^+$, and $\mu\le\overline\lambda$, then $\eta_\lambda(\mu)\in P^+$. \hfill \qed
\end{lem}

\begin{proof}
For each $i\in I_\sh$, let $m_i\in\bb Z_{\ge 0}$ such that $\mu = \overline\lambda - \sum_{i\in I_{\sh}}m_i\overline\alpha_i$. In particular, $\eta_\lambda(\mu) = \lambda - \sum_{i\in I_{\sh}}m_i\alpha_i$. Then, for $j\in I_\sh$, we have $\eta_\lambda(\mu)(h_j) = \mu(h_j)\ge 0$ while, for $j\in I\setminus I_\sh$, we have $\eta_\lambda(\mu)(h_j) = \lambda(h_j) - \sum_{i\in I_{\sh}}m_i\alpha_i(h_j)\ge\lambda(h_j)\ge 0$.
\end{proof}

The affine Kac-Moody algebra associated to $\lie g_\sh$ is naturally isomorphic to the subalgebra
\[ \hlie g_\sh := \lie g_\sh\otimes\mathbb C[t,t^{-1}]\oplus \mathbb Cc\oplus\mathbb Cd \]
of $\hlie g$ and, under this isomorphism, $\hlie h_\sh$ is identified with $\lie h_\sh \oplus \mathbb Cc\oplus\mathbb Cd$. The subalgebras $\lie g_\sh[t]$ and $\hlie n_\sh^\pm$, as well as $\hat P_\sh, \hat Q_\sh$, etc, are defined in the obvious way. Moreover, $U_\mathbb F(\tlie g_\sh)$ and $U_\mathbb F(\lie g_\sh[t])$ can be naturally identified with a subalgebra of $U_\mathbb F(\tlie g)$.

\section{Finite-dimensional modules}

\subsection{Modules for hyperalgebras} \label{sec:rev.g}

We now review the finite-dimensional representation theory of $U_\mathbb F(\lie g)$. If the characteristic of \bb F is zero, then $U_\bb F (\lie g) \cong U(\lie g_\bb F)$ and the results stated here can be found in \cite{hum:lie}. The literature for the positive characteristic setting is more often found in the context of algebraic groups, in which case $U_\bb F (\lie g)$ is known as the hyperalgebra or algebra of distributions of an algebraic group of the same Lie type as \lie g (cf. \cite[Part II]{jantzen03}). A more detailed review in the present context can be found in \cite[Section 2]{JMhyper}.

Let $V$ be a $U_\mathbb F(\lie g)$-module. A nonzero vector $v\in V$ is called a weight vector if there exists $\mu\in U_\mathbb F(\lie h)^*$ such that $hv=\mu(h)v$ for all $h\in U_\mathbb F(\lie h)$. The subspace consisting of weight vectors of weight $\mu$ is called weight space of weight $\mu$ and it will be denoted by $V_\mu$. If $V= \oplus_{\mu\in U_\mathbb F(\lie h)^*} V_\mu$, then $V$ is said to be a weight module. If $V_\mu\ne 0$, $\mu$ is said to be a weight of $V$ and $\wt(V) = \{\mu\in U_\mathbb F(\lie h)^*:V_\mu\ne 0\}$ is said to be the set of weights of $V$. Notice that the inclusion \eqref{eq:P.in.U_Fh} induces an inclusion $P\hookrightarrow U_\mathbb F(\lie h)^*$.
In particular, we can consider the partial order $\le$ on $U_\mathbb F(\lie h)^*$ given by $\mu\le\lambda$ if $\lambda-\mu\in Q^+$ and we have
\begin{equation}\label{e:xactonws}
(x_\alpha^\pm)^{(k)} V_\mu \subseteq V_{\mu\pm k\alpha}\quad\text{for all}\quad \alpha\in R^+, k>0,\mu\in U_\mathbb F(\lie h)^*.
\end{equation}
If $V$ is a weight-module with finite-dimensional weight spaces, its character is the function $\ch(V):U_\mathbb F(\lie h)^*\to \mathbb Z$ given by $\ch(V)(\mu)=\dim V_\mu$. As usual, if $V$ is finite-dimensional, $\ch(V)$ can be regarded as an element of the group ring $\mathbb Z[U_\mathbb F(\lie h)^*]$ where we denote the element corresponding to $\mu\in U_\mathbb F(\lie h)^*$ by $e^\mu$. By the inclusion \eqref{eq:P.in.U_Fh} the group ring $\mathbb Z[P]$ can be regarded as a subring of $\mathbb Z[U_\mathbb F(\lie h)^*]$ and, moreover, the action of $\cal W$ on $P$ induces an action of $\cal W$ on $\mathbb Z[P]$ by ring automorphisms where $w \cdot e^\mu = e^{w\mu}$.

If $v \in V$ is weight vector such that $(x_\alpha^+)^{(k)} v = 0$ for all $\alpha\in R^+, k>0$, then $v$ is said to be a highest-weight vector. If $V$ is generated by a highest-weight vector, then it is said to be a highest-weight module. Similarly, one defines the notions of lowest-weight vectors and modules by replacing $(x_\alpha^+)^{(k)}$ by $(x_\alpha^-)^{(k)}$.

\begin{thm} \label{t:rh}
Let $V$ be a $U_\mathbb F(\lie g)$-module.
\begin{enumerate}[(a)]
\item If $V$ is finite-dimensional, then $V$ is a weight-module, $\wt (V) \subseteq P$, and $\dim V_\mu = \dim V_{\sigma\mu}$ for all $\sigma\in\cal W, \mu \in U_\bb F(\lie h)^\ast$. In particular, $\ch(V)\in\mathbb Z[P]^\cal W$.
\item If $V$ is a highest-weight module of highest weight $\lambda$, then $\dim(V_{\lambda})=1$ and $V_{\mu}\ne 0$ only if $\mu\le \lambda$. Moreover, $V$  has a unique maximal proper submodule and, hence, also a unique irreducible quotient. In particular, $V$ is indecomposable.
\item{\label{t:rh.c}} For each $\lambda\in P^+$, the $U_\mathbb F(\lie g)$-module $W_\mathbb F(\lambda)$ given by the quotient of $U_\mathbb F(\lie g)$ by the left ideal $I_\mathbb F(\lambda)$ generated by
\begin{equation*}
U_\mathbb F (\lie n^+)^0, \quad h-\lambda(h) \quad \text{and} \quad (x_\alpha^-)^{(k)}, \quad\text{for all}\quad h\in U_\mathbb F(\lie h), \alpha\in R^+, k>\lambda(h_\alpha),
\end{equation*}
is nonzero and finite-dimensional. Moreover, every finite-dimensional highest-weight module of highest weight $\lambda$ is a quotient of $W_\mathbb F(\lambda)$.
\item If $V$ is finite-dimensional and irreducible, then there exists a unique $\lambda\in P^+$ such that $V$ is isomorphic to the irreducible quotient $V_\mathbb F(\lambda)$ of $W_\mathbb F(\lambda)$. If the characteristic of $\mathbb F$ is zero, then $W_\bb F(\lambda)$ is irreducible.
\item\label{t:chWg} For each $\lambda\in P^+$, $\ch(W_\mathbb F(\lambda))$ is given by the Weyl character formula. In particular, $\mu\in\wt(W_\mathbb F(\lambda))$ if, and only if, $\sigma\mu\le\lambda$ for all $\sigma\in\cal W$. Moreover, $W_\mathbb F(\lambda)$ is a lowest-weight module with lowest weight $w_0\lambda$.  \hfill \qed
\end{enumerate}
\end{thm}

\begin{rem}
The module $W_\mathbb F(\lambda)$ defined in Theorem \ref{t:rh} \eqref{t:rh.c} is called Weyl module (or costandard module) of highest weight $\lambda$. The known proofs of Theorem \ref{t:rh}\eqref{t:chWg} make use of geometric results such as Kempf's Vanishing Theorem.
\end{rem}

We shall need the following lemma in the proof of Lemma \ref{lem:g.naoi} below.

\begin{lem}\label{l:lowtohigh}
Let $V$ be a finite-dimensional $U_\bb F(\lie g)$-module, $\mu\in P$, and $\alpha\in R^+$. If $v\in V_\mu\setminus\{0\}$ is such that $(x_\alpha^-)^{(k)}v=0$ for all $k>0$, then $\mu(h_\alpha)\in\bb Z_{\le 0}$ and $(x_\alpha^+)^{(-\mu(h_\alpha))}v\ne 0$.\hfill\qedsymbol
\end{lem}

\begin{rem}
In characteristic zero, it is well-known that the following stronger statement holds: if $v\in V_\mu\setminus\{0\}$ is such that $\mu(h_\alpha)\in\bb Z_{\le 0}$, then $(x_\alpha^+)^{(-\mu(h_\alpha))}v\ne 0$. In positive characteristic this stronger statement is not true for all finite-dimensional representations.
\end{rem}

The next lemma can be proved exactly as in \cite[Lemma 4.5]{naoi:weyldem}.

\begin{lem}\label{l:weyl-relhr}
Let $m_i\in\mathbb Z_{\ge 0}, i\in I, V$ a finite-dimensional $U_\bb F (\lie n^-)$-module and suppose $v \in V$ satisfies $(x_i^{-})^{(k)} v = 0$ for all $i\in I, k > m_i$. Then, given $\alpha\in R^+$, we have $(x_{\alpha}^{-})^{(k)} v = 0$ for all $k > \sum_{i\in I}n_i m_i$ where $n_i$ are such that $h_\alpha = \sum_{i\in I}n_i h_i$.\hfill\qedsymbol
\end{lem}

\subsection{Modules for hyper loop algebras}{\label{sec:fdweyl}}

We now recall some basic results about the category of finite-dimensional $U_\mathbb F(\tlie g)$-modules in the same spirit as section \ref{sec:rev.g}. The results of this subsection can be found in \cite[Section 3]{JMhyper} and references therein.

Given a $U_\mathbb F(\tlie g)$-module $V$ and $\xi\in U_\mathbb F(\tlie h)^*$, let
\begin{equation*}
V_\xi=\{v\in V: \text{ for all } x\in U_\mathbb F(\tlie h), \text{ there exists } k>0 \text{ such that } (x-\xi(x))^kv = 0\}.
\end{equation*}
We say that $V$ is an $\ell$-weight module if $V = \opl_{\gb\omega\in\cal P_\mathbb F}^{} V_\gb\omega$. In this case, regarding $V$ as a $U_\bb F (\lie g)$-module, we have
\begin{equation*}
V_\mu=\opl_{\substack{\gb\omega\in\cal P_\mathbb F:\\ \wt(\gb\omega)=\mu}}^{} V_\gb\omega \quad\text{for all}\quad \mu\in P \quad\text{and}\quad V=\opl_{\mu\in P}^{} V_\mu.
\end{equation*}
A nonzero element of $V_\gb\omega$ is said to be an $\ell$-weight vector of $\ell$-weight $\gb\omega$. An $\ell$-weight vector $v$ is said to be a highest-$\ell$-weight vector if $U_\mathbb F(\tlie h)v=\mathbb Fv$ and $(x_{\alpha,r}^+)^{(k)}v = 0$ for all $\alpha\in R^+$ and all $r,k\in\mathbb Z, k>0$. If $V$ is generated by a highest-$\ell$-weight vector of $\ell$-weight $\gb\omega$, $V$ is said to be a highest-$\ell$-weight module of highest $\ell$-weight $\gb\omega$.

\begin{thm}\label{t:fdrhla} Let $V$ be a $U_\mathbb F(\tlie g)$-module.
\begin{enumerate}[(a)]
\item If $V$ is finite-dimensional, then $V$ is an $\ell$-weight module. Moreover, if $V$ is finite-dimensional and irreducible, then $V$ is a highest-$\ell$-weight module whose highest $\ell$-weight lies in $\cal P_\mathbb F^+$.
\item If $V$ is a highest-$\ell$-weight module of highest $\ell$-weight $\gb\omega\in\cal P_\mathbb F^+$, then $\dim V_\gb\omega=1$ and $V_{\mu}\ne 0$ only if $\mu\le \wt(\gb\omega)$. Moreover, $V$  has a unique maximal proper submodule and, hence, also a unique irreducible quotient. In particular, $V$ is indecomposable.
\item\label{t:wmfdim} For each $\gb\omega\in \cal P_\mathbb F^+$, the local Weyl module $W_\mathbb F(\gb\omega)$ is nonzero and finite-dimensional. Moreover, every finite-dimensional highest-$\ell$-weight-module of highest $\ell$-weight $\gb\omega$ is a quotient of $W_\mathbb F(\gb\omega)$.
\item If $V$ is finite-dimensional and irreducible, then there exists a unique $\gb\omega\in \cal P_\bb F^+$ such that $V$ is isomorphic to the irreducible quotient $V_\mathbb F(\gb\omega)$ of $W_\mathbb F(\gb\omega)$.
\item For $\mu\in P$ and $\gb\omega\in\cal P_\mathbb F^+$, we have $\mu\in\wt(W_\mathbb F(\gb\omega))$ if and only if $\mu\in\wt(W_\mathbb F(\wt(\gb\omega)))$, i.e. $w\mu \leq \wt(\gb \omega)$, for all $w \in \cal W$.
\hfill\qedsymbol
\end{enumerate}
\end{thm}

\subsection{Graded modules for hyper current algebras}\label{ss:grmhca}

Recall the following elementary fact.

\begin{lem}\label{l:gotoF}
Let $A$ be a ring, $I \subset A$ a left ideal, $B=\mathbb F\otimes_\mathbb Z A$ an $\bb F$-algebra, and $J$ the image of $I$ in $B$, i.e. $J$ is the $\mathbb F$-span of $\{ (1\otimes a)\in B : a\in I\}$. Then $\mathbb F\otimes_\mathbb Z (A/I)$ is a left $B$-module, $J$ is a left ideal of $B$, and we have an isomorphism of left $B$-modules $B/J\cong \mathbb F\otimes_\mathbb Z (A/I)$.\hfill\qedsymbol
\end{lem}

We shall use Lemma \ref{l:gotoF} with $A$ being one of the integral forms so that $B$ is the corresponding hyperalgebra.

Given $\lambda\in P^+$, let $I_\bb Z^c(\lambda) \subset U_\bb Z (\lie g [t])$ be the left ideal generated by
\begin{gather*}
U_\mathbb Z (\lie n^+[t])^0, \quad U_\mathbb Z (\lie h [t]_+)^0, \quad h - \lambda(h), \quad (x_\alpha^-)^{(k)}, \quad \text{for all}\quad h\in U_\mathbb Z(\lie h),\ \alpha\in R^+,\ k>\lambda(h_\alpha),
\end{gather*}
and set
\[ W_\bb Z^c (\lambda) = U_\bb Z (\lie g [t])/ I_\bb Z^c (\lambda). \]
Similarly, if $\ell\ge 0$ is also given, let $I_\bb Z (\ell,\lambda)$ be the left ideal of $U_\mathbb Z (\lie g[t])$ generated by
\begin{gather*}
U_\mathbb Z (\lie n^+[t])^0, \quad U_\mathbb Z (\lie h [t]_+)^0, \quad h - \lambda(h), \quad (x_{\alpha,s}^-)^{(k)},\quad \text{for all}\quad h\in U_\mathbb Z (\lie h),\ \alpha\in R^+,\\ s,k\in\mathbb Z_{\ge 0}, \ k>\ \max \{0, \lambda(h_\alpha) - r^\vee_\alpha\ell s \}.
\end{gather*}
Then set \[ D_\bb Z (\ell, \lambda) = U_\bb Z (\lie g [t])/ I_\bb Z (\ell, \lambda). \]
Notice that $W^c_\mathbb Z (\lambda)$ and $D_\mathbb Z (\ell, \lambda)$ are weight modules.

Since the ideals defining $W^c_\mathbb F(\lambda)$ and $D_\mathbb F(\ell, \lambda)$ (cf. Subsection \ref{sec:demweyl}) are the images of $I^c_\mathbb Z(\lambda)$ and $I_\mathbb Z(\ell,\lambda)$ in $U_\mathbb F(\lie g[t])$, respectively, an application of Lemma \ref{l:gotoF} gives isomorphisms of $U_\bb F (\lie g [t])$-modules
\begin{equation*}
W^c_\mathbb F(\lambda) \cong \mathbb F \otimes_\mathbb Z W^c_\mathbb Z (\lambda) \quad\text{and}\quad D_\mathbb F(\ell, \lambda) \cong \mathbb F \otimes_\mathbb Z D_\mathbb Z (\ell, \lambda).
\end{equation*}

As before, $D_\mathbb Z (\ell, \lambda)$ is a quotient of $W^c_\mathbb Z (\lambda)$ for all $\lambda\in P^+$ and all $\ell>0$. We shall see next (Proposition \ref{p:WZfg}) that the latter is a finitely generated $\mathbb Z$-module and, hence, so is the former. Together with Corollary \ref{c:indf}, this implies that
\begin{equation}
D_\mathbb Z (\ell, \lambda) \ \text{ is a free } \mathbb Z\text{-module.}
\end{equation}

The proof of the next proposition is an adaptation of that of \cite[Theorem 3.11]{JMhyper}. The extra details can be found in \cite{macedo:PhD}.

\begin{prop} \label{p:WZfg}
For every $\lambda\in P^+$, the $U_\bb Z (\lie g[t])$-module $W^c_\mathbb Z(\lambda)$ is a finitely generated $\mathbb Z$-module.
\end{prop}

We now prove an analogue of Theorem \ref{t:weylform} for graded local Weyl modules.

\begin{cor}\label{c:gweylfd}
Let $\lambda\in P^+$ and $v$ be the image of $1$ in $W^c_\mathbb C(\lambda)$. Then $U_\mathbb Z(\lie g[t])v$ is a free $\mathbb Z$-module of rank $\dim(W^c_\mathbb C(\lambda))$. Moreover, $U_\mathbb Z(\lie g[t])v = \oplus_{\mu\in P} (U_\mathbb Z(\lie g[t])v\cap W^c_\mathbb C(\lambda)_\mu)$. In particular, $U_\mathbb Z(\lie g[t])v$ is an integral form for $W^c_\mathbb C(\lambda)$.
\end{cor}

\begin{proof}
To simplify notation, set $L=U_\bb Z(\lie n^-)v$. Let also $\v$ be as in the proof of Proposition \ref{p:WZfg}. Since $v$ satisfies the relations satisfied by $\v$, it follows that there exists an epimorphism of $U_\bb Z(\lie g[t])$-modules $W_\bb Z(\lambda)\to L, \v\mapsto v$. Since $W_\bb Z(\lambda)$ is finitely generated, it follows that so is $L$. On the other hand, since $L \subseteq W_\bb C^c (\lambda)$, it is also torsion free and, hence, a free $\bb Z$-module of finite-rank. Since $U_\bb Z(\lie n^-)$ spans $U(\lie n^-)$ and $W_\mathbb C^c(\lambda)=U(\lie n^-)v$, it follows that $L$ contains a basis of $W_\mathbb C^c(\lambda)$. This implies that the rank of $L$ is at least $\dim (W^c_\mathbb C(\lambda))$. On the other hand, $\mathbb C\otimes_\mathbb Z L$ is a $\lie g[t]$-module generated by the vector $1\otimes v$ which satisfies the relations \eqref{e:glWrel}. Therefore, it is a quotient of $W^c_\mathbb C(\lambda)$. Since $\dim(\mathbb C\otimes_\mathbb Z L)={\rm rank}(L)$, the first and the last statements follow. The second statement is clear since $L$ is obviously a weight module.
\end{proof}

Consider the category $\cal G_\mathbb F$ of $\mathbb Z$-graded finite-dimensional representations of $U_\mathbb F(\lie g[t])$.
Recall the functors $\tau_m$ defined in the paragraph preceding Remark \ref{ref:UF}.
For each  $U_\bb F (\lie g)$-module $V$, let $\ev_0(V)$ be the module in $\mathcal G_\mathbb F$ obtained by extending the action of $U_\bb F(\lie g)$ to one of $U_\bb F(\lie g[t])$ on $V$ by setting $U_\bb F(\lie g [t]_+)V=0$. For $\lambda\in P^+,r\in\mathbb Z$, set $V_\mathbb F(\lambda,r)=\ev_r(V_\mathbb F(\lambda))$ where $\ev_r=\tau_r\circ\ev_0$.

\begin{thm}\label{t:fdrhca}
Let $\lambda\in P^+$
\begin{enumerate}[(a)]
\item If $V$ in $\cal G_\bb F$ is simple, then it is isomorphic to $V_\mathbb F(\lambda,r)$ for unique $(\lambda,r)\in P^+\times\bb Z$.
\item\label{t:gweylfd} $W^c_\mathbb F(\lambda)$ is finite-dimensional.
\item\label{t:gweylu} If $V$ is a graded finite-dimensional $U_\mathbb F(\lie g[t])$-module generated by a weight vector $v$ of weight $\lambda$ satisfying $U_\mathbb F(\lie n^+[t])^0v=U_\mathbb F(\lie h[t]_+)^0v=0$, then $V$ is a quotient of $W^c_\mathbb F(\lambda)$.
\end{enumerate}
\end{thm}

\begin{proof}
To prove part (a), 
suppose $V \in \cal G_\bb F$ is simple. If $V[r], V[s] \neq 0$ for $s < r \in \bb Z$, $(\oplus_{k\geq r} V [k])$ would be a proper submodule of $V$, contradicting the fact that it is simple. Thus there must exist a unique $r \in \bb Z$ such that $V[r] \neq 0$. Since $U_\bb F (\lie g[t]_+)$ changes degrees, $V = V[r]$ must be a simple $U_\bb F (\lie g)$-module. This shows that $V \cong V_\bb F (\lambda,r)$ for some $\lambda \in P^+, r\in \bb Z$.

To prove part \eqref{t:gweylfd}, observe that $W^c_\mathbb F(\lambda) \cong \bb F \otimes_\bb Z W_\bb Z^c (\lambda)$ (cf. Lemma \ref{l:gotoF}). Thus the dimension of $W^c_\mathbb F(\lambda)$ must be at most the number of generators of $W^c_\bb Z(\lambda)$, which is proved to be finite in Proposition \ref{p:WZfg}.

To prove part \eqref{t:gweylu}, observe that the $U_\mathbb F(\lie g)$-submodule $V' = U_\mathbb F(\lie g) v \subseteq V$ is a finite-dimensional highest-weight module of highest weight $\lambda$. Thus, by Theorem \ref{t:rh}\eqref{t:rh.c}, $V'$ is a quotient of $W_\mathbb F(\lambda)$. The statement follows by comparing the defining relations of $V$ and $W^c_\mathbb F(\lambda)$.
\end{proof}

\begin{rem} \label{r:gweylf}
Denote by $v$ the image of $1$ in $W_\mathbb F^c(\lambda)$. From the defining relations \eqref{e:glWrel} it follows that $\mathbb F\otimes_\mathbb Z U_\mathbb Z(\lie g[t])v$ is a quotient of $W_\mathbb F^c(\lambda)$. It follows from Theorem \ref{t:isos}\eqref{t:isofil} that $\mathbb F\otimes_\mathbb Z U_\mathbb Z(\lie g[t])v \cong W^c_\mathbb F(\lambda)$ for all $\lambda\in P^+$ (cf. Section \ref{ss:JM.conjecture} below). Moreover, since $\mathbb F\otimes_\mathbb Z W_\bb Z(\lambda) \cong W^c_\mathbb F(\lambda)$, Theorem \ref{t:isos}\eqref{t:isofil} also implies that $W_\bb Z(\lambda)$ is free.
\end{rem}

\subsection{Proof of \eqref{e:conj}} \label{ss:JM.conjecture}

The argument of the proof will use Corollary \ref{c:indf}, the characteristic zero version of parts \eqref{t:isoslg} and \eqref{t:isostp} of Theorem \ref{t:isos}, and the following proposition.

\begin{prop}{\cite[Corollary A]{naoi:weyldem}}\label{p:dimwbyfund}
Let $\lambda\in P^+$. Then, $\dim W_\bb C^c(\lambda) = \prod_{i\in I} (\dim W_\bb C^c(\omega_i))^{\lambda(h_i)}$.\hfill\qedsymbol
\end{prop}

We shall also need the following general construction. Given a $\bb Z_{s\ge 0}$-filtered $U_\bb F(\lie g[t])$-module $W$, we can consider the associated graded $U_\mathbb F(\lie g[t])$-module ${\rm gr}(W) = \oplus_{s\geq0} W_s/W_{s-1}$ which obviously has the same dimension as $W$. Suppose now that $W$ is any cyclic $U_\mathbb F(\lie g[t])$-module and fix a generator $w$. Then, the $\bb Z$-grading on $U_\mathbb F(\lie g[t])$ induces a filtration on $W$. Namely, set $w$ to have degree zero and define the $s$-th filtered piece of $W$ by $W_s = F^sU_\mathbb F(\lie g[t]) w$ where $F^s U_\mathbb F(\lie g[t]) = \oplus_{r \leq s} U_\mathbb F(\lie g[t]) [r]$. Then, ${\rm gr}(W)$ is cyclic since it is generated by the image of $w$ in $\gr(W)$.

Recall the notation fixed for \eqref{e:conj}: $\gb\omega\in\mathcal P_\mathbb A^\times$, $\lambda=\wt(\gb\omega)$, $\gb\varpi$ is the image of $\gb\omega$ in $\cal P_\mathbb F$. Also recall that, using \eqref{e:conjknown},  \eqref{e:conj} will be proved if we show that
$$\dim W_\bb F (\gb \varpi) \leq \dim W_\bb K (\gb \omega).$$

Fix $w\in W_\mathbb F(\gb\varpi)_\lambda\setminus\{0\}$. Not only $w$ generates $W_\mathbb F(\gb\varpi)$ as a $U_\bb F (\tlie g)$-module, but it also follows from the proof of \cite[Theorem 3.11]{JMhyper} (with a correction incorporated in the proof of \cite[Theorem 3.7]{jm:hlanac}) that $U_\mathbb F(\lie n^-[t])w = W_\mathbb F(\gb\varpi)$. Hence, we can apply the general construction reviewed above to $W_\mathbb F(\gb\varpi)$. Set $V=\gr(W_\mathbb F(\gb\varpi))$ and denote the image of $w$ in $V$ by $v$. The module $V$ is finite-dimensional and $v$ is a highest-weight vector of weight $\lambda$ satisfying $U_\mathbb F(\lie h[t]_+)^0v=0$ (the latter follows since $\dim(V_\lambda) =1$, $V$ is graded, and $U_\mathbb F(\lie h[t])$ is commutative). Hence, $v$ satisfies the defining relations \eqref{e:glWrel} of $W_\mathbb F^c(\lambda)$. In particular, we get that
$$\dim W_\mathbb F(\gb\varpi)\le \dim W_\mathbb F^c(\lambda).$$
Since $\dim W_\mathbb F^c(\lambda)=\dim W_\mathbb K^c(\lambda)$ by Corollary \ref{c:indf}, it now suffices to show that
$$\dim W_\mathbb K^c(\lambda) = \dim W_\mathbb K(\gb\omega).$$
For proving this, consider the decomposition of $\gb\omega$ of the form
$$\gb\omega = \prod_{j=1}^m \gb\omega_{\lambda_j,a_j} \quad\text{for some}\quad m\ge 0,\ a_j \in \bb K^\times,\ a_i\ne a_j\ \ \text{for}\ \ i\ne j,\ \lambda_j\in P^+ \quad\text{such that}\quad \lambda=\sum_{j=1}^m\lambda_j.$$
By Theorem \ref{t:isos}\eqref{t:isostp} (in characteristic zero) $W_\mathbb K(\gb\omega)\cong \otimes_{j=1}^m W_\mathbb K(\gb\omega_{\lambda_j,a_j})$. Theorem \ref{t:isos}\eqref{t:isoslg} (in characteristic zero) implies that $\dim W_\mathbb K(\gb\omega_{\lambda_j,a_j}) =  \dim W_\mathbb K^c(\lambda_j)$. Hence,
\begin{align*}
\dim W_\mathbb K(\gb\omega) =\prod_{j=1}^m \dim W^c_\bb K(\lambda_j) = \prod_{j=1}^m\prod_{i\in I} \dim W^c_\bb K(\omega_i)^{\lambda_j(h_i)} = \prod_{i\in I} W^c_\bb K(\omega_i)^{\lambda(h_i)} = \dim W_\bb K^c(\lambda).
\end{align*}
Here, the second and last equality follow from Proposition \ref{p:dimwbyfund} and the others are clear. This completes the proof of \eqref{e:conj}.

Notice that all equalities of dimensions proved here actually imply the corresponding equalities of characters. In particular, it follows that
\begin{equation}\label{e:chWeyl}
\ch (W_\bb F(\gb\varpi)) = \prod_{i\in I} (\ch (W_\bb C^c(\omega_i)))^{\wt(\gb\varpi)(h_i)} \quad\text{for all}\quad \gb\varpi\in\cal P_\bb F^+.
\end{equation}

\subsection{Joseph-Mathieu-Polo relations for Demazure modules} \label{ss:Mrel}

We now explain the reason why we call the module $D_\mathbb F(\ell,\lambda)$ by Demazure modules. We begin with the following lemma. Let $\gamma$ be the Dynkin diagram automorphism of $\lie g$ induced by $w_0$ and recall from Section \ref{ss:autom} that it induces an automorphism of $U_\bb F(\lie g[t])$ also denoted by $\gamma$.

\begin{lem}\label{l:Dem*}
Let $\lambda\in P^+, \ell\ge 0$, and set $\lambda^*=-w_0\lambda$. Let $W$ be the pull-back of $D_\mathbb F(\ell,\lambda^*)$ by $\gamma$. Then, $D_\mathbb F(\ell,\lambda)\cong W$.
\end{lem}

\begin{proof}
Let $v\in D_\mathbb F(\ell,\lambda^*)_{\lambda^*}\setminus\{0\}$. By \eqref{e:glWrel} and \eqref{e:FLrel} we have
\begin{gather*}
U_\mathbb F(\lie n^+[t])^0v = U_\mathbb F(\lie h [t]_+)^0v=0, \qquad hv= \lambda^*(h)v, \qquad (x_{\alpha,s}^-)^{(k)}v=0,
\end{gather*}
for all $h\in U_\mathbb F(\lie h), \alpha\in R^+, s,k\in\mathbb Z_{\ge 0}, k>\ \max \{0, \lambda^*(h_\alpha) - s\ell r^\vee_\alpha \}$.
Denote by $w$ the vector $v$ regarded as an element of $W$. Evidently, $W=U_\bb F(\lie g[t])w$.
Since $\gamma$ restricts to automorphisms of $U_\bb F(\lie n^+[t])$ and of $U_\bb F(\lie h[t]_+)$, it follows that $U_\mathbb F(\lie n^+[t])^0w = U_\mathbb F(\lie h [t]_+)^0w=0$, while \eqref{e:ddah} implies that $w\in W_\lambda$. Finally, \eqref{e:ddan} and \eqref{e:ddah} together imply that
\begin{gather*}
(x_{\alpha,s}^-)^{(k)}w=0 \qquad\text{for all}\qquad \alpha\in R^+, s,k\in\mathbb Z_{\ge 0}, k>\ \max \{0, \lambda(h_\alpha) - s\ell r^\vee_\alpha \}.
\end{gather*}
This shows that $w$ satisfies the defining relations of $D_\mathbb F(\ell,\lambda)$ and, hence, there exists an epimorphism from $D_\mathbb F(\ell,\lambda)$ onto $W$. Since $(\lambda^*)^*=\lambda$, reversing the roles of $\lambda$ and $\lambda^*$ we get an epimorphism on the other direction. Since these are finite-dimensional modules, we are done.
\end{proof}

In order to continue, we need the concepts of weight vectors, weight spaces, weight modules and integrable modules for $U_\bb F(\hlie g')$ which are similar to those for $U_\bb F (\lie g)$ (cf. subsection \ref{sec:rev.g}) by replacing $I$ with $\hat I$ and $P$ with $\hat P'$. Also, using the obvious analogue of \eqref{eq:P.in.U_Fh}, we obtain an inclusion $\hat P'\hookrightarrow U_\mathbb F(\hlie h')^*$. Let $V$ be a $\bb Z$-graded $U_\bb F(\hlie g')$-module whose weights lie in $\hat P'$. As before, let $V[r]$ denote the $r$-th graded piece of $V$. For $\mu\in\hat P$, say $\mu = \mu'+m\delta$ with $\mu'\in\hat P', m\in\bb Z$, set
$$V_\mu = \{ v\in V[m]: hv=\mu'(h)v \text{ for all } h\in U_\bb F(\hlie h')\}.$$
If $V_\mu\ne 0$ we shall say that $\mu$ is a weight of $V$ and let $\wt(V) = \{\mu\in\hat P:V_\mu\ne 0\}$.

We record the following partial affine analogue of Theorem \ref{t:rh}.

\begin{thm} \label{t:rha}
Let $V$ be a graded $U_\mathbb F(\hlie g')$-module.
\begin{enumerate}[(a)]
\item If $V$ is integrable, then $V$ is a weight-module and $\wt(V) \subseteq\hat P$. Moreover, $\dim V_\mu = \dim V_{\sigma\mu}$ for all $\sigma\in\wa, \mu \in \hat P$.
\item If $V$ is a highest-weight module of highest weight $\lambda$, then $\dim(V_{\lambda})=1$ and $V_{\mu}\ne 0$ only if $\mu\le \lambda$. Moreover, $V$  has a unique maximal proper submodule and, hence, also a unique irreducible quotient. In particular, $V$ is indecomposable.
\item{\label{t:rh.ca}} Let $\Lambda\in\hat P^+$ and $m=\Lambda(d)$. Then, the $U_\mathbb F(\hlie g')$-module $\hat W_\mathbb F(\Lambda)$ generated by a vector $v$ of degree $m$ satisfying the defining relations
\begin{equation*}
U_\mathbb F (\hlie n^+)^0v=0, \quad hv=\Lambda(h)v \quad \text{and} \quad (x_i^-)^{(k)}v=0, \quad\text{for all}\quad h\in U_\mathbb F(\hlie h'), i\in\hat I, k>\Lambda(h_i),
\end{equation*}
is nonzero and integrable. Moreover, for every positive real root $\alpha$, we have
\begin{equation}
(x_\alpha^-)^{(k)}v = 0 \qquad\text{for all}\qquad k>\Lambda(h_\alpha).
\end{equation}
Furthermore, every integrable highest-weight module of highest weight $\Lambda$ is a quotient of $\hat W_\mathbb F(\Lambda)$.
\hfill \qed
\end{enumerate}
\end{thm}

Given $\Lambda \in \hp^+, \sigma \in \wa$, the Demazure module $V_\mathbb F^\sigma(\Lambda)$ is defined as the $U_\mathbb F(\hlie b'^+)$-submodule generated by $\hat W_\mathbb F(\Lambda)_{\sigma\Lambda}$ (cf. \cite{foli:weyldem,mathieu89,naoi:weyldem}). In particular, $V_\mathbb F^\sigma(\Lambda) \cong V_\mathbb F^{\sigma'}(\Lambda)$ if $\sigma\Lambda=\sigma'\Lambda$ for some $\sigma'\in \wa$.
Our focus is on the Demazure modules which are stable under the action of $U_\mathbb F(\lie g)$. 
Since $V_\mathbb F^{\sigma}(\Lambda)$ is defined as a $U_\bb F (\hlie b'^+)$-module, it is stable under the action of $U_\mathbb F(\lie g)$ if, and only if,
\begin{equation}\label{e:n-inv}
U_\mathbb F(\lie n^-)^0\hat W_\mathbb F(\Lambda)_{\sigma\Lambda}=0 .
\end{equation}
In particular, since $V_\mathbb F^\sigma(\Lambda)$ is an integrable $U_\bb F(\lie{sl}_\alpha)$-module for any $\alpha\in R^+$, it follows that $(\sigma\Lambda)(h_\alpha)\le 0$ for all $\alpha\in R^+$. Conversely, using the exchange condition for Coxeter groups (see \cite[Section 5.8]{humphreys90}), one easily deduces that, for all $i\in\hat I$, we have
\begin{gather*}
(x_i^\epsilon)^{(k)}\hat W_\mathbb F(\Lambda)_{\sigma\Lambda} = 0 \quad\text{for all}\quad k>0
\end{gather*}
where $\epsilon=+$ if $\sigma\Lambda(h_i)\ge 0$ and $\epsilon=-$ if $\sigma\Lambda(h_i)\le 0$. This implies that, if $\sigma\Lambda(h_i)\le 0$ for all $i\in I$, then $V_\mathbb F^\sigma(\Lambda)$ is $U_\bb F(\lie g)$-stable. Thus, henceforth, assume $(\sigma\Lambda)(h_i)\le 0$ for all $i\in I$ and observe that this implies that $\sigma \Lambda$ must have the form
\begin{equation}\label{e:thelevel}
\sigma\Lambda = \ell \Lambda_0 + w_0 \lambda + m \delta \qquad\text{for some}\qquad \lambda\in P^+, m\in\bb Z, \ \text{and}\ \ell=\Lambda(c).
\end{equation}
Conversely, given $\ell\in\bb Z_{\ge 0}, \lambda\in P^+$, and $m\in\bb Z$, since $\wa$ acts simply transitively on the set of alcoves of $\hlie h^\ast$ (cf. \cite[Theorem 4.5.(c)]{humphreys90}), there exists a unique $\Lambda \in \hp^+$ such that $\ell\Lambda_0+w_0\lambda+m\delta\in\wa\Lambda$. Thus, if $\sigma\in\widehat{\cal W}$ and $\Lambda \in\hp^+$ are such that
\begin{equation}\label{e:notcon}
\sigma\Lambda=\ell\Lambda_0+w_0\lambda+m\delta,
\end{equation}
then $V_\mathbb F^\sigma(\Lambda)$ is $U_\bb F(\lie g)$-stable. Henceforth, we fix $\sigma,\Lambda, w_0,\lambda,$ and $m$ as in \eqref{e:notcon}.
Notice that, if $\gamma=\pm \alpha+s\delta\in\hat R^+$ with $\alpha\in R^+$, then
\begin{equation*}
\sigma\Lambda(h_\gamma) = \pm w_0\lambda(h_\alpha)+s\ell r^\vee_\alpha.
\end{equation*}
The following lemma is a rewriting of \cite[Lemme 26]{mathieu89} using the above fixed notation.

\begin{lem}\label{l:jpmrel}
The $U_\bb F(\hlie b'^+)$-module $V_\bb F^\sigma(\Lambda)$ is isomorphic to the $U_\bb F(\hlie b'^+)$-module generated by a vector $v$ of degree $m$ satisfying the following defining relations: $hv=\sigma\Lambda(h)v, h\in U_\bb F(\hlie h')$, $U_\mathbb F(\lie h [t]_+)^0v=U_\mathbb F(\lie n^-[t]_+)^0v = 0$, and
\begin{gather}\label{e:MFLrel+}
(x_{\alpha,s}^+)^{(k)}v=0 \quad\text{for all} \quad \alpha\in R^+,\ s\ge 0,\ k>\max\{0,- w_0\lambda(h_\alpha)-s\ell r^\vee_\alpha\}.
\end{gather}
\hfill\qedsymbol
\end{lem}

\begin{rem}
In \cite{mathieu89}, Mathieu attributes Lemma \ref{l:jpmrel} to Joseph and Polo. This is the reason for the title of this subsection. The original version of this Lemma in \cite{mathieu89} gives generator and relations for any Demazure module, not only for the $U_\bb F(\lie g)$-stable ones.
\end{rem}

The following is the main result of this subsection.

\begin{prop}\label{p:Demrels}
The graded $U_\bb F(\lie g[t])$-modules $V_\mathbb F^\sigma(\Lambda)$ and $D_\mathbb F(\ell,\lambda,m)$ are isomorphic.
\end{prop}

\begin{proof}
It suffices to prove the statement for $m=0$ and, thus, for simplicity, we assume that this is the case. Proceeding as in \cite[Corollary 1]{foli:weyldem} (see also \cite[Proposition 3.6]{naoi:weyldem}) one shows that $V_\mathbb F^\sigma(\Lambda)$ is a quotient of $D_\mathbb F(\ell,\lambda)$. Namely, let $v$ be a nonzero vector in $\hat W_\mathbb F(\Lambda)_{\mu}$ where $\mu = w_0\sigma\Lambda$. Quite clearly $v$ generates $V_\mathbb F^\sigma(\Lambda)$. It follows that $v$ is an extremal weight vector and, hence, satisfies the relations
\begin{equation} \label{eq:xwv?}
(x^\pm_\gamma)^{(k)}v = 0 \quad\text{for all}\quad k>\max\{0,\mp\mu(h_\gamma)\}
\end{equation}
and all positive real roots $\gamma\in\hat R^+$. In particular, taking $\gamma=\alpha+s\delta$ with $\alpha\in R^+$ and $s\ge 0$, it follows that
\begin{equation*}
-\mu(h_\gamma) = -\lambda(h_\alpha)-\ell r_\alpha^\vee s\le 0,
\end{equation*}
showing that $(x_{\alpha,s}^+)^{(k)}v=0$ for all $k>0$. Similarly, taking $\gamma=-\alpha+s\delta$, we get
\begin{equation*}
-\mu(h_\gamma) = \lambda(h_\alpha)-\ell r_\alpha^\vee s,
\end{equation*}
which shows that $v$ satisfies the relations determined by \eqref{e:FLrel}. It remains to show that $U_\bb F(\lie h[t]_+)^0v=0$. This can be proved as in \cite[Lemme 26]{mathieu89}. Alternatively, this can also be shown by proving that there exists a surjective map from $D(\ell,\lambda^*)$ to the pull-back of $V_\mathbb F^\sigma(\Lambda)$ by the automorphism $\psi$ defined in Subsection \ref{ss:autom}, similarly to what we do in the next paragraph, and then comparing weights (one uses a vector as in Lemma \ref{l:jpmrel} for proving the existence of such a map). It now suffices to show that $\dim(D_\mathbb F(\ell,\lambda))\le\dim(V_\mathbb F^\sigma(\Lambda))$.

Let this time $v\in D_\mathbb F(\ell,\lambda^*)_{\lambda^*}\setminus\{0\}$, let $W$ be the pull-back of $D_\mathbb F(\ell,\lambda^*)$ by $\psi$, and $w$ denote $v$ when regarded as an element of $W$.
Since $U_\bb F(\lie n^+[t])^0v=0$, and since \eqref{e:n-inv} implies that $\psi(U_\bb F(\lie n^-[t])^0)=U_\bb F(\lie n^+[t])^0$, it follows that $U_\bb F(\lie n^-[t])^0w=0$. Also, $\psi$ restricts to an automorphism of $U_\bb F(\lie h[t]_+)$ and, hence, $U_\bb F(\lie h[t]_+)^0w=0$. Since $hv=\lambda^*(h)v$ for all $h\in U_\bb F(\lie h)$, \eqref{e:inv+-w} implies that $hw=w_0\lambda(h)w$ for all $h\in U_\bb F(\lie h)$. Finally, the defining relations of $v$ and \eqref{e:inv+-def} imply that
\[
(x_{\alpha,s}^+)^{(k)}w=(x_{\alpha,s}^-)^{(k)}v=0 \quad\text{for all} \quad \alpha\in R^+,\ s\ge 0,\ k>\max\{0,\lambda^*(h_\alpha)-s\ell r^\vee_\alpha\}.
\]
Thus $w$ satisfies all the defining relations of $V_\mathbb F^\sigma(\Lambda)$ given in Lemma \ref{l:jpmrel}. Hence, $W$ is a quotient of $V_\mathbb F^\sigma(\Lambda)$ and, therefore, $\dim(W)\le\dim(V_\mathbb F^\sigma(\Lambda))$. Since $\dim(D_\mathbb F(\ell,\lambda^*))=\dim(D_\mathbb F(\ell,\lambda))$ by Lemma \ref{l:Dem*}, we are done.
\end{proof}

The next corollary is now immediate.

\begin{cor}\label{c:jpmrel}
$D_\bb F(\ell,\lambda)$ is isomorphic to the quotient of $U_\bb F(\lie g[t])$ by the left ideal $I_\bb F^-(\ell,\lambda)$ generated by $h - w_0\lambda(h), h\in U_\bb F(\lie h)$, $U_\mathbb F(\lie h [t]_+)^0, U_\mathbb F(\lie n^-[t])^0$, and
\begin{gather*}
(x_{\alpha,s}^+)^{(k)} \quad\text{for all} \quad \alpha\in R^+,\ s\ge 0,\ k>\max\{0,- w_0\lambda(h_\alpha)-s\ell r^\vee_\alpha\}.
\end{gather*}
\hfill\qedsymbol
\end{cor}

\begin{rem}\label{r:Demrels}
Observe that the difference between our first definition of $D_\mathbb F(\ell,\lambda)$ and the one given by Corollary \ref{c:jpmrel} lies on exchanging a ``highest-weight generator'' by a ``lowest-weight'' one. More precisely, let $v$ be as in Lemma \ref{l:jpmrel}. Then,  the isomorphism of Proposition \ref{p:Demrels} must send $v$ to a nonzero element in $D_\bb F(\ell,\lambda)_{w_0\lambda}$. In particular, if $w\in D_\bb F(\ell,\lambda)_{w_0\lambda}$, it satisfies the relations listed in Lemma \ref{l:jpmrel}.
The second part of our proof of Proposition \ref{p:Demrels} differs from the one given in \cite[Corollary 1]{foli:weyldem} in characteristic zero. It is claimed that a vector in $D_\bb F(\ell,\lambda)_{w_0\lambda}$ must satisfy several relations, including \eqref{e:MFLrel+}, without further justification. Proposition \ref{p:Demrels} implies that this is true, but we do not see how to deduce it so directly (even in characteristic zero) since we cannot use extremal-weight-vector theory to such vector as $D_\bb F(\ell,\lambda)$ is not, a priori, contained in an integrable module for the full affine hyperalgebra.
\end{rem}

\begin{cor}\label{c:naoi4.10}
Let $\lie g=\lie{sl}_2$ and consider the subalgebra $\lie a = \lie n^-[t]\oplus\lie h[t]\oplus \lie n^+[t]_+\subseteq\lie g[t]$. For $\ell,\lambda\in\bb Z_{\ge 0}$, let $I_\bb F'(\ell,\lambda)$ be the left ideal of $U_\bb F(\lie a)$ generated by the generators of $I_\bb F(\ell,\lambda)$ which lie in $U_\bb F(\lie a)$. Then, given $k,l,s\in\bb Z_{\ge 0}$ with $k>\max\{0,\lambda-s\ell\}$, we have
\begin{equation}\label{e:g.naoisl2}
(x_i^+)^{(l)}(x_{i,s}^-)^{(k)}\ \in\  U_\bb F(\lie a)U_\bb F(\lie n^+)^0\oplus I_\bb F'(\ell,\lambda)
\end{equation}
where $i$ is the unique element of $I$.
\end{cor}

\begin{proof}
The statement is a hyperalgebraic version of that of \cite[Lemma 4.10]{naoi:weyldem} and the proof follows similar general lines. Namely, by using the automorphism of $\lie g[t]$ determined by $x_{i,r}^\pm\mapsto x_{i,r}^\mp, i\in I,r\in\bb Z_{\ge 0}$, one observes that proving \eqref{e:g.naoisl2} is equivalent to proving
\begin{equation}\label{e:g.naoisl2-}
(x_i^-)^{(l)}(x_{i,s}^+)^{(k)}\ \in\  U_\bb F(\lie a^-)U_\bb F(\lie n^-)^0+I_\bb F''(\ell,\lambda) \quad\text{for all}\quad k,l,s\in \bb Z_{\ge 0}, k>\max\{0,\lambda-s\ell\},
\end{equation}
where $\lie a^- = \lie n^-[t]_+\oplus\lie h[t]\oplus \lie n^+[t]$ and $I_\bb F''(\ell,\lambda)$ is the left ideal of $U_\bb F(\lie a^-)$ generated by the generators of $I_\bb F^-(\ell,\lambda)$ given in Corollary \ref{c:jpmrel} which lie in $U_\bb F(\lie a^-)$. Since $\lie g[t]=\lie a^-\oplus\lie n^-$, the PBW Theorem implies that
$$U_\bb F(\lie g[t]) = U_\bb F(\lie a^-)U_\bb F(\lie n^-)^0\oplus U_\bb F(\lie a^-)$$
and, hence, $(x_i^-)^{(l)}(x_{i,s}^+)^{(k)} = u+u'$ with $u\in U_\bb F(\lie a^-)U_\bb F(\lie n^-)^0$ and $u'\in  U_\bb F(\lie a^-)$.
Consider the Demazure module $D_\bb F(\ell,\lambda)$ and let $w\in D_\bb F(\ell,\lambda)_{-\lambda}\setminus\{0\}$. It follows from the proof of Proposition \ref{p:Demrels} that, if $k>\max\{0,\lambda-s\ell\}$, then
\begin{equation*}
u'w = \left((x_i^-)^{(l)}(x_{i,s}^+)^{(k)} - u\right)w = 0.
\end{equation*}
Since $\hlie b'^+ = \lie a^-\oplus \bb Cc$ and $\lie a^-$ is an ideal of $\hlie b'^+$, it follows from Lemma \ref{l:jpmrel} that $I_\bb F''(\ell,\lambda)$ is the annihilating ideal of $w$ inside $U_\bb F(\lie a)$ and, hence, $u'\in I_\bb F''(\ell,\lambda)$.
\end{proof}

\section{Joseph's Demazure flags}\label{s:joseph}

\subsection{Quantum groups}

Let $\bb C(q)$ be the field of rational functions on an indeterminate $q$. Let also $C=(c_{ij})_{i,j\in\hat I}$ be the Cartan matrix of $\hlie g$ and $d_i,i\in\hat I$, nonnegative relatively prime integers such that the matrix $DC$, with $D={\rm diag}(d_i)_{i\in I}$, is symmetric. Set $q_i=q^{d_i}$ and, for $m,n\in\mathbb Z, n\ge 0$, set $[m]_{q_i}= \frac{q_{i}^m - q_{i}^{-m}}{q_{i} - q_{i}^{-1}}, [n]_{q_{i}}!= [n]_{q_{i}} [n-1]_{q_{i}} \ldots [1]_{q_{i}}$,$\left[\begin{array}{c}m\\n\end{array}\right]_{q_{i}}= \dfrac{[m]_{q_{i}}[m-1]_{q_{i}}\ldots [m-n+1]_{q_{i}}}{[n]_{q_{i}}!}$. The quantum group $U_q(\hlie g')$ is a $\bb C(q)$-associative algebra (with $1$) with generators $x_i^{\pm}, k_i^{\pm 1},\,i\in\hat{I}$ subject to the following defining relations for all $i,j \in \hat{I}$:
\begin{gather*}
k_ik_i^{-1} = 1\quad
k_ik_j = k_jk_i\quad
k_ix_j^{\pm}k_i^{-1} = q_i^{\pm c_{ij}}x_j^{\pm}\quad
[x_i^+, x_j^-] = \delta_{ij}\frac{k_i -k_i^{-1}}{q_i - q_i^{-1}}\\
{\sum_{m=0}^{1-c_{ij}}(-1)^m\left[\begin{array}{c} 1- c_{ij}\\m\end{array}\right]_{q_i}(x^{\pm}_i)^{1-c_{ij}-m}x_j^{\pm}(x_i^{\pm})^m}= 0, \quad i\neq j.
\end{gather*}
Let $U_q(\hlie n^\pm)$ be the subalgebra generated by $x_i^\pm, i\in\hat I$ and $U_q(\hlie b^\pm)$ be the subalgebra generated by $U_q(\hlie n^\pm)$ together with $k_i^{\pm 1},i\in\hat I$.

We shall need an integral form of $U(\hlie g')$. Let $\mathbb Z_q = \bb Z[q,q^{-1}]$, denote by $U_{\mathbb Z_q}(\hlie n^\pm)$ the $\mathbb Z_q$-subalgebra of $U_q(\hlie n^\pm)$ generated by $\frac{(x_i^{\pm})^m}{[m]_{q_i}!}, i\in \hat I, m\ge 0$, and by $U_{\mathbb Z_q}(\hlie g')$ the $\mathbb Z_q$-subalgebra of $U_q(\hlie g')$ generated by $U_{\mathbb Z_q}(\hlie n^\pm)$ and $k_i,i\in\hat I$. Let also $U_{\mathbb Z_q}(\hlie b^\pm)=U_q(\hlie b^\pm)\cap U_{\mathbb Z_q}(\hlie g')$. Then, $U_{\mathbb Z_q}(\lie a)$, $\lie a=\hlie g',\hlie n^\pm, \hlie b^\pm$, is a free $\mathbb Z_q$-module such that the natural map $\mathbb C(q)\otimes_{\mathbb Z_q} U_{\mathbb Z_q}(\lie a) \to U_{q}(\lie a)$ is $\mathbb C(q)$-algebra isomorphism, i.e., $U_{\mathbb Z_q}(\lie a)$ is a $\mathbb Z_q$-form of $U_q(\lie a)$. Moreover, letting $\mathbb Z$ be a $\mathbb Z_q$-module where $q$ acts as $1$, there exists an epimorphism of $\mathbb Z$-algebras $\mathbb Z\otimes_{\mathbb Z_q} U_{\mathbb Z_q}(\lie a)\to U_\mathbb Z(\lie a)$, which is an isomorphism if $\lie a = \hlie n^\pm$ and whose kernel is the ideal generated by $k_i-1,i\in\hat I$, for $\lie a=\hlie g', \hlie b^\pm$.

Given $\Lambda\in\hat P^+$, let $V_q(\Lambda)$ be the simple (type 1) $U_q(\hlie g')$-module of highest weight $\Lambda$. Given a highest-weight vector $v\in V_q(\Lambda)$, set $V_{\mathbb Z_q}(\Lambda) = U_{\mathbb Z_q}(\hlie n^-)v$, which is a $\mathbb Z_q$-form of $V_q(\Lambda)$. Given $\sigma\in\wa$ and a nonzero vector $v\in V_q(\Lambda)$ of weight $\sigma\Lambda$, set $V_{\mathbb Z_q}^\sigma(\Lambda) = U_{\mathbb Z_q}(\hlie n^+)v$, which is a free $\mathbb Z_q$-module as well as a $U_{\mathbb Z_q}(\hlie b^+)$-module and $\mathbb C\otimes_{\mathbb Z_q}V_{\mathbb Z_q}^\sigma(\Lambda) \cong V_{\mathbb C}^\sigma(\Lambda)$. In particular,
\begin{equation} \label{def:VZ}
V_{\mathbb Z}^\sigma(\Lambda):=\mathbb Z\otimes_{\mathbb Z_q}V_{\mathbb Z_q}^\sigma(\Lambda)
\end{equation}
is an integral form of $V_{\mathbb C}^\sigma(\Lambda)$.

\subsection{Crystals}

A normal crystal associated to the root data of $\hlie g$ defined as a set $B$ equipped with maps $\tilde e_i,\tilde f_i:B\to B\sqcup\{0\},\epsilon_i,\varphi_i:B\to\bb Z,$ for each $i\in\hat{I}$, and $\wt:B\to\hp$ satisfying
\begin{enumerate}[(1)]
\item $\epsilon_i(b)=\max\{n:\tilde e_ib\neq0\}, \varphi_i(b)=\max\{n:\tilde f_ib\neq0\}$, for all $i\in \hat{I},b\in B$;
\item $\varphi_i(b)-\epsilon_i(b)=\wt(b)(h_i),$ for all $i\in \hat{I},b\in B$;
\item for $b,b'\in B$, $b'=\tilde e_ib$ if and only if $\tilde f_ib'=b$;
\item if $b\in B, i\in\hat I$ are such that $\tilde e_ib\neq 0$, then $\wt(\tilde e_ib)=\wt(b)+\alpha_i$.
\end{enumerate}
For convenience, we extend $\tilde e_i,\tilde f_i,,\epsilon_i,\varphi_i,\wt$ to $B\sqcup\{0\}$ by setting them to map $0$ to $0$. Denote by $\cal E$ the submonoid of the monoid of maps $B\sqcup\{0\}\to B\sqcup\{0\}$ generated by $\{\tilde e_i:i\in\hat{I}\}$, and similarly define $\cal F$. A normal crystal is said to be of highest weight $\Lambda\in\hp^+$ if there exists $b_\Lambda\in B$ satisfying
$$ \wt(b_\Lambda)=\Lambda, \qquad \cal Eb_\Lambda=\{0\}, \qquad\text{and}\qquad \cal Fb_\Lambda=B.$$
Given $B'\subset B$ and $\mu\in\hat P$, define $B'_\mu = \{b\in B': \wt(b)=\mu\}$ and define the character of $B'$ as $\ch(B') = \sum_{\mu \in \hp} \#B'_\mu e^\mu \in \bb Z[\hp]$.

Given crystals $B_1, B_2$, a morphism from $B_1$ to $B_2$ is a map $\psi:B_1\to B_2\sqcup\{0\}$ satisfying:
\begin{enumerate}[(1)]
\item if $\psi(b)\ne 0$, then $\wt(\psi(b))=\wt(b),\epsilon_i(\psi(b))=\epsilon_i(b),\varphi_i(\psi(b))=\varphi_i(b)$, for all $i\in\hat{I}$;
\item if $\tilde e_i b\neq 0$, then $\psi(\tilde e_ib)=\tilde e_i\psi(b)$;
\item if $\tilde f_i b\neq 0$, then $\psi(\tilde f_ib)=\tilde f_i\psi( b)$.
\end{enumerate}
The set $B_1\times B_2$ admits a structure of crystal denoted by $B_1\otimes B_2$ (cf. \cite[Section 2.4]{joseph03}).
There is, up to isomorphism, exactly one family $\{B(\Lambda):\Lambda\in\hp^+\}$ of normal highest weight crystals such that, for all $\lambda,\mu\in\hat P^+$, the crystal structure of $B(\lambda)\otimes B(\mu)$ induces a crystal structure on its subset $\cal F(b_\lambda\otimes b_\mu)$, the inclusion is a homomorphism of crystals, and $\cal F(b_\lambda\otimes b_\mu)\cong B(\lambda+\mu)$.

Given a crystal $B$ and $\sigma\in\wa$ with a fixed reduced expression $\sigma=s_{i_1}\dots s_{i_n}$, define
$$\cal E^\sigma = \{\tilde e_{i_1}^{m_1}\dots\tilde e_{i_n}^{m_n}: m_j\in\bb N\} \subset\cal E\quad\text{and}\quad \cal F^\sigma = \{\tilde f_{i_1}^{m_1}\dots\tilde f_{i_n}^{m_n}: m_j\in\bb N\} \subset\cal F.$$
If $B=B(\Lambda)$, $\Lambda\in\hp^+$ and $\sigma\in\wa$, define the Demazure subset $B^\sigma(\Lambda)=\cal F^\sigma b_\Lambda\subseteq B(\Lambda)$. Then $B^\sigma(\Lambda)$ is $\cal E$-stable, i.e., $\cal EB^\sigma(\Lambda)\subset B^\sigma(\Lambda)\sqcup\{0\}$. It was proved in \cite[Section 4.6]{joseph03} that $\ch(V_{\mathbb C}^\sigma(\Lambda)) = \ch(B^\sigma(\Lambda))$. This fact and the following theorem are the main results of \cite{joseph03} that we shall need.

\begin{thm}\label{t:jodemc}
Let $\Lambda,\mu\in\hp^+$. For any $\sigma\in\wa$, there exist a finite set $J$ and elements $\sigma_j\in\wa, b_j\in B^\sigma(\Lambda)$ for each $j\in J$, satisfying:
\begin{enumerate}[(1)]
\item $b_{\mu}\otimes B^\sigma(\Lambda)=\sqcup_{j\in J}B_j$ where $B_j:= \cal F^{\sigma_j} (b_{\mu}\otimes b_j)$;
\item $\cal E(b_{\mu}\otimes b_j)=\{0\}$;
\item $\ch(B_j) = \ch(B^{\sigma_j}(\nu_j))$, where $\nu_j=\mu+\wt(b_j)\in\hat P^+$.
\end{enumerate}
\end{thm}

\begin{rem}
The proof of Theorem \ref{t:jodemc} establishes an algorithm to find the set $J$ and the elements $\sigma_j,b_j$.
\end{rem}

\subsection{Globalizing}\label{joseph:global}

The theory of global basis of Kashiwara shows, in particular, that, for each $\Lambda\in\hat P^+$, there is a map $G:B(\Lambda)\to V_q(\Lambda)$ such that
\begin{equation} \label{eq:gl.basis}
V_{\mathbb Z_q}(\Lambda) = \opl_{b\in B(\Lambda)}^{} \mathbb Z_qG(b),
\end{equation}
the weight of $G(b)$ is $\wt(b)$ and $G(b_\Lambda)$ is a highest-weight vector of $V_q(\Lambda)$.

Fix $\Lambda,\mu\in\hat P^+,\sigma\in\wa$ and let $J, b_j, \sigma_j, \nu_j, j\in J$, be as in Theorem \ref{t:jodemc}. Let $b\in B(\Lambda)_{\sigma\Lambda}$ and set $V_{\bb Z_q}^\sigma(\Lambda) = U_{\bb Z_q}(\hlie n^+)G(b)$. Similarly, let $b'_j$ be the unique element of $B_j$ such that $\wt(b'_j)=\sigma_j\nu_j$. Choose a linear order on $J$ such that $\wt(b_j)<\wt(b_k)$ only if $j>k$. For $j\in J$, let $Y_j$ be the $\mathbb Z_q$-submodule of $V_q(\mu)\otimes V_q^\sigma(\Lambda)$ spanned by $G(b_\mu)\otimes G(b)$ with $b\in B_k, k\le j$, and set
\begin{equation}\label{e:defyj}
y_j = G(b_\mu)\otimes G(b'_j).
\end{equation}
Let also $Z_j = \sum_{k\le j} U_{\mathbb Z_q}(\hlie n^-)\left(G(b_\mu)\otimes G(b_k)\right)$. Since $J$ is linearly ordered and finite, say $\#J=n$, identify it with $\{1,\dots,n\}$. For convenience, set $Y_0=\{0\}$. Observe that $0=Y_0\subset Y_1\subset\dots\subset Y_k$ is a filtration of the $U_{\bb Z_q}(\hlie b^+)$-module $G(b_{\Lambda_0})\otimes V_{\mathbb Z_q}^\sigma(\Lambda)$. The following result was proved in \cite[Corollary 5.10]{joseph06}.

\begin{thm} \label{t:jodemf}
Suppose $\lie g$ is simply laced and $\mu(h_i)\le 1$ for all $i\in\hat I$. Then:
\begin{enumerate}[(a)]
\item The $\mathbb Z_q$-module $Y_j$ is $U_{\mathbb Z_q}(\hlie n^+)$-stable for all $j\in J$.
\item \label{t:jodemf.b}
For all $j\in J$, $Y_j/Y_{j-1}$ is isomorphic to $V_{\mathbb Z_q}^{\sigma_j}(\nu_j)$. In particular, $Y_j/Y_{j-1}$ is a free $\mathbb Z_q$-module.
\item For all $j\in J$, the image of $\{G(b_\mu)\otimes G(b): b\in B_j\}$ in $Y_j/Y_{j-1}$ is a $\mathbb Z_q$-basis of $Y_j/Y_{j-1}$.
\item \label{t:jodemf.d} For each $j\in J$, $Z_j$ is $U_{\bb Z_q}(\hlie g')$-stable and $Y_j=Z_j\cap \left(G(b_\mu)\otimes V_{\bb Z_q}^\sigma(\Lambda)\right)$.
\end{enumerate}
\end{thm}

\begin{rem}
The above theorem was proved in \cite{joseph06} for any simply-laced symmetric Kac-Moody Lie algebra. However, as pointed out in \cite[Remark 4.15]{naoi:weyldem}, the proof also holds for $\hlie{sl}_2$.
\end{rem}

It follows from Theorem \ref{t:jodemf} and the fact that $G(b_\mu)$ is a highest-weight vector of $V_q(\Lambda)$ \eqref{eq:gl.basis} that
\begin{equation}\label{e:jodemf}
Y_j = \sum_{k\le j} U_{\bb Z_q}(\hlie n^+)y_j.
\end{equation}

\subsection{Simply laced Demazure flags}

Given $\ell\geq 0,\lambda\in P^+,m\in \bb Z$, let $D_\bb F(\ell,\lambda,m)=\tau_{m}(D_\bb F(\ell,\lambda))$ and $D_\bb Z(\ell,\lambda,m)=\tau_{m}(D_\bb Z(\ell,\lambda))$.

\begin{thm}\label{t:filtration}
Suppose $\lie g$ is simply laced, let $\mu\in P^+$ and $\ell'>\ell\ge 0$. Then, there exist $k> 0, \mu_1,\dots,\mu_k\in P^+, m_1,\dots,m_k\in\mathbb Z_{\ge 0}$, and a filtration of $U_\mathbb Z(\lie g[t])$-modules $0=D_0\subseteq D_1\subseteq\cdots\subseteq D_k= D_\mathbb Z(\ell,\mu)$ such that $D_j$ and $D_j/D_{j-1}$ are free $\mathbb Z$-modules for all $j=1,\dots,k$, and $D_j/D_{j-1}\cong D_\mathbb Z(\ell',\mu_j,m_j)$. Moreover, for all $j\in J$, there exists $\v_j\in D_j$ such that \vspace{-10pt}
\begin{enumerate}[(i)]
\item The image of $\v_j$ in $D_j/D_{j-1}$ satisfies the defining relations of $D_\mathbb Z(\ell',\mu_j,m_j)$;
\item $D_j = \sum_{k\le j} U_\bb Z(\lie n^-[t])\v_k$.
\end{enumerate}
\end{thm}

\begin{proof}
The proof follows closely that of \cite[Corollary 4.16]{naoi:weyldem}. First notice that it is enough to prove the theorem for $\ell'=\ell+1$. Then let $\Lambda\in\widehat{P}^+$ and $w\in\wa$ be such that $w\Lambda=\ell\Lambda_0+w_0\mu$, and let $V_{\mathbb Z_q}^w(\Lambda)=U_{\mathbb Z_q}(\hlie n^+)G(b)$ where $b\in B(\Lambda)_{w\Lambda}$.

From subsection \ref{joseph:global}, we know that the $U_{\bb Z_q}(\hlie b^+)$-submodule $G(b_{\Lambda_0})\otimes V_{\mathbb Z_q}^w(\Lambda) \subseteq V_q(\Lambda_0)\otimes V_q(\Lambda)$ admits a filtration $0=Y_0\subset Y_1\subset\dots\subset Y_k$. For each $j =1, \ldots, k$, let $D_j=\bb Z\otimes_{\bb Z_q}Y_j$, and observe that
\[
D_k = \bb Z \otimes_{\bb Z_q} \left(G(b_{\Lambda_0})\otimes_{\bb Z_q}V_{\bb Z_q}^w(\Lambda)\right)
\cong \left( \bb Z \otimes_{\bb Z_q} G(b_{\Lambda_0}) \right) \otimes_\bb Z \left( \bb Z \otimes_{\bb Z_q} V_{\bb Z_q}^w (\Lambda) \right)
\cong \bb Z_{\Lambda_0}\otimes_\bb ZD_{\bb Z}(\ell,\mu),
\]
where $\bb Z_{\Lambda_0}$ is a $U_\bb Z (\hlie b^+)$-module on which $U_\bb Z(\hlie n^+)^0$ and $U_\mathbb Z(\lie g)^0$ act trivially and $U_\bb Z(\hlie h)$ acts by $\Lambda_0$. Moreover, as a \bb Z-module it is free of rank $1$. Thus $D_k$ is isomorphic to $D_\mathbb Z(\ell,\mu)$ as a $U_\mathbb Z(\lie g[t])$-module. It follows from Theorem \ref{t:jodemf} \eqref{t:jodemf.d} that $D_j$ is a $U_\mathbb Z(\lie g[t])$-module for all $j = 1, \ldots, k$ and, hence, so is $D_j/D_{j-1}$. So we have a filtration of $U_\mathbb Z(\lie g[t])$-modules $0=D_0\subset D_1\subset\dots\subset D_k = D_\bb Z(\ell, \mu)$.

By Theorem \ref{t:jodemf} \eqref{t:jodemf.b}, $Y_j/Y_{j-1}\cong V_{\mathbb Z_q}^{\sigma_j}(\nu_j)$ for some $\sigma_j\in\wa$, $\nu_j\in\widehat{P}^+$. By \eqref{def:VZ} $D_j/D_{j-1}\cong V_{\mathbb Z}^{\sigma_j}(\nu_j)$. Thus $D_j/D_{j-1}$ is isomorphic to $D_\mathbb Z(\ell_j,\mu_j,m_j)$ for some $\mu_j\in P^+, m_j\in\mathbb Z$ and $\ell_j=\nu_j(c)$ (cf. \eqref{e:thelevel}). Since all the weights of $V_q(\Lambda_0)\otimes V_q(\Lambda)$ are of the form $\Lambda+\Lambda_0-\eta$ for some $\eta\in\hat Q^+$, and $\alpha_i(c)=0$ for all $i\in\hat I$, it follows that $\ell_j= \ell+1$ for all $j$.

Keep denoting the image of $y_j$ in $D_j$ by $y_j$ (cf. \eqref{e:defyj}). It follows that $D_j = \sum_{k\le j} U_{\bb Z}(\hlie n^+)y_j$ by \eqref{e:jodemf}. As in Remark \ref{r:Demrels}, we now replace the ``lowest-weight'' generators $y_j$ by ``highest-weight generators''. Thus, let $b_j''$ be the unique element of $B_j$ such that $\wt(b_j'')= w_0\sigma_j\nu_j = (\ell+1)\Lambda_0+\mu_j+m_j\delta$ and let $\v_j$ be defined similarly to $y_j$ by replacing $b'_j$ by $b_j''$.
\end{proof}

The next corollary is now immediate.

\begin{cor} \label{c:filtration}
Let $\lie g, \mu,\ell',\ell,k,\mu_j,j=1,\dots,k$, be as in Theorem \ref{t:filtration}. Then, there exists a filtration of $U_\mathbb F(\lie g[t])$-modules $0=D_0\subseteq D_1\subseteq\cdots\subseteq D_k= D_\mathbb F(\ell,\mu)$, such that $D_j/D_{j-1}\cong D_\mathbb F(\ell',\mu_j)$ for all $j=1,\dots,k$. \hfill\qed
\end{cor}

\section{Proof of Theorem \ref{t:isos}}\label{s:proof}

\subsection{The isomorphism between Demazure and graded local Weyl modules}\label{ss:isodg}

Recall that, for $\lie g=\lie{sl}_2$, a characteristic-free proof of Theorem \ref{t:isos}\eqref{t:isosdg} was given in \cite{jm:weyl}. Thus, assume $\lie g$ is simply laced of rank higher than 1 and recall from  Remark \ref{ref:UF} that $D_\bb F (1, \lambda)$ is a quotient of $W_{\bb F}^{c} (\lambda)$. To prove the converse, let $w$ be the image of $1$ in $W_\mathbb F^c(\lambda)$. In order to show that $W_{\bb F}^{c} (\lambda)$ is a quotient of $D_\bb F (1, \lambda)$, it remains to prove that
\begin{equation}{\label{e:isodg}}
(x_{\alpha, s}^{-})^{(k)} w= 0 \quad\text{for all}\quad \alpha \in R^+,\ s > 0, \ k > \max\{0,\lambda(h_\alpha) - s\}.
\end{equation}
Given $\alpha \in R^+$, consider the subalgebra $U_\bb F (\lie{sl}_\alpha [t])$ (see Subsection \ref{ss:saofr12}) and let $W_\alpha$ be the $U_\bb F (\lie{sl}_\alpha [t])$-submodule of $W_{\bb F}^{c} (\lambda)$ generated by $w$. Clearly, $W_\alpha$ is a quotient of the graded local Weyl module for $U_\bb F (\lie{sl}_\alpha [t])$ with highest weight $\lambda(h_\alpha)$, where we have identified the weight lattice of $\lie{sl}_2$ with $\bb Z$ as usual. Since we already know that the theorem holds for $\lie{sl}_2$, it follows that $w$ must satisfy the same relations as the generator of the corresponding Demazure module for $U_\bb F (\lie{sl}_\alpha [t])$. In particular, \eqref{e:isodg} holds and so does Theorem \ref{t:isos}\eqref{t:isosdg}.

\subsection{A smaller set of relations for non simply laced Demazure modules}\label{ss:relnsl}

In this subsection we assume $\lie g$ is not simply laced and prove the following analogue of \cite[Proposition 4.1]{naoi:weyldem}.

\begin{prop} \label{p:lessreldem}
For all $\lambda \in P^+$, $D_\mathbb F(1,\lambda)$ is isomorphic to the quotient of $U_\mathbb F(\lie g[t])$ by the left ideal $I_\bb F(\lambda)$ generated by
\begin{gather} \label{eq:lessreldem}
U_\mathbb F(\lie n^+[t])^0, \quad
U_\mathbb F(\lie h [t]_+)^0, \quad
h- \lambda(h), \quad (x_i^-)^{(k)}, \quad
(x_{\alpha,s}^-)^{(\ell)}
\end{gather}
for all $h\in U_\mathbb F(\lie h), i\in I\setminus I_\sh, \alpha\in R^+_\sh, s\ge 0, k>\lambda(h_i), \ell> \max \{0, \lambda(h_\alpha) - s r^\vee \}$.
\end{prop}

Let $w\in D_\mathbb F(1,\lambda)_\lambda\setminus\{0\}$ and let $V$ be the $U_\mathbb F(\lie g[t])$-module generated by a vector $v$ with defining relations given by \eqref{eq:lessreldem}. In particular, there exists a unique epimorphism $V\to D_\mathbb F(1,\lambda)$ mapping $v$ to $w$. To prove the converse, observe first that, since $(x_i^-)^{(k)}v=0$ for all $i\in I,k>\lambda(h_i)$, Lemma \ref{l:weyl-relhr} implies that $(x_\alpha^-)^{(k)}v=0$ for all $\alpha\in R^+, k>\lambda(h_\alpha)$. In particular, $V$ is a quotient of $W^c_\mathbb F(\lambda)$ and, hence, it is finite-dimensional. It remains to show that
\begin{equation*}\label{e:lessreldem}
(x_{\alpha,s}^-)^{(k)}v = 0\quad\text{for all}\quad \alpha\in R^+\setminus R_\sh^+,\ s>0,\ k> \max \{0, \lambda(h_\alpha) - s r^\vee_\alpha \}.
\end{equation*}
These relations will follow from the next few lemmas.

\begin{lem}{\label{lem:long}}
Let $V$ be a finite-dimensional $U_\bb F (\lie g [t])$-module, let $\lambda\in P^+$, and suppose $v \in V_\lambda$ satisfies $U_\bb F (\lie n^+ [t])^0 v = U_\bb F (\lie h [t]_+)^0 v = 0$. If $\alpha \in R^+$ is long, then $( x _{\alpha, s}^{-} )^{(k)} v = 0$ for all $s\geq 0,k>\max \{0, \lambda(h_\alpha) - s \}$.
\end{lem}

\begin{proof}
Consider the subalgebra $U_\bb F (\lie{sl}_\alpha[t])$ (see Subsection \ref{ss:saofr12}). By Theorem \ref{t:fdrhca} \eqref{t:gweylu}, the submodule $W=U_\bb F (\lie{sl}_\alpha[t]) v$ is a quotient of the local graded Weyl module for $U_\bb F (\lie{sl}_\alpha [t])$ with highest weight $\lambda (h_\alpha)$. Theorem \ref{t:isos} \eqref{t:isosdg} implies that $W\cong D_\mathbb F^\alpha(1,\lambda(h_\alpha))$ where the latter is the corresponding Demazure module for $U_\bb F (\lie{sl}_\alpha[t])$. In particular, $v$ satisfies the relations \eqref{e:FLrel}.
\end{proof}

\begin{lem}{\label{lem:bcf.naoi}}
Assume $\lie g$ is not of type $G_2$. Let $V$ be a finite-dimensional $U_\bb F (\lie g [t])$-module, $\lambda\in P^+$, and suppose $v \in V_\lambda$ satisfies $U_\bb F (\lie n^+ [t])^0 v = U_\bb F (\lie h [t]_+)^0 v = 0$ and $(x_{\alpha,s}^{-})^{(k)}v=0$ for all $\alpha \in R_\sh^+,k> \max\{0,\lambda(h_\alpha)-2s\}$. Then, for every short root $\gamma$, we have $(x_{\gamma, s}^{-})^{(k)} v = 0$ for all $s \geq 0,k > \max \{ 0, \lambda(h_\gamma) - 2s \}$.
\end{lem}

\begin{proof}
The proof will proceed by induction on ${\rm{ht}} (\gamma)$. If ${\rm{ht}} (\gamma) = 1$, then $\gamma$ is simple and, hence, $\gamma\in R_\sh^+$. Thus, suppose ${\rm{ht}}(\gamma)>1$ and that $\gamma\notin R_\sh^+$. By \cite[Lemma 4.6]{naoi:weyldem}, there exist $\alpha, \beta \in R^+$ such that $\gamma = \alpha + \beta$ with $\alpha$ long and $\beta$ short. Notice that $\{\alpha,\beta\}$ form a simple system of a rank-two root subsystem. In particular, $h_\gamma = 2 h_{\alpha} + h_\beta$ and, hence $\lambda(h_\gamma) = 2\lambda(h_\alpha)+\lambda(h_\beta)$.

Fix $s\ge 0$ and suppose first that $\lambda(h_\gamma)-2s\ge 0$. In this case, we can choose $a,b\in\mathbb Z_{\ge 0}$ such that
\begin{equation*}
a+b=s, \qquad \lambda(h_\alpha)-a\ge 0, \qquad\text{and}\qquad \lambda(h_\beta)-2b\ge 0.
\end{equation*}
Indeed, $b=\max\{0,s-\lambda(h_\alpha)\}$ and $a=s-b$ satisfy these conditions. Then, Lemma \ref{lem:long} implies that $(x_{\alpha,a}^{-})^{(k)} v = 0$ for all $k>\lambda(h_\alpha)-a$, while the induction hypothesis implies that $(x_{\beta,b}^{-})^{(k)}v=0$ for all $k >\lambda(h_\beta) - 2b$. Applying Lemma \ref{l:weyl-relhr} to the subalgebra $U_\mathbb F(\lie g_{\alpha,\beta}^{a,b})$ (cf. Subsection \ref{ss:saofr12}), it follows that $(x_{\gamma, s}^{-})^{(k)} v= 0$ for all $k > 2(\lambda(h_\alpha) - a) + (\lambda(h_\beta) - 2b) = \lambda(h_\gamma)-2s$.

Now suppose $\lambda(h_\gamma)-2s\le 0$ and notice that this implies $s-\lambda(h_\alpha) = s- \frac{1}{2}\left(\lambda(h_\gamma)-\lambda(h_\beta)\right)\ge \frac{\lambda(h_\beta)}{2}\ge 0$. We need to show that $(x_{\gamma, s}^{-})^{(k)} v= 0$ for all $k > 0$. Letting $a=\lambda(h_\alpha)$ and $b=s-\lambda(h_\alpha)$, we have
\begin{equation*}
a+b=s, \qquad \lambda(h_\alpha)-a\le 0, \qquad\text{and}\qquad \lambda(h_\beta)-2b\le 0.
\end{equation*}
Then, Lemma \ref{lem:long} implies that $(x_{\alpha,a}^{-})^{(k)} v = 0$ for all $k>0$, while the induction hypothesis implies that $(x_{\beta,b}^{-})^{(k)}v=0$  for all $k >0$. The result follows from an application of Lemma \ref{l:weyl-relhr} as before.
\end{proof}

It remains to prove an analogue of Lemma \ref{lem:bcf.naoi} for $\lie g$ of type $G_2$. This is much more technically complicated and will require that we assume that characteristic $\mathbb F$ is at least 5. For the remainder of this subsection we assume $\lie g$ is of type $G_2$ and set $I=\{1,2\}$ so that $\alpha_1$ is short.
Given $\gamma = s\alpha_1+l\alpha_2\in R^+$, set $s_\gamma=s$. Set also
\begin{equation*}
\lie n^+[t]_> = \opl_{\gamma\in R^+}^{}\ \opl_{s\ge s_\gamma}^{} \mathbb Cx_{\gamma,s}^+,\quad \lie n^+[t]_< = \opl_{\gamma\in R^+}^{}\ \opl_{s=0}^{s_\gamma-1} \mathbb Cx_{\gamma,s}^+, \quad \lie a = \lie n^-[t]\oplus\lie h[t]\oplus \lie n^+[t]_>,
\end{equation*}
and observe that $\lie n^+[t]_>$ and $\lie n^+[t]_<$ are subalgebras of $\lie n^+[t]$ such that $\lie n^+[t] = \lie n^+[t]_>\oplus\lie n^+[t]_<$. The hyperalgebras $U_\bb F(\lie n^+[t]_>), U_\bb F(\lie n^+[t]_<)$, and $U_\bb F(\lie a)$ are then defined in the usual way (see Subsection \ref{ss:Zforms}) and the PBW theorem implies that
\begin{equation}\label{e:n<>}
U_\bb F(\lie n^+[t]) = U_\bb F(\lie n^+[t]_>) \oplus U_\bb F(\lie n^+[t])U_\bb F(\lie n^+[t]_<)^0.
\end{equation}
We now prove a version of \cite[Lemma 4.11]{naoi:weyldem} for hyperalgebras.

\begin{lem}\label{l:naoi411}
Given $\lambda\in P^+$, let $I'_\bb F(\lambda)$ be the left ideal of $U_\bb F(\lie a)$ generated by the generators of $I_\bb F(\lambda)$ described in \eqref{eq:lessreldem} which lie in $U_\bb F(\lie a)$. Then,
\begin{equation*}
I_\bb F(\lambda)\subseteq I'_\bb F(\lambda)\oplus U_\bb F(\lie a)U_\bb F(\lie n^+[t]_<)^0.
\end{equation*}
\end{lem}

\begin{proof}
Recall that $I_\bb F(\lambda)$ is the left ideal of $U_\bb F(\lie g[t])$ generated by the set $\cal I$ whose elements are the elements in $U_\bb F(\lie n^+[t])^0$, $U_\bb F(\lie h[t]_+)^0$, together with the elements
\begin{equation*}
\tbinom{h_i}{l}-\tbinom{\lambda(h_i)}{l},\quad (x_2^-)^{(m)}, \quad (x_{1,s}^-)^{(k)} \quad\text{for}\quad i\in I,\ k,l,m,s\in\bb Z_{\ge 0},\ m>\lambda(h_2),\ k>\max\{0,\lambda(h_1)-3s\}.
\end{equation*}
To simplify notation, set $U_< = U_\bb F(\lie n^+[t]_<)$ and $J=I_\bb F'(\lambda)\oplus U_\bb F(\lie a)U_\bb F(\lie n^+[t]_<)^0$. Observe that $U_\bb F(\lie a)J\subseteq J$. Therefore, since $U_\bb F(\lie g[t])=U_\bb F(\lie a)U_<$ by \eqref{e:n<>} and we clearly have $\cal I\subseteq J$, it suffices to show that
\begin{equation*}
U_<^0\ \cal I\subseteq J.
\end{equation*}
We will decompose the set \cal I into parts, and prove the inclusion for each part. Namely, we first decompose $\cal I$ into $\left( \cal I \cap U_\bb F(\lie n^+[t])U_\bb F(\lie h[t])\right) \sqcup \left(\cal I\cap U_\bb F(\lie n^-[t])\right)$, and then we further decompose $\cal I\cap U_\bb F(\lie n^-[t])$ into $\{ (x_2^-)^{(m)}: m>\lambda(h_2)\} \sqcup \{(x_{1,s}^-)^{(k)} : s\in\bb Z_{\ge 0},k>\max\{0,\lambda(h_1)-3s\}\}$.

Since $\lie h[t]\oplus \lie n^+[t]$ is a subalgebra of $\lie g[t]$,$U_\bb F(\lie n^+[t])U_\bb F(\lie h[t]) = U_\bb F(\lie h[t])U_\bb F(\lie n^+[t])$ by PBW Theorem, and, therefore,
\begin{equation*}
U_<^0\left(\cal I\cap U_\bb F(\lie n^+[t])U_\bb F(\lie h[t])\right)\subseteq U_\bb F(\lie h[t])U_\bb F(\lie n^+[t]).
\end{equation*}
Now, by \eqref{e:n<>}, $U_\bb F(\lie h[t])U_\bb F(\lie n^+[t])\subseteq J$, showing that $U_<^0\left(\cal I\cap U_\bb F(\lie n^+[t])U_\bb F(\lie h[t])\right)\subseteq J$. In particular, we have shown that
\begin{equation}\label{e:naoiI1}
U_\bb F(\lie g[t])U_\bb F(\lie n^+[t])^0 \subseteq J.
\end{equation}
It remains to show that
\begin{equation*}
U_<^0\left(\cal I\cap U_\bb F(\lie n^-[t])\right)\subseteq J.
\end{equation*}

We begin by proving that $U_<^0U_\bb F(\lie n_2^-) \subseteq J$, where $\lie n_2^-$ is the subalgebra spanned by $x_2^-$. Consider the natural $Q$-grading on $U_\bb F(\lie g[t])$ and, for $\eta\in Q$, let $U_\bb F(\lie g[t])_\eta$ denote the corresponding graded piece. Observe that $\lie m_2:=\lie n^+[t]_<\oplus \lie n_2^-$ is a subalgebra of $\lie g[t]$ and that
\begin{equation*}
U_<^0 U_\bb F(\lie n^-_2)\subseteq \opl_{\eta}^{} U_\bb F(\lie m_2)_\eta,
\end{equation*}
where the sum runs over $\bb Z_{>0}\alpha_1\oplus\bb Z\alpha_2$. Together with the PBW Theorem, this implies that
\begin{equation*}
U_<^0 U_\bb F(\lie n^-_2)\subseteq U_\bb F(\lie n^-_2)U_<^0\subseteq U_\bb F(\lie a)U_<^0\subseteq J.
\end{equation*}

Finally, we show that $U_<^0\ \cal I_1\subseteq J$, where $\cal I_1=\left(\cal I\cap U_\bb F(\lie n^-_1[t])\right)$ and $\lie n_1^-$ is the subalgebra spanned by $x_1^-$. Consider
\begin{equation*}
\lie n^+[t]_<^1 = \opl_{\gamma\in R^+\setminus\{\alpha_1\}}^{}\ \opl_{s=0}^{s_\gamma-1} \mathbb Cx_{\gamma,s}^+,
\end{equation*}
which is a subalgebra of $\lie n^+[t]_<$ such that $\lie n^+[t]_< = \lie n_1^+\oplus \lie n^+[t]_<^1$, where $\lie n_1^+=\bb Cx_1^+$. Moreover, $\lie m_1:=\lie n^+[t]_<^1\oplus\lie n_1^-[t]$ is a subalgebra of $\lie g[t]$ such that $U(\lie m_1)_\eta \ne 0$ only if $\eta\in \bb Z\alpha_1\oplus \bb Z_{\ge 0}\alpha_2$ and $U(\lie m_1)_0=\bb C$. This implies that
\begin{equation*}
U_\bb F(\lie n^+[t]_<^1)^0 U_\bb F(\lie n^-_1[t])= U_\bb F(\lie n^-_1[t])U_\bb F(\lie n^+[t]_<^1)^0.
\end{equation*}
Since $U_<^0 = U_\bb F(\lie n^+_1)U_\bb F(\lie n^+[t]_<^1)^0 \oplus U_\bb F(\lie n^+_1)^0$, we get
\begin{align*}
U_<^0\ \cal I_1
& \subseteq \left(U_\bb F(\lie n^+_1)U_\bb F(\lie n^+[t]_<^1)^0 + U_\bb F(\lie n^+_1)^0\right)\cal I_1\\
& \subseteq U_\bb F(\lie n^+_1)U_\bb F(\lie n^-_1[t])U_\bb F(\lie n^+[t]_<^1)^0  + U_\bb F(\lie n^+_1)^0\cal I_1\\
&\subseteq U_\bb F(\lie g[t])U_\bb F(\lie n^+[t])^0  + U_\bb F(\lie n^+_1)^0\cal I_1.
\end{align*}
The first summand in the last line is in $J$ by \eqref{e:naoiI1} while the second one is in $J$ by Corollary \ref{c:naoi4.10} (with $\lambda=\lambda(h_1)$ and $\ell=3$) together with \eqref{e:naoiI1}.
\end{proof}

Set $\lie h_i=\mathbb Ch_i, i\in I$, and $\lie b = \lie n^-[t]\oplus\lie h[t]_+\oplus \lie h_2\oplus\lie n^+[t]_>$. Observe that $\lie b$ is an ideal of $\lie a$ such that $\lie a=\lie b\oplus\lie h_1$. One easily checks that there exists a unique Lie algebra homomorphism $\phi:\lie b\to \lie g[t]$ such that
\begin{equation*}
\phi(x_{\gamma,r}^\pm) = x_{\gamma,r\mp s_\gamma}^\pm \quad\text{for all}\quad \gamma\in R^+.
\end{equation*}
Moreover, $\phi$ is the identity on $\lie h[t]_+ + \lie{sl}_{\alpha_2}$. Also, $\phi$ can be extended to a Lie algebra map $\lie a\to U(\lie g[t])$ by setting $\phi(h_1) = h_1-3$ (cf. \cite[Section 4.2]{naoi:weyldem}). Proceeding as in Section \ref{ss:autom}, one sees that $\phi$ induces an algebra homomorphism $U_\bb F(\lie a)\to U_\bb F(\lie g[t])$ also denoted by $\phi$.

We are ready to prove the analogue of Lemma \ref{lem:bcf.naoi} for type $G_2$.

\begin{lem}{\label{lem:g.naoi}}
Let $V$ be a finite-dimensional $U_\bb F (\lie g [t])$-module, $\lambda\in P^+$, and suppose $v \in V_\lambda$ satisfies $U_\bb F (\lie n^+ [t])^0 v = U_\bb F (\lie h [t]_+)^0 v = 0$ and $(x_{1,s}^{-})^{(k)}v=0$ for all $k> \max\{0,\lambda(h_1)-3s\}$. Then, for every short root $\gamma$, we have $(x_{\gamma, s}^{-})^{(k)} v = 0$ for all $s \geq 0,k > \max \{ 0, \lambda(h_\gamma) - 3s \}$.
\end{lem}

\begin{proof}
Notice that the conclusion of the lemma is equivalent to
\begin{equation*}
(x_{\gamma,s}^-)^{(k)}\in I_\bb F(\lambda) \quad\text{for all}\quad s\geq0,\ \ k>\max\{0,\lambda(h_\gamma)-3s\}
\end{equation*}
for every short root $\gamma$. Recall that the short roots in $R^+$ are $\alpha_1$, $\alpha:=\alpha_1+\alpha_2$ and $\vartheta:=2\alpha_1+\alpha_2$ while the long roots are $\alpha_2$, $\beta:=3\alpha_1+\alpha_2$ and $\theta:=3\alpha_1+2\alpha_2$. 
For $\gamma=\alpha$, we have $h_\gamma=h_1+3h_2$ and the proof is similar to that of Lemma \ref{lem:bcf.naoi} (the details can be found in \cite{macedo:PhD}). We shall use that the lemma holds for $\gamma=\alpha$ in the remainder of the proof. It remains to show that the lemma holds with $\gamma=\vartheta$. Notice that $h_\vartheta=2h_1+3h_2$ and, thus, we want to prove that
\begin{equation}\label{e:g.naoi}
(x_{\vartheta,s}^-)^{(k)}\in I_\bb F(\lambda) \quad\text{for all}\quad s\geq0,\ \ k>\max\{0,2\lambda(h_1)+3\lambda(h_2)-3s\}.
\end{equation}
We prove \eqref{e:g.naoi} by induction on $\lambda(h_1)$. Following \cite{naoi:weyldem}, we prove the cases $\lambda(h_1)\in\{0,1,2\}$ and then we show that \eqref{e:g.naoi} for $\lambda-3\omega_1$ in place of $\lambda$ implies it for $\lambda$.
To shorten notation, set $a=\lambda(h_1),b=\lambda(h_2)$.

\noindent {\bf 1)} Assume $a=0$. Since $\alpha_1\in R_\sh^+$, it follows that $(x_1^{-})^{(k)}v=0$ for all $k>0$. By Lemma \ref{lem:long}, we have $(x_{2,s}^{-})^{(k)}v=0$ for all $k>\max\{0,b-s\}$. Applying Lemma \ref{l:weyl-relhr} to the subalgebra $U_\mathbb F(\lie g_{\alpha_1,\alpha_2}^{0,s})$, it follows that $(x_{\vartheta, s}^{-})^{(k)}v=0$ for all $k>3\max\{0,b-s\}=\max\{0,2a+3b-3s\}$ as desired.

\noindent {\bf 2)} Assume $a=1$. This time we have $(x_1^{-})^{(k)}v=0$ for all $k>1$. We split in 3 subcases.

\noindent {\bf 2.1)} Suppose $b>s-1$ and notice $2a+3b-3s>0$. Lemma \ref{lem:long} implies $(x_{2,s}^{-})^{(k)}v=0$ for all $k>\max\{0,b-s\}=b-s$. Applying Lemma \ref{l:weyl-relhr} to the subalgebra $U_\mathbb F(\lie g_{\alpha_1,\alpha_2}^{0,s})$, it follows that $(x_{\vartheta, s}^{-})^{(k)}v=0$ for all $k>2+3(b-s)=2a+3b-3s$.

\noindent {\bf 2.2)} Suppose $b=s-1$ in which case $2a+3b-3s<0$. Notice that $h_\beta = h_1+h_2$ and, hence, $\lambda(h_\beta)= a+b=s$. Lemma \ref{lem:long} then implies that $(x_{\beta,s}^-)^{(k)}v=0$ for all $k>0$.
Notice that $\{-\alpha_1,\beta\}$ form a basis for $R$. Since, $(x_1^+)^{(k)}v=0$ for all $k>0$, Lemma \ref{l:weyl-relhr} applied to the subalgebra $U_\mathbb F(\lie g_{-\alpha_1,\beta}^{0,s})$ implies that $(x_{\vartheta, s}^{-})^{(k)}v=0$ for all $k>0$.

\noindent {\bf 2.3)} Suppose $b<s-1$ in which case $2a+3b-3s<0$. This time we apply Lemma \ref{l:weyl-relhr} to the subalgebra $U_\mathbb F(\lie g_{\alpha_1,\alpha_2}^{1,s-2})$. Indeed, we have $(x_{1,1}^{-})^{(k)}v=0$ for all $k>\max\{0,a-3\}=0$ and Lemma \ref{lem:long} implies that $(x_{2,s-2}^{-})^{(k)}v=0$ for all $k>\max\{0,b-(s-2)\}=0$. Thus, since$3(b-s)<-3$ and $a=1$, we have $\max\{0,2a+3b-3s\}=0$ and Lemma \ref{l:weyl-relhr} implies that $(x_{\vartheta, s}^{-})^{(k)}v=0$ for all $k>0$.

\noindent {\bf 3)} Assume $a=2$. We split in subcases as before.

\noindent {\bf 3.1)} If $b>s-1$ the proof is similar to that of step 2.1.

\noindent {\bf 3.2)} Suppose $b=s-1$ and notice that $2a+3b-3s=1$. Hence, we want to show that \eqref{e:g.naoi} holds for $k>1$. For $k>3$ we apply Lemma \ref{l:weyl-relhr} to the subalgebra $U_\mathbb F(\lie g_{\alpha_1,\alpha_2}^{1,s-2})$ in a similar fashion as we did in step 2.3 (the same can be conclude using the argument from step 2.2). For $k\in\{2,3\}$ we need our hypothesis on the characteristic of $\mathbb F$. Assume we have chosen the Chevalley basis so that $x_{\vartheta}^-= [x_1^+,x_\beta^-]$ and observe that \eqref{e:cbstructure} implies that $[x_1^+,x_{\vartheta}^-]=\pm 2 x_\alpha^-$. Using this, one easily checks that
\[(x_{\vartheta,s}^-)^{(2)}= (x_1^+)^{(2)}(x_{\beta,s}^-)^{(2)}-\frac 1 2x_1^+(x_{\beta,s}^-)^{(2)}x_1^+ -\frac 1 2 x_{\beta,s}^-x_{\vartheta,s}^-x_1^+ \mp x_{\beta,s}^- x_{\alpha,s}^-.\]
Using the case $\gamma=\alpha$ and Lemma \ref{lem:long} we see that $x_{\alpha,s}^-v=(x_{\beta,s}^-)^{(2)}v=0$. Hence, since $2\in\mathbb F^\times$, \eqref{e:g.naoi} holds for $k=2$. For $k=3$, we have $(x_{\vartheta,s}^-)^{(3)} = \frac 1 3 x_{\vartheta,s}^-(x_{\vartheta,s}^-)^{(2)}$ and, since $3\in\mathbb F^\times$, \eqref{e:g.naoi} also holds for $k=3$.

\noindent {\bf 3.3)} If $b<s-1$ the proof is similar to that of step 2.3.

\noindent {\bf 4)} Assume $a\ge 3$ and that \eqref{e:g.naoi} holds for $\lambda-3\omega_1$.

\noindent {\bf 4.1)} Suppose $s\ge 2$ and recall the definition of the map $\phi:U_\bb F(\lie a)\to U_\bb F(\lie g[t])$. The induction hypothesis together with Lemma \ref{l:naoi411} implies that
\begin{equation*}
(x_{\vartheta,s-2}^-)^{(k)}\in I_\bb F'(\lambda-3\omega_1) \quad\text{for all}\quad  k>\max\{0,2a+3b-3s\}
\end{equation*}
and, therefore
\begin{equation*}
(x_{\vartheta,s}^-)^{(k)} = \phi\left((x_{\vartheta,s-2}^-)^{(k)}\right) \in \phi(I_\bb F'(\lambda-3\omega_1)) \quad\text{for all}\quad  k>\max\{0,2a+3b-3s\}.
\end{equation*}
One easily checks that $\phi$ sends the generators of $I_\bb F'(\lambda-3\omega_1)$ to generators of $I_\bb F(\lambda)$, completing the proof of  \eqref{e:g.naoi} for $s\ge 2$.

\noindent {\bf 4.2)} For $s=0$, notice that $U_\mathbb F(\lie g)v$ is a quotient of $W_\mathbb F(\lambda)$, and \eqref{e:g.naoi} follows. Equivalently, apply Lemma \ref{l:weyl-relhr} to $U_\mathbb F(\lie g_{\alpha_1,\alpha_2}^{0,0})=U_\mathbb F(\lie g)$ and the proof is similar to that of step 2.1.

\noindent {\bf 4.3)} If $s=1$ and $b\ge 1$, we have $2a+3b-3s>0$ and the usual application of Lemma \ref{l:weyl-relhr} to $U_\mathbb F(\lie g_{\alpha_1,\alpha_2}^{0,1})$ completes the proof of \eqref{e:g.naoi}.
If $s=1$ and $b=0$, we need to show that $(x_{\vartheta,1}^-)^{(k)}v=0 $ for $k>2a-3$.

 Consider the subalgebra $U_\bb F(\lie{sl}_\vartheta[t])\cong U_\mathbb F(\lie{sl}_2[t])$ defined in Section \ref{ss:saofr12}. Since $\lambda(h_\vartheta)=2a$, it follows that $W:=U_\bb F(\lie{sl}_\vartheta[t])v$ is a quotient of $U_\mathbb F(\lie{sl}_2[t])$-module $W_\bb F^c(2a)$, where we identified the weight lattice of $\lie{sl}_2$ with $\mathbb Z$ as usual. Since $W_\bb F^c(2a)\cong D_\mathbb F(1,2a)$ by Theorem \ref{t:isos}\eqref{t:isosdg}, the defining relations of $D_\mathbb F(1,2a)$ imply $(x_{\vartheta,1}^-)^{(k)}v=0 $ for $k>2a-1$. It remains to check that $(x_{\vartheta,1}^-)^{(k)}v=0 $ for $k\in\{2a-2,2a-1\}$.

Suppose by contradiction that
$(x_{\vartheta,1}^-)^{(2a-1)}v\ne 0$ and notice that
\begin{equation}\label{e:low2a-1}
(x_{\vartheta}^-)^{(k)}(x_{\vartheta,1}^-)^{(2a-1)}v= 0 \quad\text{for all}\quad k>0.
\end{equation}
Indeed,
\begin{equation*}
(x_{\vartheta}^-)^{(k)}(x_{\vartheta,1}^-)^{(2a-1)}v\in W_\bb F^c(2a)_{-2a-2(k-1)}
\end{equation*}
is a vector of degree $2a-1>1$ for all $k\ge0$. By the Weyl group invariance of the character of $W_\bb F^c(2a)$, we know that $W_\bb F^c(2a)_{-2a-2(k-1)}=0$ if $k>1$, and that $W_\bb F^c(2a)_{-2a-2(k-1)}$ is one-dimensional concentrated in degree zero if $k=1$. This proves \eqref{e:low2a-1}. Then, Lemma \ref{l:lowtohigh} implies that
\begin{equation*}
(x_{\vartheta}^+)^{(2a-2)}(x_{\vartheta,1}^-)^{(2a-1)}v\ne 0.
\end{equation*}
On the other hand, it follows from Lemma \ref{l:garland} that
\begin{equation*}
(x_{\vartheta}^+)^{(2a-2)}(x_{\vartheta,1}^-)^{(2a-1)}v = x_{\vartheta, 2a-1}^-v.
\end{equation*}
Since $2a-1\ge 2$ and $2a-3(2a-1)=-4a+3<0$, it follows from step 4.1 that $x_{\vartheta, 2a-1}^-v=0$ yielding a contradiction as desired.

Similarly, assume  by contradiction that $(x_{\vartheta,1}^-)^{(2a-2)}v\ne 0$ and notice that
\begin{equation*}
(x_{\vartheta}^-)^{(k)}(x_{\vartheta,1}^-)^{(2a-2)}v= 0 \quad\text{for all}\quad k>1.
\end{equation*}
Suppose first that we also have $x_{\vartheta}^-(x_{\vartheta,1}^-)^{(2a-2)}v= 0$. It then follows from Lemma \ref{l:lowtohigh} that
\begin{equation*}
(x_{\vartheta}^+)^{(2a-4)}(x_{\vartheta,1}^-)^{(2a-2)}v\ne 0.
\end{equation*}
On the other hand, Lemma \ref{l:garland} implies that
\begin{equation*}
(x_{\vartheta}^+)^{(2a-4)}(x_{\vartheta,1}^-)^{(2a-2)}v = (x_{\vartheta, a-1}^-)^{(2)}v + \sum_{r=a}^{2a-2} x_{\vartheta,2a-2-r}^-x_{\vartheta, r}^- v.
\end{equation*}
Since $a-1\ge 2$, step 4.1 implies that $(x_{\vartheta, r}^-)^{(k)}v=0$ for all $r\ge a-1,k>0$, implying that the right-hand side is zero, which is a contradiction. It remains to check the possibility that $x_{\vartheta}^-(x_{\vartheta,1}^-)^{(2a-2)}v\ne 0$. In this case it follows that $x_{\vartheta}^-(x_{\vartheta,1}^-)^{(2a-2)}v$ is a lowest-weight vector for the algebra $U_\bb F(\lie{sl}_\vartheta)$ and, hence, Lemma \ref{l:lowtohigh} implies that
\begin{equation*}
(x_{\vartheta}^+)^{(2a-2)}x_{\vartheta}^-(x_{\vartheta,1}^-)^{(2a-2)}v\ne 0.
\end{equation*}
Using \eqref{e:koslem} we get
\begin{equation*}
(x_{\vartheta}^+)^{(2a-2)}x_{\vartheta}^-(x_{\vartheta,1}^-)^{(2a-2)}v = \left(x_{\vartheta}^-(x_{\vartheta}^+)^{(2a-2)} + (x_{\vartheta}^+)^{(2a-3)}\right)(x_{\vartheta,1}^-)^{(2a-2)}v.
\end{equation*}
Lemma \ref{l:garland} together with step 4.1 will again imply that the right-hand side is zero. This completes the proof.
\end{proof}

\subsection{Existence of Demazure flag} \label{ss:isofil} 

If $\lie g$ is simply laced, Theorem \ref{t:isos}\eqref{t:isofil} follows immediately from part \eqref{t:isosdg} with $k=1$. Thus, assume from now on that $\lie g$ is not simply laced and recall the notation introduced in Section \ref{ss:gsh}.

Given $\lambda\in P^+$, let $\mu=\overline\lambda\in P_\sh^+$ and $v$ be the image of $1$ in $W_\mathbb C^c(\lambda)$. Consider $W^\sh_\mathbb C:= U(\lie g_\sh[t])v$ and $W^\sh_\mathbb Z:=U_\mathbb Z(\lie g_\sh[t])v$.
By \cite[Lemma 4.17]{naoi:weyldem}, there is an isomorphism of $U(\lie g_\sh [t])$-modules $W^\sh_\mathbb C \cong D_\mathbb C(1,\mu)$.
By Corollary \ref{c:gweylfd}, $W^\sh_\mathbb Z$ is an integral form of $W^c_\mathbb C (\mu)\cong D_\mathbb C(1,\mu)$. Hence, we have an isomorphism of $U_\bb Z (\lie g_\sh [t])$-modules $W^\sh_\mathbb Z\cong D_\mathbb Z(1,\mu)$.

Since $\lie g_\sh$ is of type $A$, Theorem \ref{t:filtration} implies that there exist $k> 0, \mu_1,\dots,\mu_k\in P_\sh^+, m_1,\dots,m_k\in\mathbb Z_{\ge 0}$, and a filtration of $U_\mathbb Z(\lie g_\sh [t])$-modules $0=D_0\subseteq D_1\subseteq\cdots\subseteq D_k= W^\sh_\mathbb Z$, such that $D_j$ and $D_j/D_{j-1}$ are free $\mathbb Z$-modules, and $D_j/D_{j-1}\cong D_\mathbb Z(r^\vee,\mu_j,m_j)$ for all $j=1,\dots,k$. In particular,
\begin{equation}\label{e:Dk/Djfree}
W^\sh_\mathbb Z/D_j \quad\text{is a free} \quad\mathbb Z\text{-module for all}\quad j=0,\dots,k.
\end{equation}
Set $\lambda_j = \eta_\lambda(\mu_j) \in P^+$ where $\eta_{\lambda}$ is defined in \eqref{defn:eta.lambda}, $W_\mathbb Z^j = U_\mathbb Z(\lie g[t])D_j$ and $W_\mathbb F^j = \mathbb F\otimes_\mathbb Z W_\mathbb Z^j$. It is easy to see that we have $0=W_\mathbb F^0\subseteq W_\mathbb F^1\subseteq\cdots\subseteq W_\mathbb F^k$, and $\lambda_k=\lambda$ since $\mu_k = \mu$. Hence, we are left to show that
$$W_\mathbb F^j/W_\mathbb F^{j-1}\cong D_\mathbb F(1,\lambda_j,m_j)\quad\text{for all}\quad j=1,\dots,k, \quad\text{and}\quad W_\mathbb F^k\cong W^c_\mathbb F(\lambda).$$

Notice that $W_\mathbb Z^k=U_\mathbb Z(\lie g[t])v$. Then, Corollary \ref{c:gweylfd} implies that $W_\mathbb Z^k$ is an integral form of $W_\mathbb C^c(\lambda)$. Since $\bb Z$ is a PID and $W_\mathbb Z^k$ is a finitely generated, free \bb Z-module, it follows that $W^j_\mathbb Z$ is a free $\mathbb Z$-module of finite-rank for all $j=1,\dots,k$. Set $W_\mathbb C^j = U(\lie g[t])D_j$. It follows from \cite[Proposition 4.18]{naoi:weyldem} (which is Theorem \ref{t:isos}\eqref{t:isofil} in characteristic zero) that $W_\mathbb C^j/W_\mathbb C^{j-1}\cong D_\mathbb C(1,\lambda_j,m_j)$ for all $j=1,\dots,k$. Moreover, since $W^j_\mathbb C\cong\mathbb C\otimes_\mathbb Z W^j_\mathbb Z$, we have
\[
\mathbb C\otimes_\mathbb Z (W_\mathbb Z^j/W_\mathbb Z^{j-1}) \cong (W_\mathbb C^j/W_\mathbb C^{j-1}) \cong D_\mathbb C(1,\lambda_j,m_j).
\]
Therefore, $W_\mathbb Z^j/W_\mathbb Z^{j-1}$ is a finitely generated $\mathbb Z$-module of rank $\dim(D_\mathbb C(1,\lambda_j,m_j))$ for all $j=1,\dots,k$. Since
$W_\bb F^j/W_\bb F^{j-1}\cong \mathbb F\otimes_\mathbb Z (W_\mathbb Z^j/W_\mathbb Z^{j-1})$, it follows that
\[
\dim(W_\mathbb F^j/W_\mathbb F^{j-1})\ge \dim(D_\mathbb C(1,\lambda_j,m_j)) = \dim(D_\mathbb F(1,\lambda_j,m_j)).
\]
Now, let $v_j\in D_j$ be as in Theorem \ref{t:filtration}, $w$ be the image of $v$ in $W^k_\mathbb F$, $u_j\in U_\mathbb Z(\lie n^-_\sh[t])$ be such that $v_j=u_jv$, and $w_j=u_jw$. It follows that
$$W_\mathbb Z^j= \sum_{n\le j} U_\mathbb Z(\lie g[t])v_n \quad\text{and}\quad W_\mathbb F^j=\sum_{n\le j} U_\mathbb F(\lie g[t])w_n.$$
We will show that the image $\bar w_j$ of $w_j$ in $W_\mathbb F^j/W_\mathbb F^{j-1}$ satisfies the relations described in Proposition \ref{p:lessreldem}, which implies that $W_\mathbb F^j/W_\mathbb F^{j-1}$ is a quotient of $D_\mathbb F(1,\lambda_j,m_j)$ and, hence, $W_\mathbb F^j/W_\mathbb F^{j-1}\cong D_\mathbb F(1,\lambda_j,m_j)$ for all $j=1,\dots,k$.

By construction, $v_j$ is a weight vector of weight $\lambda_j$ and degree $m_j$, and so is $w_j$. Since $D_j/D_{j-1}\cong D_\mathbb Z(r^\vee,\mu_j,m_j)$, it follows that
\begin{gather*}
U_\mathbb F(\lie n^+_\sh[t])^0\bar w_j= U_\mathbb F(\lie h_\sh [t]_+)^0\bar w_j = 0 \quad\text{and}\quad (x_{\alpha,s}^-)^{(k)}\bar w_j = 0
\end{gather*}
for all $\alpha\in R^+_\sh, s\ge 0, k> \max \{0, \lambda(h_\alpha) - s r^\vee \}, j=1,\dots,k$. Thus, it remains to show that
\begin{gather*}
(x_{\alpha,s}^+)^{(m)}\bar w_j= \Lambda_{i,r}\bar w_j =  (x_\alpha^-)^{(k)}\bar w_j = 0
\end{gather*}
for all $\alpha\in R^+\setminus R_\sh^+, s\ge 0,r,m>0, k>\lambda_j(h_i), j=1,\dots,k$.
Since,
\begin{equation}\label{e:+nshrel}
\lambda_j+m\alpha\notin \lambda-Q^+ \qquad\text{for all}\qquad \alpha\in R^+\setminus R_\sh^+, m>0,
\end{equation}
we get $(x_{\alpha,s}^+)^{(m)}w_j=0$ for all $m>0, s\ge 0$. In particular, it follows that $\bar w_j$ is a highest-weight vector of weight $\lambda_j$ and, hence, $(x_\alpha^-)^{(k)}\bar w_j = 0$ for all $\alpha\in R^+, k>\lambda(h_\alpha)$. Finally, we show that
\begin{equation}\label{e:Lambdarelfil}
\Lambda_{i,r}\bar w_j=0 \quad\text{for all}\quad i\in I\setminus I_\sh, r>0, j=1,\dots,k.
\end{equation}
Observe that
\begin{equation*}
\Lambda_{i,r}u_j \in U_\mathbb Z(\lie n_\sh^-)U_\mathbb Z(\lie h[t]_+).
\end{equation*}
In particular, $\Lambda_{i,r} v_j\in W^\sh_\mathbb Z\cap W_\mathbb Z^j$. We will show that $\Lambda_{i,r} v_j\in D_{j-1}$ which implies \eqref{e:Lambdarelfil}.
Let $y_j\in U_\mathbb Z(\lie n_\sh^-)$ be such that $\Lambda_{i,r}u_j  = y_j$ modulo $U_\mathbb Z(\lie n_\sh^-)U_\mathbb Z(\lie h[t]_+)^{0}$. Thus, we want to show that
\begin{equation}\label{e:yjrel}
y_jv \in D_{j-1}.
\end{equation}
We prove this recursively on $j=1,\dots, k$. Notice that, since $\mathbb C\otimes_\mathbb Z(W^j_\mathbb Z/W^{j-1}_\mathbb Z)\cong D_\mathbb C(1,\lambda_j,m_j)$, there exists $n_j\in\mathbb Z_{>0}$ such that $n_jy_jv\in W_\mathbb Z^{j-1}, j=1,\dots,k$. In particular, since $W_\mathbb Z^{0}=0$ and $W_\mathbb Z^1$ is a torsion-free $\mathbb Z$-module, \eqref{e:yjrel} follows for $j=1$. Next, we show that
\eqref{e:yjrel} implies
\begin{equation}\label{e:WcapD}
W_\mathbb Z^j\cap W_\mathbb Z^{\sh} = D_j.
\end{equation}
Indeed, it follows from \eqref{e:+nshrel} and \eqref{e:yjrel} that
\begin{equation*}
W_\mathbb Z^j = U_\mathbb Z(\lie n^-[t])U_\mathbb Z(\lie g_\sh[t])v_j + W_\mathbb Z^{j-1}.
\end{equation*}
Since $U_\mathbb Z(\lie h_\sh[t]_+)^0U_\mathbb Z(\lie n_\sh^+)^0v_j\in D_{j-1}$ and, by induction hypothesis, $W_\mathbb Z^{j-1}\cap W_\mathbb Z^{\sh} = D_{j-1}$, \eqref{e:WcapD} follows by observing that
$$\left(U_\mathbb Z(\lie n^-[t])v_j\right)\cap W_\mathbb Z^{\sh} \subseteq D_j$$
(which is easily verified by weight considerations). Finally, observe that, since $n_{j+1}y_{j+1}v\in W_\mathbb Z^j\cap W_\mathbb Z^{\sh}=D_j$, \eqref{e:Dk/Djfree} implies that $y_{j+1}v\in D_j$. Thus, \eqref{e:WcapD} for $j$ implies \eqref{e:yjrel} for $j+1$ and the recursive step is proved.

\begin{rem}
It follows from the above that $W_\mathbb F^j/W_\mathbb F^{j-1}\cong D_\mathbb F(1,\lambda_j,m_j)$ for any field $\mathbb F$. Hence, $W_\mathbb Z^j/W_\mathbb Z^{j-1}$ must be isomorphic to $D_\bb Z (1, \lambda_j,m_j)$ for all $j = 1, \dots, k$.
\end{rem}

It remains to show that $W^k_\mathbb F \cong W_\mathbb F^c(\lambda)$. Since Theorem \ref{t:fdrhca}\eqref{t:gweylu} implies that we have a projection $W_\mathbb F^c(\lambda)\sobre W_\mathbb F^k$ of $U_\bb F (\lie g[t])$-modules, it suffices to show that $\dim(W_\mathbb F^c(\lambda))\le\dim(W_\mathbb F^k)$. This follows if we show that there exists a filtration $0=\tilde W_\mathbb F^0\subseteq \tilde W_\mathbb F^1\subseteq\cdots\subseteq \tilde W_\mathbb F^k=W_\mathbb F^c(\lambda)$ such that $\tilde W_\mathbb F^j/\tilde W_\mathbb F^{j-1}$ is a quotient of $D_\mathbb F(1,\lambda_j,m_j)$ for all $j=1,\dots,k$.
Let  $w'$ be the image of $1$ in $W_\mathbb F^c(\lambda)$, $w_j'=u_jw'\in W_\bb F^c (\lambda)$, $\tilde W^j_\mathbb F:= \sum_{n\le j}U_\mathbb F(\lie g[t])w_n' \subseteq W_\bb F^c (\lambda),$ and $\bar w_j'$ be the image of $w_j'$ in $\tilde W_\mathbb F^j/\tilde W_\mathbb F^{j-1}$. Observe that $\tilde W_\mathbb F^k = W_\mathbb F^c(\lambda)$. We need to show that $\bar w_j'$ satisfies the defining relations of $D_\mathbb F(1,\lambda_j)$ listed in Proposition \ref{p:lessreldem}.
Let $\tilde D_j = \mathbb F\otimes_\mathbb Z D_j$ and $D'_j = \sum_{n\le j}U_\mathbb F(\lie g_\sh[t])w'_n$. Notice that $D_k'$ is a quotient of $W^c_\mathbb F(\mu)\cong \tilde D_k$ and let $\pi:\tilde D_k\to D'_k$ be a $U_\mathbb F(\lie g_\sh[t])$-module epimorphism such that $v_k\mapsto w'_k$ (we keep denoting the image of $v_j$ in $\tilde D_j$ by $v_j$). In particular, $w'_j = \pi(v_j)$ and $\pi$ induces an epimorphism $\tilde D_j\to D'_j$ for all $j=1,\dots,k$. Hence,
\begin{equation*}
xw'_j\in D'_{j-1} \quad\text{for all}\quad x\in U_\mathbb Z(\lie g_\sh[t]) \quad\text{such that}\quad xv_j\in D_{j-1} .
\end{equation*}
This immediately implies that
\begin{gather*}
U_\mathbb F(\lie n^+_\sh[t])^0\bar w'_j= U_\mathbb F(\lie h [t]_+)^0\bar w'_j = 0 \quad\text{and}\quad (x_{\alpha,s}^-)^{(k)}\bar w'_j = 0
\end{gather*}
for all $\alpha\in R^+_\sh,  s\ge 0, k> \max \{0, \lambda(h_\alpha) - s r^\vee \}$, $j=1,\dots,k$. Note that \eqref{e:yjrel} has been used here.
The relations
$$(x_{\alpha,s}^+)^{(m)}\bar w'_j=  (x_i^-)^{(k)}\bar w'_j = 0$$
for all $\alpha\in R^+\setminus R_\sh^+,i\in I\setminus I_\sh, s\ge 0,m>0, k>\lambda_j(h_i), j=1,\dots,k$ follow from \eqref{e:+nshrel} as before.

\subsection{The isomorphism between local Weyl modules and graded local Weyl modules}{\label{pf13b}}

We now prove Theorem \ref{t:isos}\eqref{t:isoslg}. Recall the definition of the automorphism $\varphi_a$ of $U_\bb F(\lie g[t])$ from Section \ref{ss:autom}. In particular, let $\tilde a\in\bb A^\times$ be such that its image in $\bb F$ is $a$. Denote by $\varphi_a^*(W_\mathbb F(\gb\omega_{\lambda,a}))$ the pull-back of $W_\mathbb F(\gb\omega_{\lambda,a})$ (regarded as a $U_\bb F(\lie g[t])$-module) by $\varphi_a$.

Notice that
$$\dim W_\mathbb F(\gb\omega_{\lambda,a}) = \dim W_\mathbb K(\gb\omega_{\lambda,\tilde a})=  \dim W_\mathbb K^c(\lambda) = \dim W_\mathbb F^c(\lambda).$$
Here, the first equality follows from \eqref{e:conj}, the second from \eqref{e:chWeyl} (with $\bb F=\bb K$) together with Proposition \ref{p:dimwbyfund}, and the third from Corollary \ref{c:indf}. Since $\dim \varphi_a^*(W_\mathbb F(\gb\omega_{\lambda,a}))=\dim W_\mathbb F(\gb\omega_{\lambda,a})$,  Theorem \ref{t:isos}\eqref{t:isoslg} follows if we show that $\varphi_a^*(W_\mathbb F(\gb\omega_{\lambda,a}))$ is a quotient of $W_\mathbb F^c(\lambda)$.

Let $w\in W_\mathbb F(\gb\omega_{\lambda,a})_\lambda\setminus\{0\}$ and use the symbol $w_a$ to denote $w$ when regarded as an element of $\varphi_a^*(W_\bb F(\gb\omega_{\lambda,a}))$. Since $W_\mathbb F(\gb\omega_{\lambda,a})=U_\bb F(\lie g[t])w$ and $\varphi_a$ is an automorphism of $U_\bb F (\lie g [t])$, it follows that $\varphi_a^* W_\bb F (\gb\omega_{\lambda,a})=U_\bb F(\lie g[t])w_a$.  Thus, we need to show that $w_a$ satisfies the defining relations \eqref{e:glWrel} of $W_\bb F^c(\lambda)$. Since $\varphi_a$ fixes every element of $U_\bb F(\lie g)$, $w_a$ is a vector of weight $\lambda$ annihilated by $(x_{\alpha}^-)^{(k)}$ for all $\alpha\in R^+, k>\lambda(h_\alpha)$. Equation \eqref{e:t->t-a} implies that $\varphi_a$ maps $U_\bb F (\lie n^+[t])$ to itself and, hence, $U_\bb F (\lie n^+[t])^0w_a = 0$. Therefore, it remains to show that
$$U_\bb F(\lie h[t]_+)^0w_a=0.$$
For showing this, let $v\in W_\bb K(\gb\omega_{\lambda,\tilde a})_\lambda\setminus\{0\}$ and $L=U_\bb A(\lie g[t])v$. By \eqref{e:conj}, $\bb F\otimes_\bb A L \cong W_\bb F(\gb\omega_{\lambda,a})$. In particular, the action of $U_\bb F(\lie h[t]_+)^0$ on $\varphi_a^*(W_\mathbb F(\gb\omega_{\lambda,a}))$ is obtained from the action of $U_\bb A(\lie h[t]_+)^0$ on $\varphi_{\tilde a}^*(W_\mathbb K(\gb\omega_{\lambda,\tilde a}))$ which, in turn, is obtained from the action of $U_\bb K(\lie h[t]_+)^0$. Since $U_\bb K(\lie h[t]_+)$ is generated by $h_{i,r}, i\in I, r>0$, we are left to show that
$$h_{i,r}v_a=0,$$
where $v_a$ is the vector $v$ regarded as an element of $\varphi_{\tilde a}^*(W_\mathbb K(\gb\omega_{\lambda,\tilde a}))$. It is well known that the irreducible quotient of $W_\mathbb K(\gb\omega_{\lambda,\tilde a})$ is the evaluation module with evaluation parameter $\tilde a$ (cf. \cite[Section 3B]{JMhyper}). Hence, $h_{i,s}v=\tilde a^s\lambda(h_i)v$ for all $i\in I,s\in\bb Z$. Using this, it follows that, for all $i\in I, r>0$, we have
\begin{equation*}
h_{i,r}v_a = (h_i\otimes(t-\tilde a)^r)v =\sum_{s=0}^r \tbinom{r}{s}(-\tilde a)^s h_{i,r-s}v= \lambda(h_i)\tilde a^r\sum_{s=0}^r \tbinom{r}{s}(-1)^sv= 0.
\end{equation*}

\subsection{A tensor product theorem}\label{ss:tpprime}

We say that $\gb\omega,\gb\pi\in\mathcal P_\bb F^+$ are relatively prime if, for all $i,j\in I$, the polynomials $\gb\omega_i(u)$ and $\gb\pi_j(u)$ are relatively prime in $\bb F[u]$. The goal of this subsection is to prove the following theorem from which we will deduce Theorem \ref{t:isos}\eqref{t:isostp}.

\begin{thm}\label{t:tpwm}
Suppose $\gb \omega,\gb\pi\in\mathcal P_\mathbb F^+$ are relatively prime and that $V$ and $W$ are quotients of
$W_{\mathbb F} (\gb \omega)$ and $ W_{\mathbb F} (\gb \pi)$, respectively. Then, $V\otimes W$ is generated by its top weight space.
\end{thm}

Theorem \ref{t:tpwm} was proved in \cite{CPweyl} in the case $\mathbb F=\mathbb C$. Although the proof we present here follows the same general lines, there are several extra technical issues to be taken care of arising from the fact that $U_\bb C(\tlie g)$ is generated by $x_{\alpha,r}^\pm, \alpha\in R^+, r\in\bb Z$, while, in the case of $U_\bb F(\tlie g)$, we also need arbitrarily large divided powers of these elements. We start the proof establishing a few technical lemmas. Recall the definition of $X^-_{\alpha,m,s}(u)$ in Section \ref{ss:comut} and set
\begin{equation*}
X^-_{\alpha;s}(u)=X^-_{\alpha,1,s+1}(u)
\end{equation*}
To shorten notation, we shall often write $X^-_{\alpha;s}$ instead of $X^-_{\alpha;s}(u)$.

Fix $\gb\omega\in\mathcal P^+$ and let $w$ be a highest-$\ell$-weight vector of $W_\mathbb F(\gb\omega)$.
Given $\beta \in R^+$, define $\gb\omega_\beta(u)\in\mathbb F[u]$ by
\begin{equation*}
\gb\omega_\beta(u)w = \Lambda_\beta(u)w.
\end{equation*}
One can easily check (cf. \cite[Lemma 3.1]{CPweyl}) that, if $\vartheta$ is the highest short root of $\lie g$ and $\beta\in R^+$, then there exists $\gb\omega_{\vartheta,\beta}\in\mathcal P^+$ such that
\begin{equation*}
\gb\omega_{\vartheta}=\gb\omega_\beta\ \gb\omega_{\vartheta,\beta}.
\end{equation*}

\begin{lem}\label{kill}
For all $\beta\in R^+, k,l,s\in\mathbb Z, 0\le l\le k, k>\lambda(h_\beta)$, we have
$$\left(\gb\omega_{\vartheta}{X_{\beta;s}^-}^{(k-l)}\right)_{k+\deg(\gb\omega_{\vartheta,\beta})}w=0.$$
\end{lem}

\begin{proof}
We will need the following particular case of Lemma \ref{l:garland}:
\begin{equation}\label{e:garlandjm}
\left(x^+_{\beta,-s}\right)^{(l)}\left(x^-_{\beta,s+1}\right)^{(k)} = (-1)^l \left((X_{\beta;s}^-(u))^{(k-l)}\Lambda_{\beta}(u)\right)_k \mod U_\mathbb Z(\tlie g) U_\mathbb Z(\tlie n^+)^0
\end{equation}
for all $k,l,s\in\bb Z, 0\le l\le k$. It follows from \eqref{e:garlandjm} and the definition of $\gb\omega_\beta$ that
\begin{equation}\label{e:1}
    \left(\gb\omega_\beta{X_{\beta;s}^-}^{(k-l)} \right)_{k} w=0 \quad\text{for all}\quad k,l,s\in\bb Z,\ 0\le l\le k,\ k>\lambda(h_\beta).
\end{equation}
Hence, for such $k,l,s$, we have
\begin{align*}
\left(\gb\omega_{\vartheta}{X_{\beta;s}^-}^{(k-l)}\right)_{k+\deg(\gb\omega_{\vartheta,\beta})}w &= \left(\gb\omega_{\vartheta,\beta}\ \gb\omega_\beta\ {X_{\beta;s}^-}^{(k-l)}\right)_{k+\deg(\gb\omega_{\vartheta,\beta})}w \\
&= \sum_{j=0}^{\deg(\gb\omega_{\vartheta,\beta})} (\gb\omega_{\vartheta,\beta})_j\left(\gb\omega_\beta{X_{\beta;s}^-}^{(k-l)}\right)_{k+\deg(\gb\omega_{\vartheta,\beta})-j}w = 0,
\end{align*}
where the last equality follows from \eqref{e:1} since $k+\deg(\gb\omega_{\vartheta,\beta})-j >\lambda(h_\beta)$.
\end{proof}

Let $\cal R =R^+\times \mathbb Z\times\mathbb Z_{\ge0} $ and $\Xi$ be the set of functions $\xi:\mathbb N\to \cal R$ given by $j\mapsto \xi_j=(\beta_j,s_j,k_j)$, such that $k_j=0$ for all $j$ sufficiently large. Define the degree of $\xi$ to be $d(\xi) = \sum_j k_j$. Let $\Xi_{d}$ be the the subset of functions of degree $d$ and $\Xi^<_{d}=\bigcup\limits_{d'<d} \Xi_{d'}$.  Given $\xi\in\Xi$ such that $\xi_j=(\beta_j,s_j,k_j)$ for all $j\in\mathbb N$ and $k_j=0$ for $j>m$, set
\begin{equation}\label{vxi}
x^\xi =(x_{\beta_1,s_1}^-)^{(k_1)}\cdots (x_{\beta_m,s_m}^-)^{(k_m)} \qquad\text{and}\qquad w^\xi = x^\xi w.
\end{equation}
It will be convenient to write $\deg(w^\xi)=d(\xi)=\deg(x^\xi)$.
The next lemma is an easy consequence of \cite[Lemma 4.2.13]{mitz}.

\begin{lem} \label{l+k} Let $\alpha\in R^+,s\in\mathbb Z$, $d,k\in \mathbb Z_{\ge 0}$, and $\xi\in\Xi_d$. Then, $x^\xi(x_{\alpha,s}^-)^{(k)}$ is in the span of $$\{(x_{\alpha ,s}^-)^{(k)}x^\xi\}\cup\{x^\varsigma:\varsigma\in \Xi^<_{d+k}\}.$$\hfill\qedsymbol
\end{lem}

\begin{lem} \label{<r+d} Let $\beta\in R^+$, $k\in \mathbb Z$, $d,r,s \in \mathbb Z_{\ge0}$, $r\le s$, $s>\lambda(h_{\beta})$ and $\xi\in \Xi_d$. Then, $\left(\gb\omega_{\vartheta}{X_{\beta;k}^-}^{(r)}\right)_s w^{\xi}$ is in the span of vectors of the form $w^\varsigma$ with $\varsigma\in\Xi_{r+d}^<$.
\end{lem}

\begin{proof}
If $d=0$, it follows from \eqref{e:1} that $\left(\gb\omega_{\vartheta}{X_{\beta;k}^-}^{(r)}\right)_s w^{\xi}=0$, which proves the lemma in this case. We now proceed by induction on $d$. Thus, let $d>0$ and write $w^\xi=(x_{\beta_1,s_1}^-)^{(k_1)}\cdots (x_{\beta_l, s_l}^-)^{(k_l)}w$ with $k_1\ne 0$. Let also $\xi'\in\Xi$ be such that
$$\xi'_j=
\begin{cases}
\xi_j, &\text{ if } j\ne 1,\\
(\beta_1,s_1,0),  &\text{ if } j= 1.
\end{cases}
$$
Then, by Lemma \ref{l+k}, we have
\begin{align*}
\left(\gb\omega_{\vartheta}{X_{\beta;k}^-}^{(r)}\right)_s w^\xi &=\left(\gb\omega_{\vartheta}{X_{\beta;k}^-}^{(r)}\right)_s (x_{\beta_1, s_1}^-)^{(k_1)} w^{\xi'}\\
&=(x_{\beta_1,s_1}^-)^{(k_1)}\left(\gb\omega_{\vartheta}{X_{\beta;k}^-}^{(r)}\right)_s w^{\xi'}+ Xw^{\xi'}
\end{align*}
where $X$ is in the span of $\{x^\varsigma:\varsigma\in \Xi^<_{r+k_1}\}$.
In particular, $Xw^{\xi'}$ is in the span of vectors of the desired form.
Since $d(\xi')=d-k_1<d$, the induction hypothesis implies that  $\left(\gb\omega_{\vartheta}{X_{\beta;k}^-}^{(r)}\right)_s w^{\xi'}$ is in the span of vectors associated to elements of $\Xi^<_{r+d-k_1}$. Therefore, $(x_{\beta_1, s_1}^-)^{(k_1)}\left(\gb\omega_{\vartheta}{X_{\beta;k}^-}^{(r)}\right)_s w^{\xi'}$ is in the span of vectors associated to elements of $\Xi^<_{r+d}$ as desired.
\end{proof}

\begin{proof}[Proof of Theorem \ref{t:tpwm}]
Let $w_{\gb\omega}$ and $w_{\gb\pi}$ be highest-$\ell$-weight vectors for $V$ and $W$, respectively. Let also
\begin{equation*}
M=U_{{\mathbb F}}(\tlie g)(w_{\gb\omega}\otimes w_{\gb\pi})=U_{{\mathbb F}}(\tlie n^-)(w_{\gb\omega}\otimes w_{\gb\pi}).
\end{equation*}
Our goal is to show that $M=V\otimes W$. Since the vectors $w_\gb\omega^\xi\otimes w_{\gb\pi}^{\xi'}, \xi,\xi'\in\Xi$, span $V\otimes W$, it suffices to show that these vectors are in $M$. We do this by induction on $d(\xi)+d(\xi')$ which obviously starts when $d(\xi)+d(\xi')=0$ since, in this case, $w_\gb\omega^\xi\otimes w_{\gb\pi}^{\xi'}= w_\gb\omega\otimes w_{\gb\pi}$.

Let $n\ge0$, and suppose, by induction hypothesis, that
\begin{equation}\label{IH}
w_\gb\omega^\xi \otimes w_\gb\pi^{\xi'}\in M \quad\text{for all}\quad \xi,\xi'\in\Xi \quad\text{such that}\quad d(\xi)+d(\xi')\le n.
\end{equation}
In order to complete the induction step, it suffices to show that
\begin{equation}\label{e:tpwm}
w_\gb\omega^\xi \otimes ({x_{\beta,l}^-})^{(r)} w_\gb\pi^{\xi'} \in M \qquad\text{and}\qquad (({x_{\beta,l}^-})^{(r)}w_\gb\omega^\xi) \otimes w_\gb\pi^{\xi'} \in M
\end{equation}
for all $\beta\in R^+$, $r,l\in \mathbb Z,r\ge 1$, $\xi,\xi'\in\Xi$, such that $d(\xi)+d(\xi')+r=n+1$. We prove \eqref{e:tpwm} by a further induction on $r\ge1$. Henceforth we fix $\beta\in R^+$.

Observe that the hypothesis on $\gb\omega$ and $\gb\pi$ implies that $\gb\omega_{\vartheta}$ and $\gb\pi_{\vartheta}$ are relatively prime. Therefore,  we can choose $R,S \in \mathbb F[u]$ such that $$R\gb\omega_{\vartheta}+S\gb\pi_{\vartheta}=1.$$
Set
$$\delta=\deg(R\gb\omega_{\vartheta})=\deg(S\gb\pi_{\vartheta}) \quad\text{and}\quad m=\max\{\wt(\gb\omega)(h_\beta),\wt(\gb\pi)(h_\beta)\}.$$
We claim that, for all $\xi \in \Xi$, $k\in \mathbb Z$,
\begin{equation}\label{e:left}
(R\gb\omega_{\vartheta}{X_{\beta;k}^-}^{(r)})_sw_\gb\omega^\xi \quad\text{is in the span of vectors}\quad w_\gb\omega^\varsigma \quad\text{with}\quad \varsigma\in\Xi_{d(\xi)+r}^< \quad\text{for all}\quad s>m+\delta.
\end{equation}
Indeed,
\begin{align*}
(R\gb\omega_{\theta}{X_{\beta;k}^-}^{(r)})_sw_\gb\omega^\xi = \sum_{j=0}^{\deg R} R_j(\gb\omega_{\vartheta}{X_{\beta;k}^-}^{(r)})_{s-j}w_\gb\omega^\xi
\end{align*}
and, since $s-j>m+\delta-j\ge m+\deg(\gb\omega_{\vartheta})\ge \wt(\gb\omega)(h_\beta)$, the claim follows from Lemma \ref{<r+d}. Similarly one proves that
\begin{equation}\label{e:right}
(S\gb\pi_{\vartheta}{X_{\beta;k}^-}^{(r)})_sw_\gb\pi^\xi \quad\text{is in the span of vectors}\quad w_\gb\pi^\varsigma \quad\text{with}\quad \varsigma\in\Xi_{d(\xi)+r}^< \quad\text{for all}\quad s>m+\delta.
\end{equation}

We are ready to start the proof of \eqref{e:tpwm}. Suppose $d(\xi)+d(\xi')=n$ and let $\ell>m+\delta$. Then,
\begin{equation*}
\begin{array}{ll}
(R\gb\omega_{\vartheta} {X_{\beta;k}^-}^{})_{\ell}(w_\gb\omega^\xi \otimes w_\gb\pi^{\xi'})&= ((R\gb\omega_{\vartheta} {X_{\beta;k}^-}^{})_{\ell}w_\gb\omega^\xi) \otimes w_\gb\pi^{\xi'} +w_\gb\omega^\xi\otimes ((1-S\gb\pi_{\vartheta} ){X_{\beta;k}^-}^{})_{\ell}w_\gb\pi^{\xi'}\\
&= ((R \gb\omega_{\vartheta} {X_{\beta;k}^-})_{\ell}w_\gb\omega^\xi) \otimes w_\gb\pi^{\xi'} - w_\gb\omega^\xi \otimes (S\gb\pi_{\vartheta}{X_{\beta;k}^-})_{\ell}w_\gb\pi^{\xi'}+w_\gb\omega^\xi \otimes {x_{\beta;\ell+k}^-}w_\gb\pi^{\xi'}.
\end{array}
\end{equation*}
It follows from \eqref{e:left}, \eqref{e:right}, and \eqref{IH} that $((R \gb\omega_{\vartheta} {X_{\beta;k}^-})_{\ell}w_\gb\omega^\xi) \otimes w_\gb\pi^{\xi'}\in M$ and $w_\gb\omega^\xi \otimes (S\gb\pi_{\vartheta}{X_{\beta;k}^-})_{\ell}w_\gb\pi^{\xi'}\in M$. Since $(R\gb\omega_{\vartheta} {X_{\beta;k}^-}^{})_{\ell}(w_\gb\omega^\xi \otimes w_\gb\pi^{\xi'})\in M$ by definition, it follows that $w_\gb\omega^\xi \otimes {x_{\beta;\ell+k}^-}w_\gb\pi^{\xi'}\in M$ for all $k\in \mathbb Z$, which proves the first statement in \eqref{e:tpwm} with $r=1$. The second statement is proved similarly by looking at $(S\gb\pi_{\vartheta} {X_{\beta;k}^-}^{})_{\ell}(w_\gb\omega^\xi \otimes w_\gb\pi^{\xi'})$.

Let $r>1,\xi,\xi'\in\Xi$ be such that $r+ d(\xi)+d(\xi')=n+1$ and set $\ell=r\ell'$ with $\ell'$ such that $\ell>m+\delta$. Then,
\begin{equation*}
        \begin{array}{ll}
        (R \gb\omega_{\vartheta} {X_{\beta;k}^-}^{(r)})_{\ell}(w_\gb\omega^{\xi} \otimes w_\gb\pi^{\xi'})& =
        ((R \gb\omega_{\vartheta} {X_{\beta;k}^-}^{(r)})_{\ell}w_\gb\omega^\xi) \otimes w_\gb\pi^{\xi'}
        +w_\gb\pi^{\xi'} \otimes (R \gb\omega_{\vartheta} {X_{\beta;k}^-}^{(r)})_{\ell}w_\gb\pi^{\xi'} + v\\
        &=  ((R \gb\omega_{\vartheta} {X_{\beta;k}^-}^{(r)})_{\ell}w_\gb\omega^\xi) \otimes w_\gb\pi^{\xi'}
        +w_\gb\omega^\xi \otimes ((1-S\gb\pi_{\vartheta}){X_{\beta;k}^-}^{(r)})_{\ell}w_\gb\pi^{\xi'}+v\\
        &=  ((R \gb\omega_{\vartheta} {X_{\beta;k}^-}^{(r)})_{\ell}w_\gb\omega^\xi) \otimes w_\gb\pi^{\xi'}
        -w_\gb\omega^\xi \otimes (S\gb\pi_{\vartheta}{X_{\beta;k}^-}^{(r)})_{\ell}w_\gb\pi^{\xi'} +w_\gb\omega^\xi \otimes ({X_{\beta;k}^-} ^{(r)})_{\ell}w_\gb\pi^{\xi'}+v\\
        &=  ((R \gb\omega_{\vartheta} {X_{\beta;k}^-}^{(r)})_{\ell}w_\gb\omega^\xi) \otimes w_\gb\pi^{\xi'}
        -w_\gb\omega^\xi \otimes (S\gb\pi_{\vartheta}{X_{\beta;k}^-}^{(r)})_{\ell}w_\gb\pi^{\xi'}\\
        &\quad +\ w_\gb\omega^\xi \otimes (x_{\beta,\ell'+k}^-)^{(r)}w_\gb\pi^{\xi'}+w_\gb\omega^\xi\otimes X w_\gb\pi^{\xi'}+v,
        \end{array}
\end{equation*}
where $v$ is in the span of vectors of the form
$$\left(\prod_i(x^-_{\beta,s_i})^{(a_i)}w_\gb\omega^\xi\right) \otimes \left(\prod_j (x^-_{\beta,s_j})^{(b_j)}w_\gb\pi^{\xi'}\right) \quad\text{with}\quad 1\le a_i,b_j<r, \quad \sum_i a_i+ \sum_j b_j =r,$$
and $X$ is in the span of elements of the form
$$(x_{\beta,s_1}^-)^{(r_1)}(x_{\beta,s_2}^-)^{(r_2)}\cdots (x_{\beta,s_n}^-)^{(r_n)} \quad\text{with}\quad r_1+\cdots+r_n=r,\quad 0< r_j < r.$$
Again, $(R \gb\omega_{\vartheta} {X_{\beta;k}^-}^{(r)})_{\ell}(w_\gb\omega^\xi \otimes w_\gb\pi^{\xi'}) \in M$ by definition, while \eqref{e:left}, \eqref{e:right}, and \eqref{IH}, imply that
$$((R \gb\omega_{\vartheta} {X_{\beta;k}^-}^{(r)})_\ell w_\gb\omega^\xi) \otimes w_\gb\pi^{\xi'}\in M\quad \text{and} \quad w_\gb\omega^\xi \otimes (S\gb\pi_{\vartheta}{X_{\beta;k}^-})_\ell w_\gb\pi^{\xi'}\in M.$$ By induction hypothesis on $r$, it follows that $v\in M$ and $w_\gb\omega^\xi\otimes X w_\gb\pi^{\xi'}\in M$, which then implies that $w_\gb\omega^\xi \otimes (x_{\beta,\ell'+k}^-)^{(r)}w_\gb\pi^{\xi'}\in M$ for all $k\in \mathbb Z$, completing the proof of the first statement of \eqref{e:tpwm}. The second statement is proved similarly by looking at $(S\gb\pi_{\vartheta} {X_{\beta;k}^-}^{(r)})_{\ell}(w_\gb\omega^{\xi} \otimes w_\gb\pi^{\xi'})$.
\end{proof}

\subsection{The tensor product factorization of local Weyl modules}\label{ss:ptp}

Theorem \ref{t:isos}\eqref{t:isostp} clearly follows if we prove:
\begin{equation}\label{tensoriso}
    W_\mathbb F (\gb \varpi_1) \otimes W_\mathbb F (\gb\varpi_2) \cong W_\mathbb F (\gb \varpi_1\gb\varpi_2)
\end{equation}
whenever $\gb \varpi_1,\gb\varpi_2 \in \mathcal P_\mathbb F^+$ are relatively prime.

In order to show \eqref{tensoriso}, let $w_{\gb\varpi_1}$ and $w_{\gb\varpi_2}$ be highest-$\ell$-weight vectors for $W_\mathbb F (\gb \varpi_1)$ and $W_\mathbb F (\gb\varpi_2)$, respectively. It is well-known that $w_{\gb\varpi_1}\otimes v_{\gb\varpi_2}$ satisfies the defining relations of $W_\mathbb F(\gb\varpi_1\gb\varpi_2)$, so there exists a $U_\mathbb F(\tlie g)$-module map $\phi: W_\mathbb F(\gb\varpi_1\gb\varpi_2) \to W_\mathbb F(\gb\varpi_1)\otimes W_\mathbb F(\gb\varpi_2)$ that sends $w_{\gb\varpi_1\gb\varpi_2}$ to $w_{\gb\varpi_1}\otimes w_{\gb\varpi_2}$.
Theorem \ref{t:tpwm} implies that $\phi$ is surjective. Hence, it suffices to show that
\begin{equation}\label{e:dimtp}
\dim(W_\mathbb F(\gb\varpi_1\gb\varpi_2)) = \dim (W_\mathbb F(\gb\varpi_1)\otimes W_\mathbb F(\gb\varpi_2)).
\end{equation}
In fact, recall from Remark \ref{r:Aforms} that there exist $\gb\omega_1,\gb\omega_2\in\mathcal P_\mathbb A^\times$ such that $\gb\varpi_1$ and $\gb\varpi_2$ are the images of $\gb\omega_1$ and $\gb\omega_2$ in $\cal P_\mathbb F^+$, respectively. It then follows from \eqref{e:conj} that
\begin{equation}\label{e:dimtpp0}
\dim(W_\mathbb F(\gb\varpi_1\gb\varpi_2))=\dim(W_\mathbb K(\gb\omega_1\gb\omega_2))\quad \text{ and } \quad \dim( W_\mathbb K(\gb\omega_i)) = \dim( W_\mathbb F(\gb\varpi_i)), \ i=1,2.
\end{equation}
On the other hand, it follows from Theorem \ref{t:isos}\eqref{t:isostp} in characteristic zero that
\begin{equation}\label{e:dimtp0}
\dim(W_\mathbb K(\gb\omega_1\gb\omega_2)) = \dim( W_\mathbb K(\gb\omega_1)) \dim( W_\mathbb K(\gb\omega_2)).
\end{equation}
Since \eqref{e:dimtpp0} and \eqref{e:dimtp0} clearly imply \eqref{e:dimtp}, we are done.

\subsection{Fusion products}

We finish the paper with an application of Theorems \ref{t:isos} and \ref{t:tpwm} related to the concept of fusion products originally introduced in in the characteristic zero setting. Namely,  we deduce the positive characteristic counterpart of \cite[Corollary B]{naoi:weyldem} (cf. \cite[Corollary A]{foli:weyldem} for simply laced $\lie g$).

Let $V$ and $W$ be as in Theorem \ref{t:tpwm}, set $\lambda=\wt(\gb\omega)+\wt(\gb\pi)$, and fix $v\in (V\otimes W)_\lambda\setminus\{0\}$. Then, Theorem \ref{t:tpwm} implies that $V\otimes W=U_\bb F(\tlie g)v$. In fact, as mentioned in Section \ref{ss:JM.conjecture}, we actually have
\begin{equation*}
V\otimes W=U_\bb F(\lie n^-[t])v.
\end{equation*}
Define the fusion product of $V$ and $W$, denoted $V*W$, as the $U_\bb F(\lie g[t])$-module $\gr(V\otimes W)$ with the module structure determined by $v$ as described in the paragraph after Proposition \ref{p:dimwbyfund}. Evidently, if we have a collection $\gb\omega_1,\dots,\gb\omega_m$ of relatively prime elements of $\cal P_\bb F^+$ and, for each $j\in\{1,\dots,m\}$, $V_j$ is a quotient of $W_\bb F(\gb\omega_j)$, we can define the fusion product $V_1*\cdots *V_m$ in a similar way.

\begin{prop}\label{p:fusion}
Let $\lambda\in P^+$, $m\in\bb Z_{> 0}$, and $\gb\omega_j\in\cal P_\bb F^+, j=1,\dots,m$, be relatively prime and such that $\lambda=\sum_{j=1}^m\wt(\gb\omega_j)$. Then,
$$W^c_\bb F(\lambda)\cong W_\bb F(\gb\omega_1)*\cdots *W_\bb F(\gb\omega_m).$$
\end{prop}

\begin{proof}
One easily checks that a vector in $(W_\bb F(\gb\omega_1)*\cdots *W_\bb F(\gb\omega_m))_\lambda$ satisfies the defining relations of $W^c_\bb F(\lambda)$ (cf. the proof of \eqref{e:conj} in Section \ref{ss:JM.conjecture}), showing that $W_\bb F(\gb\omega_1)*\cdots *W_\bb F(\gb\omega_m)$ is a quotient of $W^c_\bb F(\lambda)$. On the other hand, setting $\gb\omega=\prod_{j=1}^m\gb\omega_j$, we have
\begin{equation*}
\dim( W_\bb F(\gb\omega_1)*\cdots *W_\bb F(\gb\omega_m)) =  \dim(W_\bb F(\gb\omega_1)\otimes \cdots \otimes W_\bb F(\gb\omega_m)) = \dim(W_\bb F(\gb\omega)) = \dim(W^c_\bb F(\lambda)).
\end{equation*}
\end{proof}

The following corollary, which is the characteristic-free version of \cite[Corollary B]{naoi:weyldem}, is now easily deduced.

\begin{cor}
Let $m\in\bb Z_{>0},\lambda_j\in P^+, a_j\in\bb F^\times,  j=1,\dots,m$ be such that $a_i\ne a_j$ for $i\ne j$. Then, for $\lambda=\sum_{j=1}^m \lambda_j$, we have $W^c_\bb F(\lambda)\cong W_\bb F(\gb\omega_{\lambda_1,a_1})*\cdots *W_\bb F(\gb\omega_{\lambda_m,a_m})$.\hfill\qedsymbol
\end{cor}

\providecommand{\bysame}{\leavevmode\hbox to3em{\hrulefill}\thinspace}


\begin{thebibliography}{10}

\bibitem{bimo:htwisted}
A.~Bianchi and A.~Moura, {\em Finite-dimensional representations of twisted hyper loop algebras}, Communications in Algebra {\bf 42} (2014), 3147--3182,
DOI 10.1080/00927872.2013.781610.

\bibitem{car:sLieG}
R. Carter, Simple groups of {L}ie type, John Wiley \& Sons (1972).

\bibitem{cha:ibma}
S.~Chamberlin, {\em Integral bases for the universal enveloping algebras of map algebras},
J. Algebra \textbf{377} (2013), 232--249.

\bibitem{cfk}
V.~Chari, G.~Fourier, and T.~Khandai, {\em A categorical approach to Weyl modules}, Transf. Groups {\bf 15} (2010), 517--549.

\bibitem{chlo:wfd}
V.~Chari and S.~Loktev, {\em Weyl, {D}emazure and fusion modules for the current algebra of {$\mathfrak{sl}_{r+1}$}}, Adv. Math. \textbf{207} (2006), no.~2, 928--960.

\bibitem{cpnew}
V.~Chari and A.~Pressley, {\em New unitary representations of loop groups}, Math. Ann. \textbf{275} (1986), no.~1, 87--104.

\bibitem{CPweyl}
\bysame, {\em Weyl modules for classical and quantum affine algebras}, Represent. Theory \textbf{5} (2001), 191--223.

\bibitem{chaven}
V. Chari and R. Venkatesh, {\em Demazure modules, Fusion products and Q-systems}, arXiv:1305.2523.

\bibitem{FL04}
B. Feigin and S. Loktev, {\em Multi-dimensional {W}eyl modules and symmetric functions}, Comm. Math. Phys. \textbf{251} (2004), no.~3, 427--445.

\bibitem{fkks}
G. Fourier, T. Khandai, D. Kus, and A. Savage, {\em Local Weyl modules for equivariant map algebras with free abelian group actions}, J. Algebra {\bf 350} (2012), 386--404, DOI:	10.1016/j.jalgebra.2011.10.018.

\bibitem{foku}
G. Fourier and D. Kus, {\em Demazure modules and Weyl modules: The twisted current case}, arXiv:1108.5960.

\bibitem{foli:weyldem}
G.~Fourier and P.~Littelmann, {\em Weyl modules, {D}emazure modules, {KR}-modules, crystals, fusion products and limit constructions}, Adv. Math.
  \textbf{211} (2007), no.~2, 566--593.

\bibitem{fms}
G. Fourier, N. Manning, and A. Savage, {\em Global Weyl modules for equivariant map algebras}, arXiv:1303.4437.

\bibitem{garala}
H.~Garland, {\em The arithmetic theory of loop algebras}, J. Algebra \textbf{53} (1978), no.~2, 480--551.

\bibitem{hum:lie}
J.~E.~Humphreys, Introduction to {L}ie algebras and representation theory, Springer-Verlag (1978).

\bibitem{humphreys90}
\bysame, {Reflection groups and {C}oxeter groups}, Cambridge University Press (1990).

\bibitem{jantzen03}
J.~C.~Jantzen, {Representations of algebraic groups}, American Mathematical Society (2003).

\bibitem{JMhyper}
D.~ Jakeli{\'c} and A.~Moura, {\em Finite-dimensional representations of hyper loop algebras}, Pacific J. Math. \textbf{233} (2007), no.~2, 371--402.

\bibitem{jm:hlanac}
\bysame, {\em Finite-dimensional representations of hyper loop algebras over non-algebraically closed fields}, Algebr. Represent. Theory \textbf{13} (2010), no.~3, 271--301.

\bibitem{jm:weyl}
\bysame, {\em On Weyl modules for quantum and hyper loop algebras}, Contemp. Math. {\bf 623} (2014), 99--134.

\bibitem{joseph03}
A.~Joseph, {\em A decomposition theorem for {D}emazure crystals}, J. Algebra \textbf{265} (2003), no.~2, 562--578.

\bibitem{joseph06}
\bysame, {\em Modules with a {D}emazure flag}, In Studies in {L}ie theory, Progr. Math., vol. 243, Birkh\"auser (2006), 131--169.

\bibitem{lit:contract}
P.~Littelmann, {\em Contracting modules and standard monomial theory for symmetrizable {K}ac-{M}oody algebras}, J. Amer. Math. Soc. \textbf{11}
  (1998), no.~3, 551--567.

\bibitem{macedo:PhD}
T.~Macedo, {Characters and cohomology of modules for affine Kac-Moody algebras and generalizations}, Ph.D. Thesis, Unicamp (2013).

\bibitem{mathieu88}
O.~Mathieu, {\em Formules de caract\`eres pour les alg\`ebres de {K}ac-{M}oody g\'en\'erales}, Ast\'erisque volumes 159--160 (1988).

\bibitem{mathieu89}
\bysame, {\em Construction d'un groupe de {K}ac-{M}oody et applications}, Compositio Math. \textbf{69} (1989), no.~1, 37--60.

\bibitem{mitz}
D.~Mitzman, {\em Integral bases for affine {L}ie algebras and their universal enveloping algebras}, Contemporary Mathematics, vol.~40, American
  Mathematical Society (1985).

\bibitem{naoi:weyldem}
K.~Naoi, {\em Weyl modules, {D}emazure modules and finite crystals for non-simply laced type}, Adv. Math. \textbf{229} (2012), no.~2, 875--934.

\end{thebibliography}
\end{document}